\documentclass[11pt]{amsart}
\usepackage{amssymb}
\usepackage{amscd}
\usepackage{amsmath}
\usepackage[all]{xy}
\usepackage{mathrsfs}
 \usepackage{color}

\theoremstyle{plain}
\newtheorem{theorem}{Theorem}[section]
\newtheorem{proposition}[theorem]{Proposition}
\newtheorem{lemma}[theorem]{Lemma}

\newtheorem{conjecture}[theorem]{Conjecture}

\theoremstyle{remark}
\newtheorem{remark}[theorem]{Remark}

\theoremstyle{definition}
\newtheorem{definition}[theorem]{Definition}
\newtheorem{example}[theorem]{Example}
\newtheorem{hypothesis}[theorem]{Hypothesis}
\newtheorem*{acknowledgements}{Acknowledgements}

\font\russ=wncyr10  1
\def\sha{\hbox{\russ\char88}}

\DeclareMathOperator{\Gal}{Gal}

\DeclareMathOperator{\Hom}{Hom}

\DeclareMathOperator{\N}{N}

\DeclareMathOperator{\im}{im}

\newcommand{\bQ}{\mathbb{Q}}

\newcommand{\bZ}{\mathbb{Z}}
\newcommand{\QQ}{\mathbb{Q}}
\newcommand{\cA}{\mathcal{A}}
\newcommand{\cB}{\mathcal{B}}

\newcommand{\cD}{\mathcal{D}}

\newcommand{\cF}{\mathcal{F}}
\newcommand{\cG}{\mathcal{G}}
\newcommand{\cH}{\mathcal{H}}
\newcommand{\cI}{\mathcal{I}}

\newcommand{\cK}{\mathcal{K}}

\newcommand{\cM}{\mathcal{M}}
\newcommand{\cN}{\mathcal{N}}
\newcommand{\cO}{\mathcal{O}}
\newcommand{\cP}{\mathcal{P}}

\newcommand{\cR}{\mathcal{R}}
\newcommand{\cS}{\mathcal{S}}
\newcommand{\cT}{\mathcal{T}}
\newcommand{\cW}{\mathcal{W}}

\newcommand{\fq}{\mathfrak{q}}
\newcommand{\fp}{\mathfrak{p}}

\newcommand{\fn}{\mathfrak{n}}
\newcommand{\fm}{\mathfrak{m}}

\newcommand{\CC}{\mathbb{C}}

\newcommand{\FF}{\mathbb{F}}
\newcommand{\GG}{\mathbb{G}}

\newcommand{\Q}{\mathbb{Q}}
\newcommand{\RR}{\mathbb{R}}

\newcommand{\ZZ}{\mathbb{Z}}
\newcommand{\Z}{\mathbb{Z}}

\newcommand{\bz}{\mathbb{Z}}

\begin{document}

\title[]{On the theory of higher rank \\
Euler, Kolyvagin and Stark systems, III:\\
applications}

\author{David Burns, Ryotaro Sakamoto and Takamichi Sano}

\begin{abstract} In an earlier article we proved the existence of a canonical Kolyvagin derivative homomorphism between the modules of Euler and Kolyvagin systems (in any given rank) that are associated to $p$-adic representations over number fields. We now explain how the existence of such a homomorphism leads to new results on the structure of the Selmer modules of Galois representations over Gorenstein orders and to a strategy for verifying (refinements of) the Tamagawa number conjecture of Bloch and Kato. We describe concrete applications relating to the multiplicative group over arbitrary number fields and to elliptic curves over abelian extensions of the rational numbers.
\end{abstract}

\address{King's College London,
Department of Mathematics,
London WC2R 2LS,
U.K.}
\email{david.burns@kcl.ac.uk}

\address{Graduate School of Mathematical Sciences, The University of Tokyo,
3-8-1 Komaba, Meguro-Ku, Tokyo, 153-8914, Japan}
\email{sakamoto@ms.u-tokyo.ac.jp}

\address{Osaka City University,
Department of Mathematics,
3-3-138 Sugimoto\\Sumiyoshi-ku\\Osaka\\558-8585,
Japan}
\email{sano@sci.osaka-cu.ac.jp}

\thanks{Version of February 2019}

\maketitle

\tableofcontents

\section{Introduction}

\subsection{Discussion of the main results}In an earlier article \cite{bss} we proved the existence, under a variety of technical hypotheses on a $p$-adic Galois representation $T$, of a canonical `higher Kolyvagin derivative' homomorphism between the modules of Euler and Kolyvagin systems that are associated to $T$ in any given rank, as had been conjectured to exist by Mazur and Rubin.

This homomorphism was constructed in the setting of representations $T$ that are free with respect to the action of an arbitrary Gorenstein $\Z_p$-order $\mathcal{R}$ and, in particular, implied that, under the stated hypotheses, higher rank Euler, Kolyvagin and Stark systems could control the $\mathcal{R}$-module structures of Selmer modules associated to the representation.

This article is a natural sequel to \cite{bss} and has three main purposes.

Firstly, we will prove that, for both the canonical and unramified Selmer structures, all of the technical hypotheses that are necessary to develop the general theory of \cite{bss} essentially follow from the same `standard' hypotheses on Galois representations that are introduced by Mazur and Rubin in \cite{MRkoly} in order to develop the classical theory of (rank one) Kolyvagin systems. This shows that the theory developed in \cite{bss} can be used in natural arithmetic settings to study Selmer modules that are endowed with a natural action of a Gorenstein order.

Secondly, with potential applications to leading term conjectures in mind, we explain how the constructions and results in \cite{bss} can be combined with an analysis of the determinant modules of Galois cohomology complexes to develop a strategy for reducing the proof of natural equivariant refinements of the Tamagawa number conjecture of Bloch and Kato to the proof of the original (non-equivariant) conjecture, as formulated in \cite{bk} and later re-worked and extended by Fontaine and Perrin-Riou in \cite{fpr}.

Finally, we shall use these results to give the first arithmetic applications of the theory of higher rank Euler, Kolyvagin and Stark systems.

To give a little more detail, we note first that, to help the reader, we have listed the basic hypotheses under which we can develop the general theory of \cite{bss} in \S\ref{stand hyp sec}.

The proof that these explicit hypotheses are sufficient (and, in some cases, also necessary) to guarantee that all of the hypotheses stated in loc. cit. are satisfied requires a detailed analysis of Galois cohomology groups and is given in \S\ref{section euler}.

This analysis culminates in the statement of Theorem \ref{main} but, in the course of the argument, we will also establish several supplementary results that are of independent interest.

For example, we will describe an explicit, and usable, criterion for the higher Kolyvagin derivative homomorphism constructed in \cite{bss} to be surjective and hence for the structure of Selmer modules over Gorenstein orders to be controlled by a higher rank Kolyvagin system that arises as the derivative of a higher rank Euler system (see Theorem \ref{main}(v)).

This result essentially resolves in arbitrary rank the problem that is explicitly raised  by Mazur and Rubin in \cite[Question 5.3.22]{MRkoly} in the context of rank one Kolyvagin systems and which they comment `seems to be very difficult'.

In the course of proving Theorem \ref{main} we shall also give an explicit characterization of those representations over Gorenstein orders for which there exists a suitable analogue of the `core vertices' that plays a key role in the (non-equivariant) theory of Kolyvagin systems that is developed by Mazur and Rubin in \cite{MRkoly}.

Such a result is essential for the development of any `equivariant' theory of Kolyvagin systems and is stated as Theorem \ref{free} (and see also Remarks \ref{core comp} and \ref{cartesian}).


Next we turn to discuss how this theory can be applied in the study of leading term conjectures and, in particular, how it can be used to reduce the proof of `equivariant' Tamagawa number conjectures to the proof of Tamagawa number conjectures over discrete valuation rings.

At the outset we recall that such reductions have in fact already occurred, implicitly at least, in numerous previous articles starting, we believe, with the article \cite{BG} of Greither and the first author.

However, when considering Tamagawa number conjectures with respect to coefficient rings that are not regular, such arguments have hitherto always relied on difficult Iwasawa-theoretic assumptions such as the validity of
 main conjectures, the vanishing of $\mu$-invariants and the validity of appropriate versions of the `refined class number formula' conjectured independently by Mazur and Rubin \cite{MRGm} and the third author \cite{sano} and of the `order of vanishing' conjecture for $p$-adic $L$-series due to Gross \cite{Gp}.

We would like to stress that the key feature of the approach we describe here is that it is completely independent of all such hypotheses.

Nevertheless, if one wishes to apply the theory developed in \cite{bss} in a given arithmetic setting, then one of course needs an appropriate supply of Euler, Kolyvagin or Stark systems of the appropriate rank.

As a first step in this direction, we will explain in \S\ref{bloch kato} how the general formalism of \cite{bk} and \cite{fpr} naturally leads to the conjectural construction of a family of higher rank Euler systems (in the sense of \cite{bss}) that are explicitly related to the leading terms of motivic $L$-series.

We show that this family of `Bloch-Kato Euler systems' provides a simultaneous (conjectural) generalization of several well-known families of Euler systems, including both the higher rank Euler system of `Rubin-Stark' elements that is constructed by Rubin in \cite{rubinstark} and the Euler system of `zeta elements' that is constructed by Kato in \cite{kato}.

In Theorems \ref{cor1} and \ref{hes2} we will then apply the general theory of \cite{bss} to show that, under certain mild hypotheses, the existence of the Bloch-Kato Euler system allows one both to prove that natural Iwasawa-theoretic Selmer modules are torsion and also to reduce the proof of equivariant Tamagawa number conjectures with respect to coefficient rings that are not regular to the proof of a collection of Tamagawa number conjectures over discrete valuation rings.

To help explain the interest of the latter result, we recall that whilst the equivariant Tamagawa number conjecture is, even in important arithmetic settings, still rather poorly understood and supported by little concrete evidence, in the same cases the corresponding Tamagawa number conjectures over discrete valuation rings can often coincide with classical theorems or conjectures and so are either known to be true or at least to be strongly supported by theory.

In the remainder of the article we shall then focus on two important special cases in order to describe the first arithmetic applications of this general theory.

Firstly, in \S\ref{section 4} we discuss the example that was originally considered by Rubin in \cite{rubinstark} and then subsequently by B\"uy\"ukboduk in \cite{Buyuk}.

In particular, the main result (Theorem \ref{main RS}) of this section shows that Rubin-Stark Euler systems control detailed aspects of the fine Galois structure of ideal class groups over abelian extensions $L/K$ of arbitrary number fields and also give rise to strong new evidence in support of the equivariant Tamagawa number conjecture for the untwisted Tate motive over $L/K$.

This result strongly refines all previous results in this direction including, in particular, the main results of B\"uy\"ukboduk in \cite{Buyuk} as well as of the earlier articles \cite{rubincrelle} and \cite{rubinstark} of Rubin.

For example, whilst the main result of \cite{Buyuk} deals only with initial  Fitting ideals and assumes, amongst other things, that $K$ is totally real and, crucially, that both $L/K$ has degree prime to $p$ and Leopoldt's conjecture is valid, our approach allows us to determine all Fitting ideals of the ideal class groups whilst simultaneously avoiding any hypotheses concerning either $K$, the degree of $L/K$ or the validity of Leopoldt's Conjecture.

Then, finally,  in \S\ref{section ell curves} we shall consider the case of elliptic curves $E$ over an abelian extension $F$ of $\QQ$ with $G := \Gal(F/\QQ)$.

In this case we shall first investigate the precise link between Kato's zeta elements and the relevant `Bloch-Kato elements' that we define in \S\ref{section 3}, and are thereby led to formulate a precise generalization and refinement of the well-known conjectures that are formulated by Perrin-Riou in \cite{PR} (see Conjecture \ref{kato conj} and Propositions \ref{kato prop} and \ref{kato cor}).

We shall then apply our general approach to prove, under certain very mild hypotheses, that if $E$ validates both Perrin-Riou's Conjecture and the `$p$-part' of the Birch and Swinnerton-Dyer Conjecture over $\QQ$, then Kato's Euler system completely determines all higher Fitting ideals over $\ZZ_p[G]$ of the strict $p$-Selmer group of $E$ over $F$ (see Theorem \ref{cor kato}) and also that the equivariant Tamagawa number is valid for the motive $h^1(E_{/F})(1)$ with respect to orders in $\QQ_p[G]$ that are in general far from regular and are in many cases equal to $\ZZ_p[G]$ (see Theorem \ref{theorem ell 2}).

We recall that there are by now many circumstances in which $E$ is known to validate both Perrin-Riou's Conjecture and the $p$-part of the Birch and Swinnerton-Dyer Conjecture over $\QQ$ (see Remark \ref{cor kato rem}) and, for this reason, Theorems \ref{cor kato} and \ref{theorem ell 2} constitute very strong improvements of what has been proved previously concerning both the Galois structure of Selmer groups and the validity of the equivariant Tamagawa number conjecture for elliptic curves (for more details in this regard see Remarks \ref{kurihara rem} and \ref{rem ell 2}).

\subsection{Some general notation} In this article, $K $ will always denote a number field. We also fix an algebraic closure $\QQ^c$ of $\QQ$ and regard every algebraic extension of $\QQ$ as a subfield of $\QQ^c$. For any subfield $F$ of $\QQ^c$ we set $G_F := \Gal(\QQ^c/F)$.

Let $p$ be a prime number and $A$ a continuous $G_K$-module. Then for each extension $F$ of $K$ we use the following notation.

\begin{itemize}
\item[] $S_\infty(F)$ is the set of archimedean places of $F$;
\item[] $S_p(F)$ is the set of $p$-adic places of $F$;
\item[] $S_{\rm ram}(F/K)$ is the set of places of $K$ that ramify in $F/K$;
\item[] $S_{\rm ram}(A)$ is the set of places of $K$ at which $A$ is ramified.
\end{itemize}

If $S$ is any finite set of places of $K$ that contains $S_\infty(K)$, then the ring of $S$-integers of $K$ is denoted by $\cO_{K,S}$. The ring of integers of $K$ is denoted by $\cO_K$. For a finite extension $F/K$, we denote by $S_F$ the set of places of $F$ which lie above a place in $S$ and we usually abbreviate $\cO_{F,S_F}$ to $\cO_{F,S}$.

For a natural number $n$, the group of $n$-th roots of unity in $\QQ^c$ is denoted by $\mu_n$.

Let $R$ be a commutative ring, and $M$ an $R$-module.
For any ideal $I$ we set
$$M[I] := \{x \in M \mid I x = 0\}.$$
If $M$ is finitely presented, then we denote by ${\rm Fitt}_R^j(M)$ the $j$-th Fitting ideal of $M$, as defined by Northcott in \cite{north}.

For an abelian group $M$ we write $M_{\rm div}$ for its maximal divisible subgroup. For a topological $\bZ_{p}$-module $M$ we write $M^{\vee}$ for its Pontryagin dual $\Hom_{\rm cont}(M, \bQ_{p}/\bZ_{p})$.

In the sequel we shall refer to non-archimedean places of number fields as `primes' and say that a condition is satisfied `for almost all primes' of a number field if the set of primes that fail to satisfy the condition has analytic density zero.



\section{Selmer structures} \label{section sel}

For the reader's convenience we shall first quickly recall some basic notation and definitions regarding the theory of Selmer structures. More details can be found in, for example, either the original article of Mazur and Rubin \cite{MRkoly} or our earlier article \cite{bss}.

\begin{definition}\label{sel def} Let $R$ be a finite $\Z_p$-algebra and $A$ a topological $R$-module endowed with a continuous $R$-linear action of $G_K$ for which $S_{\rm ram}(A)$ is finite. Then a `Selmer structure of $R$-modules' $\mathcal{F}$ on $A$ is a collection of the following data:
\begin{itemize}
\item a finite set $S(\mathcal{F})$ of places of $K$ such that
$S_\infty(K)\cup S_p(K) \cup S_{\rm ram}(A) \subset S(\mathcal{F})$;
\item for each place $v$ in $S(\mathcal{F})$ a choice of an $R$-submodule $H_\mathcal{F}^1(K_v, A)$ of $H^1(K_v, A)$.
\end{itemize}
In the case that $R=\ZZ_p$ we shall simply refer to a `Selmer structure' rather than to a Selmer structure of $\ZZ_p$-modules.
\end{definition}


In the following examples we fix  a finitely generated free $\ZZ_p$-module $T$ that is endowed with a continuous action of $G_K$ for which $S_{\rm ram}(T)$ is finite and write $S(T)$ for the set $S_\infty(K) \cup S_p(K) \cup S_{\rm ram}(T)$.

\begin{example}\label{can def} The `canonical Selmer structure' $\cF_{\rm can}$ on $T$ is defined (in \cite[Def. 3.2.1]{MRkoly}) so that $S(\cF_{\rm can}) = S(T)$ and for each $v$ in $S(\cF_{\rm can})$ one has
\begin{eqnarray*}
&&H_{\cF_{\rm can}}^1(K_v,T)\\
&:=&\begin{cases}
H^1_f(K_v,T):=\ker(H^1(K_v,T) \to  H^1(K_v^{\rm ur}, T \otimes_{\ZZ_p}\QQ_p)) &\text{if $v \notin S_\infty(K)\cup S_p(K)$,}\\
H^1(K_v,T) &\text{if $v \in S_\infty(K)\cup S_p(K)$.}
\end{cases}
\end{eqnarray*}
Here we write $K_v^{\rm ur}$ for the maximal unramified extension of $K_v$. 
\end{example}

\begin{example}\label{un def} The `unramified Selmer structure' $\cF_{\rm ur}$ on $T$ is defined (in \cite[Def.~5.1]{MRselmer}) so that $S(\cF_{\rm ur}) =S(T)$ and $H_{\cF_{\rm ur}}^1(K_v,T) = H_{\cF_{\rm can}}^1(K_v,T)$ for each $v$ in $S(\cF_{\rm ur})\setminus S_p(K)$. For each $v$ in $S_p(K)$ the group $H_{\cF_{\rm ur}}^1(K_v,T)$ is defined to be the saturation in $H^1(K_v,T)$ of the universal norm subgroup $\bigcap_L {\rm Cor}_{L/K_v}(H^1(L,T))$
where $L$ runs over all finite unramified extensions of $K_v$. \end{example}

\begin{example}\label{rel def} We define the `relaxed Selmer structure' $\cF_{{\rm rel}}$ on $T$ by setting $S(\cF_{\rm rel}) = S(T)$ and $H_{\cF_{{\rm rel}}}^1(K_v,T) := H^1(K_v,T)$  for every $v$ in $S(\cF_{\rm rel})$.
\end{example}
%

\begin{remark} In \cite[Cor.~5.3]{MRselmer}, Mazur and Rubin show that the Selmer structures $\cF_{\rm can}$ and $\cF_{\rm ur}$ coincide if and only if for every $p$-adic place $\fp$ of $K$ the group
$H^{0}(K_{\fp}, T^{\vee}(1))$ is finite. The Selmer structure $\cF_{\rm rel}$ naturally arises in the context of the article \cite{sbA}. 
\end{remark}

\begin{remark} There are several ways in which a Selmer structure $\cF$ on a representation $T$ as in the above examples induces Selmer structures on associated representations.

\noindent{}(i) The `dual Selmer structure' $\cF^\ast$ on the Kummer dual representation $T^\vee(1)$ is defined by the condition that $S(\cF^\ast) = S(\cF)$ and for each $v$ in $S(\cF^\ast)$ the group $H^1_{\cF^\ast}(K_v, T^\vee(1))$ is the kernel of the composite homomorphism
\[ H^1(K_v, T^\vee(1)) \simeq H^1(K_v, T)^\vee \to H^1_{\cF}(K_v,T)^\vee \]
where the isomorphism is given by local Tate duality and the second map is the dual of the inclusion $H^1_{\cF}(K_v,T) \to  H^1(K_v,T)$.

\noindent{}(ii) The induced Selmer structure $\cF(T')$ on a submodule $T'$ of $T$ is defined by the condition that $S(\cF(T')) = S(\cF)$ and for each $v$ in $S(\cF(T'))$ the group $H^1_{\cF(T')}(K_v, T')$ is the kernel of the natural map $H^1(K_v,T') \to  H^1(K_v,T)/H^1_\cF(K_v,T)$. For simplicity we shall write $\cF$ rather than $\cF(T')$ (since the subrepresentation will always be clear from context).

\noindent{}(iii) The induced Selmer structure $\cF(A)$ on a quotient $A$ of $T$ is defined by the condition that $S(\cF(A)) = S(\cF)$ and for each $v$ in $S(\cF(A))$ the group $H^1_{\cF(A)}(K_v, A)$ is the image of the natural map $H^1_{\cF}(K_v,T) \to H^1(K_v,A)$. For simplicity we write $\cF$ in place of $\cF(A)$.
\end{remark}


For any Selmer structure $\cF$ as in Definition \ref{sel def} and any place $v$ of $K$ outside $S(\cF)$ we set $H^1_\cF(K_v,A):= H^1_{\rm ur}(K_v,A)$. For each place $v$ of $K$ we also set
\[ H^{1}_{/\cF}(K_{v}, A) := H^{1}(K_v, A)/H^{1}_{\cF}(K_v,A).\]

The `Selmer module' $H^1_{\cF}(K,A)$ of $\cF$ on $A$ is then defined to be the $R$-module obtained as the kernel of the diagonal localization map
\[ H^1(K, A) \to \bigoplus_{v} H^1(K_v, A)/H^1_{\cF}(K_v,A)\]
where $v$ runs over all places of $K$.

The principal aim of the theory of higher rank Euler, Kolyvagin and Stark systems is to understand the structure of such Selmer modules.

\begin{example}\label{selmer exams} In concrete cases the Selmer modules recover natural arithmetic objects. We recall a few relevant examples.

\noindent{}(i) Let $\cF_{\rm can}$ and  $\cF_{\rm ur}$ be the canonical and unramified  Selmer structures on $\ZZ_p(1)$ respectively. Then the modules $H^1_{\cF^\ast_{\rm can}}(K,\QQ_p/\ZZ_p)$ and $H^1_{\cF^\ast_{\rm ur}}(K,\QQ_p/\ZZ_p)$ 
respectively identify with the Pontryagin duals of the $p$-primary subgroups of the ideal class groups of the rings $\cO_{K}[1/p]$ and $\cO_{K}$. 

\noindent{}(ii) Let $E$ be an elliptic curve defined over $K$, write ${\rm Sel}_{p}(E/K)$ for its classical $p$-Selmer group and define the `strict $p$-Selmer group' ${\rm Sel}_{p}^{\rm str}(E/K)$ to be the kernel of the natural localization map
\[ {\rm Sel}_{p}(E/K) \to \bigoplus_{\fp \in S_{p}(K)}H^{1}(K_{\fp}, T \otimes_{\ZZ_p}\QQ_p /\ZZ_p).\]
where $T$ denotes the $p$-adic Tate module of $E$. Then both of the groups $H^1_{\cF_{\rm can}^\ast}(K,T^\vee(1))$ and $H^1_{\cF_{\rm ur}^\ast}(K,T^\vee(1))$ identify with ${\rm Sel}_{p}^{\rm str}(E/K)$.

\noindent{}(iii) 
 Let $\cF_{\rm rel}$ be the relaxed Selmer structure on $T$ as in Example \ref{rel def}. Then an easy exercise in global duality shows that $H^1_{\cF_{\rm rel}^\ast}(K,T^\vee(1))^\vee$ identifies with
$$\ker\left(H^2(\cO_{K,S(T)},T) \to \bigoplus_{v \in S(T)}H^2(K_v,T)\right).$$ \end{example}

\section{Higher rank Euler, Kolyvagin and Stark systems} \label{section euler}

\subsection{Notation and hypotheses} \label{section hyp}

Throughout this section we assume that $p$ is odd.

\subsubsection{Field extensions} For each prime $\fq$ of $K$ we write $K(\fq)$ for the maximal $p$-extension inside the ray class field of $K$ modulo $\fq$. We denote by $K(1)$ the maximal $p$-extension inside the Hilbert class field of $K$.

For a natural number $n$, we set
$$K_{p^n}:=K(1)K(\mu_{p^n},(\cO_K^\times)^{1/p^n})$$
and we write $K_{p^\infty}$ for the field $\bigcup_{n >0} K_{p^n}.$

We then fix an abelian pro-$p$-extension $\cK$ of $K$ that satisfies the following hypothesis. (Recall that we say that a condition is satisfied `for almost all primes' if the set of primes that fail to satisfy the condition has analytic density zero.)

\begin{hypothesis}\label{hyp K} The field $\cK$ contains $K(\fq)$ for almost all primes $\fq$ of $K$ and a non-trivial $\ZZ_p$-power extension $K_\infty$ of $K$ in which no non-archimedean place of $K$ splits completely.
\end{hypothesis}

\begin{remark} The condition that no non-archimedean place splits completely in $K_\infty/K$ is automatically satisfied if $K_\infty$ contains the cyclotomic $\ZZ_p$-extension of $K$.\end{remark}

We write $\Omega(\cK/K)$ for the set of finite extensions of $K$ in $\cK$.

\subsubsection{The representations} We fix a finite extension $Q$ of $\QQ_p$ and a uniformizer $\varpi$ of the valuation ring $\cO$ of $Q$ and we write $\Bbbk$ for the residue field $\cO/\varpi\cO$.

We also fix a finitely generated free $\cO$-module $T$ that is endowed with a continuous $\cO$-linear action of $G_K$ for which the set $S_{\rm ram}(T)$ is finite.

We denote the Kummer dual $\Hom_\cO(T, \cO(1))$ of $T$ by $T^\ast(1)$ and write $\overline T$ for the residual representation $T\otimes_\cO \Bbbk$.

We then fix a field $F$ in $\Omega(\cK/K)$ and define representations
\[ \cT :={\rm Ind}_{G_K}^{G_F} (T)\,\,\text{ and }\,\, \cA_1:=\cT/\varpi\cT.\]

We note that both of the representations are unramified outside the finite set
\begin{equation}\label{Sdef} S := S_{\rm ram}(F/K) \cup S_{\rm ram}(T) \cup S_{p}(K)\end{equation}
and are also endowed with a natural action of the ring $\cO[\Gal(F/K)]$.

Finally we set $G := \Gal(F/K)$ and note that the injective homomorphism of $\mathcal{O}$-modules
\begin{align}\label{nat inj}
\Bbbk \to \Bbbk [G]
\end{align}
that sends $1$ to $\sum_{g\in G}g$ will play an important role in the sequel.

\subsubsection{The standard hypotheses}\label{stand hyp sec} For each $G_K$-module $A$ we write $K(A)$ for the minimal Galois extension of $K$ such that $G_{K(A)}$ acts trivially on $A$.

For each natural number $n$ we then set
$$F_{p^n}:=FK_{p^n}, \ F_{p^\infty}:=FK_{p^\infty}, \ F(T)_{p^n}:=F_{p^n}K(T/p^nT) \text{ and } F(T)_{p^\infty}:=F_{p^\infty}K(T)$$
%

We can now state our standard hypotheses on the representation $T$ and field $F$.

\begin{itemize}
\item[(H$_0$)] For almost all primes $\fq$ of $K$ the map ${\rm Fr}_\fq^{p^k}-1$ is injective on $T$ for every $k \geq 0$;
\item[(H$_1$)] the $\Bbbk[G_{K}]$-module $\overline T$ is irreducible;
\item[(H$_2$)] there exists $\tau$ in $G_{F_{p^\infty}}$ for which the $\cO$-module $T/(\tau -1)T$ is free of rank one;
\item[(H$_3$)] the groups $H^{1}(F(T)_{p^{\infty}}/K, \overline T)$ and $H^{1}(F(T)_{p^{\infty}}/K, \overline T^\vee(1))$ vanish;
\item[(H$_4$)] if $p = 3$, then $\overline{T}$ and $\overline{T}^{\vee}(1)$ have no nonzero isomorphic $\bZ_{p}[G_{K}]$-subquotients;
\item[(H$_5$)] for each $\fq$ in $S$ the group $H^{0}(K_{\fq}, \overline T^{\vee}(1))$ vanishes;
\item[(H$^{\rm c}_5$)] for each $\fq$ in $S$ the map $H^{1}_{/\cF_{\rm can}}(K_{\fq}, \overline{T}) \to H^{1}_{/\cF_{\rm can}}(K_{\fq}, \cA_{1})$ induced by (\ref{nat inj}) is injective;
\item[(H$_5^{\rm u}$)] for each $\fq$ in $S$ the map $H^{1}_{/\cF_{\rm ur}}(K_{\fq}, \overline{T}) \to H^{1}_{/\cF_{\rm ur}}(K_{\fq}, \cA_{1})$ induced by (\ref{nat inj}) is injective.
    \end{itemize}

\begin{remark}\label{numbers remark} The above numbering is motivated by the fact that, after taking account of \cite[Rem. 4.8]{bss}, the hypotheses (H$_1$), (H$_2$), (H$_3$) and (H$_4$) respectively correspond to the hypotheses (H.1), (H.2), (H.3) and (H.4) that are used by Mazur and Rubin in \cite{MRkoly} and \cite{MRselmer}. In addition, hypothesis (H$_0$) corresponds to the assumption (b) in \cite[Th. 3.2.4]{MRkoly}. In Lemma~\ref{compare} below it is shown that the hypothesis (H$_5$) is stronger than either (H$^{\rm c}_5$) or (H$^{\rm u}_5$) and implies that the Selmer structures $\cF_{\rm can}$, $\cF_{\rm ur}$ and $\cF_{\rm rel}$ on $\cT$ coincide.
\end{remark}

\begin{remark}\label{restrict scalars rem}  It is clear that if the data $T$ and $F$ satisfy any given subset of the above hypotheses, then so does the data $T$ and $F$ for any intermediate field $F'$ of $F/K$.
\end{remark}

\subsection{Statement of the main result}\label{main statements} In this section we fix data $K, \cK, F$ and $T$ as above. Before stating the main result we quickly review some basic notation from \cite{bss}.


For each natural number $r$ we write ${\rm ES}_r(T,\cK)$ for the module of Euler systems of rank $r$, as defined in \cite[Def.~6.4]{bss}.


Let $\cF$ denote either of the Selmer structures $\cF_{\rm can}$ or $\cF_{\rm ur}$ on $\cT$.
 We assume that the hypotheses (H$_1$), (H$_2$), (H$_3$) and (H$_{4}$) are satisfied.
We also assume that (H$_{5}^{\rm c}$), respectively (H$_{5}^{\rm u}$), is satisfied if $\cF$ denotes $\cF_{\rm can}$, respectively $\cF_{\rm ur}$.

Then, by Proposition~\ref{core equiv} below, for each natural number $r$, there exist $\cO[G]$-modules ${\rm KS}_r(\cT,\cF)$ and ${\rm SS}_r(\cT,\cF)$ of Kolyvagin and Stark systems of rank $r$, as defined in \cite[Def. 5.24 and~4.11]{bss}.

We also recall that each system $c=(c_{F'})_{F' \in \Omega(\cK/K)}$ in ${\rm ES}_r(T,\cK)$, $\kappa$ in ${\rm KS}_r(\cT,\cF)$ and $\epsilon$ in ${\rm SS}_r(\cT,\cF)$ gives rise for each non-negative integer $j$ to a canonical ideal
\[ I(c_F), \ I_j(\kappa), \text{ and }I_j(\epsilon)\]
of $\cO[G]$.

We recall, in particular, that $I(c_F)$ is simply defined to be the image of $c_F$ which is regarded as a homomorphism ${\bigwedge}_{\cO[G]}^r \Hom_{\cO[G]}(H^1(\cO_{F,S},T),\cO[G]) \to \cO[G]$ by means of the definition \cite[Def.~2.1]{bss}, whilst the definitions of $I_j(\kappa)$ and $I_j(\epsilon)$ are respectively given in \cite[Def.~5.24 and 4.11]{bss}.

We follow Mazur and Rubin in defining the `core rank' of a Selmer structure $\cF$ on $\cT$ to be the integer
\begin{equation}\label{core rank def}
\chi(\cF)=\chi(\cF,\cT) := \dim_{\Bbbk}H^{1}_{\cF}(K, \overline{T}) - \dim_{\Bbbk}H^{1}_{\cF^{*}}(K, \overline{T}^{\vee}(1)).
\end{equation}
For convenience we also abbreviate $\chi(\cF_{\rm can},\cT)$ and $\chi(\cF_{\rm ur},\cT)$ to $\chi_{\rm can}(\cT)$ and $\chi_{\rm ur}(\cT)$.

\begin{remark} The canonical (resp. unramified) Selmer structure on $T$ is in general different from the structure that is induced on $T$ by $\cF_{\rm can}$ (resp. $\cF_{\rm ur}$) on $\cT$. Thus, whilst the notation used in (\ref{core rank def}) suggests that $\chi(\cF,\cT)$ only depends only on $ T$, it does actually depend on the induced representation $\cT$.\end{remark}

We can now state the main result of this article concerning the theories of higher rank Euler, Kolyvagin and Stark systems.

\begin{theorem} \label{main} Let $K, \cK, F$ and $T$ be as above and assume that the hypotheses (H$_{0}$),  (H$_{1}$), (H$_2$), (H$_{3}$) and (H$_{4}$) are satisfied. Then all of the following claims are valid.

\begin{itemize}
\item[(i)] For each strictly positive integer $r$ there exists a canonical `higher Kolyvagin derivative' homomorphism 
$$
\cD_{r} \colon {\rm ES}_{r}(T, \cK) \to {\rm KS}_{r}(\cT, \cF_{{\rm can}}).
$$
\end{itemize}

In the rest of this result we let $\cF$ denote either $\cF_{\rm can}$ or $\cF_{\rm ur}$ and we always assume that the hypothesis (H$_{5}^{\rm c}$), respectively (H$_{5}^{\rm u}$), is satisfied if $\cF$ denotes $\cF_{\rm can}$, respectively $\cF_{\rm ur}$.

\begin{itemize}
\item[(ii)] One has
\[ \chi_{\rm can}(\cT) = \sum_{v \in S_\infty(K)\cup S_p(K)}{\rm rank}_\cO\left(H^0(K_v, T^\ast(1))\right)\]
and
\[ \chi_{\rm ur}(\cT) = \sum_{v \in S_\infty(K)}{\rm rank}_\cO\left(H^0(K_v, T^\ast(1))\right),\]
and the $\mathcal{O}[G]$-module ${\rm SS}_{\chi(\cF)}(\cT, \cF)$ is free of rank one.

Furthermore, if $\chi(\cF) $ is strictly positive, then there exists a canonical `regulator' isomorphism of $\mathcal{O}[G]$-modules
\begin{align*}
{\cR}_{\chi(\cF)} \colon {\rm SS}_{\chi(\cF)}(\cT, \cF) \xrightarrow{\sim} \, {\rm KS}_{\chi(\cF)}(\cT, \cF)
\end{align*}
and the $\mathcal{O}[G]$-module ${\rm KS}_{\chi(\cF)}(\cT, \cF)$ is free of rank one. 
\item[(iii)] If the core rank $r := \chi(\cF)$ is strictly positive, then any Euler system $c$ in ${\rm ES}_{r}(T, \cK)$ for which $\cD_{r}(c)$ belongs to ${\rm KS}_r(\cT,\cF)$ (as is automatically the case if $\cF= \cF_{\rm can}$ by claim (i)) has all of the following properties.
\begin{itemize}
\item[(a)] For all $j\ge 0$ one has
\[
I_j({\cR}_r^{-1}(\cD_{r}(c))) \subseteq {\rm Fitt}_{\cO[G]}^{j}(H^{1}_{\cF^{*}}(K, \cT^{\vee}(1))^{\vee})
\]
with equality if and only if $\cD_{r}(c)$ is an $\mathcal{O}[G]$-basis of ${\rm KS}_r(\cT,\cF)$. %
\item[(b)] For all $j\ge 0$ one has
\[ I_j(\cD_{r}(c)) \subseteq {\rm Fitt}_{\cO[G]}^{j}(H^{1}_{\cF^{*}}(K, \cT^{\vee}(1))^{\vee})\]
with equality if both $\cD_r(c)$ is an $\cO[G]$-basis of ${\rm KS}_r(\cT,\cF)$ and $\cO[G]$ is a principal ideal ring.
\item[(c)] One has
\[
I(c_{F}) \subseteq {\rm Fitt}_{\cO[G]}^{0}(H^{1}_{{\cF}^{*}}(K, \cT^{\vee}(1))^{\vee})
\]
with equality if and only if all inclusions in claim (a) are equalities.
\end{itemize}

\item[(iv)] If the hypothesis (H$_5$) is satisfied, then one has
$$\chi_{\rm can}(\cT)=\chi_{\rm ur}(\cT)=\chi(\cF_{\rm rel}).$$
Further, if this common integer, $r$ say, is strictly positive, then the homomorphism $\cD_r$ in claim (i) is surjective and there exist Euler systems $c$ in ${\rm ES}_{r}(T, \cK)$ for which all of the inclusions in claims (iii)(a) and (c) are equalities.
\end{itemize}
\end{theorem}

\begin{remark} The hypothesis (H$_0$) is only required for the construction of $\cD_r$ and so, in particular, 
is not required for the validity of Theorem \ref{main}(ii). In the context of Theorem \ref{main}(iv) we recall that the hypothesis (H$_5$) is stronger than both (H$_5^{\rm c})$ and (H$_5^{\rm u})$. 
\end{remark}

After establishing certain preliminary results, including an explicit characterization of the existence of `equivariant core vertices' in \S\ref{core vertex sec}, we shall prove Theorem \ref{main} in \S\ref{proof main}.

Then, in \S\ref{section 4} and \S\ref{section ell curves} we shall present concrete applications of Theorem \ref{main} in the context of the multiplicative group and of elliptic curves respectively.

\subsection{Preliminary observations}\label{prelim sec} We prove several useful preliminary results.

\begin{lemma}\label{vanish general}
Let $(R, \fm_{R})$ be an artinian local ring and $\cG$ a profinite group.
Fix a discrete free $R$-module $X$ endowed with a continuous $R$-linear action of $\cG$ and a non-negative integer $i$.

Then if the group $H^{i}(\cG, X/\fm_{R} X)$ vanishes the module $H^{i}(\cG, X)$ also vanishes.
\end{lemma}

\begin{proof} Let $j$ be a non-negative integer. Then, since $X$ is a free $R$-module, the sequence of $R[\cG]$-modules
$$
0 \to \fm^{j+1}_{R}X \to \fm^{j}_{R}X \to \fm^{j}_{R}/\fm^{j+1}_{R} \otimes_{R} X \to 0
$$
is exact and so induces, upon passing to cohomology, an exact sequence of $R$-modules
$$
H^{i}(\cG, \fm^{j+1}_{R}X) \to H^{i}(\cG, \fm^{j}_{R}X) \to H^{i}(\cG, \fm^{j}_{R}/\fm^{j+1}_{R} \otimes_{R} X).
$$
In particular, since the $R[\cG]$-module $\fm^{j}_{R}/\fm^{j+1}_{R} \otimes_{R} X$ is isomorphic to a direct sum of $\dim_{R/\fm_{R}} (\fm^{j}_{R}/\fm^{j+1}_{R})$ copies of $X/\fm_{R} X$, the assumed vanishing of $H^{i}(\cG, X/\fm_{R} X)$ implies that the group $H^{i}(\cG, \fm^{j}_{R}/\fm^{j+1}_{R} \otimes_{R} X)$ vanishes.

This in turn implies that the map $H^{i}(\cG, \fm^{j+1}_{R}X) \to H^{i}(\cG, \fm^{j}_{R}X)$ is surjective and, as $j$ is an arbitrary non-negative integer we conclude that the map $H^{i}(\cG, \fm^{j}_{R}X) \to H^{i}(\cG, X)$ is surjective.
Since $\fm^{j}_{R}X$ vanishes for sufficiently large integer $j$, the module $H^{i}(\cG, X)$ vanishes.

\end{proof}

\begin{lemma}\label{surjective}
Let $\fq$ be a prime of $K$ and $(R, \fm_{R})$ a compact complete noetherian local ring. 
Let $X$ be a free $R$-module of finite rank with an $R$-linear continuous $G_{K_{\fq}}$-action. If the group
 $H^{0}(K_{\fq}, (X/\fm_{R} X)^{\vee}(1))$ vanishes, then the following claims are valid.
\begin{itemize}
\item[(i)] For any finitely generated $R$-module $M$, the group $H^{2}(K_{\fq}, M \otimes_{R} X)$ vanishes.
\item[(ii)] Let $N_{1} \to N_{2}$ be a surjective homomorphism of finitely generated $R$-modules.
Then the induced homomorphism $H^{1}(K_{\fq}, N_{1} \otimes_{R} X) \to H^{1}(K_{\fq}, N_{2} \otimes_{R} X)$ is surjective.
\end{itemize}
\end{lemma}
\begin{proof}
To prove claim (i) we note first the assumed vanishing of $H^{0}(K_{\fq}, (X/\fm_{R} X)^{\vee}(1))$ combines with the local duality isomorphism $H^{2}(K_{\fq}, X/\fm_{R} X) \simeq H^{0}(K_{\fq}, (X/\fm_{R} X)^{\vee}(1))^{\vee}$ and the result of Lemma~\ref{vanish general} to imply that $H^{2}(K_{\fq}, X/\fm_{R}^{j}X)$ vanishes for any positive integer $j$.

In addition, since the order of $R/\fm_{R}^{j}$ is finite, the module $H^{1}(K_{\fq}, X/\fm_{R}^{j}X)$ is also finite. This implies that 
$H^{2}(K_{\fq}, X)$ identifies with the inverse limit $\varprojlim_{j} H^{2}(K_{\fq}, X/\fm_{R}^{j}X)$ and hence vanishes.

Next we fix a surjective homomorphism of $R$-modules of the form $R^{s} \to M$. Then, since the cohomological dimension of $G_{{K}_{\fq}}$ is $2$, the induced homomorphism
\[ H^{2}(K_{\fq}, X)^{s} = H^{2}(K_{\fq}, R^s\otimes_R X)\to  H^{2}(K_{\fq}, M \otimes_{R} X)\]
is surjective, and so $H^{2}(K_{\fq}, M \otimes_{R} X)$ vanishes, as claimed.

To prove claim (ii) we take $M$ to be the kernel of the given surjection $N_{1} \to N_{2}$. Then, since $X$ is a free $R$-module, the induced sequence of $R$-modules
$$
H^{1}(K_{\fq}, N_{1} \otimes_{R} X) \to H^{1}(K_{\fq}, N_{2} \otimes_{R} X) \to H^{2}(K_{\fq}, M \otimes_{R} X)
$$
is exact and so  claim~(ii) follows directly from claim~(i).\end{proof}

We can now use Lemma \ref{surjective} to justify Remark \ref{numbers remark}.

\begin{lemma}\label{compare} Fix a place $\fq$ in $S$ for which the group $H^{0}(K_{\fq}, \overline{T}^{\vee}(1))$ vanishes. Then for any ideal $I$ of $\cO[G]$ one has
$$
H^{1}_{\cF_{\rm ur}}(K_{\fq}, \cT/I\cT) = H^{1}_{\cF_{\rm can}}(K_{\fq}, \cT/I\cT) = H^{1}(K_{\fq}, \cT/I\cT).
$$
In particular, if hypothesis (H$_{5}$) is satisfied, then the Selmer structures $\cF_{\rm can}$, $\cF_{\rm ur}$ and $\cF_{\rm rel}$ coincide and the hypothesis (H$_{5}^{\rm c}$), and hence also (H$_{5}^{\rm u}$), is satisfied.
\end{lemma}

\begin{proof}
%

At the outset we note that it suffices to prove the first claim. Indeed, if this is true, then hypothesis (H$_5$) implies
$H^{1}_{/\cF_{\rm can}}(K_{\fq}, \overline{T})$ vanishes for each $\fq$ in $S$ and the second claim follows easily from this.

In addition, if $\fq$ is $p$-adic, then the first claim follows directly upon combining \cite[Cor.~5.3]{MRselmer} with Lemma~\ref{surjective}(ii).

It thus suffices to prove the first claim for a place $\fq$ that is not $p$-adic. But, in this case, both groups $H^{1}_{\cF_{\rm ur}}(K_{\fq}, \cT)$ and $H^{1}_{\cF_{\rm can}}(K_{\fq}, \cT)$ are defined to be
 $H^{1}_{f}(K_{\fq}, \cT)$. In addition, the assumed vanishing of $H^{0}(K_{\fq}, \overline{T}^{\vee}(1))$ combines with Lemma~\ref{surjective}(i) and \cite[Cor.~1.3.3(ii)]{R} to imply that
 $H^{1}_{\rm ur}(K_{\fq}, \cT \otimes_{\bZ_{p}}\bQ_{p}) = H^{1}(K_{\fq}, \cT \otimes_{\bZ_{p}}\bQ_{p})$ and hence that $H^{1}_{f}(K_{\fq}, \cT) = H^{1}(K_{\fq}, \cT)$. Given this, the first claim follows from Lemma~\ref{surjective}(ii). \end{proof}

We now end this section by recording a useful consequence of hypotheses (H$_1$) and (H$_3$).

\begin{lemma}\label{torsionfree} Assume the hypotheses (H$_1$) and (H$_3$). Then for any finite abelian $p$-power degree extension $F'$ of $K$ the group $H^{0}(F', T)$ 
vanishes and the group $H^{1}(F', T)$ is $\cO$-torsion-free.
\end{lemma}

\begin{proof} 
The argument of \cite[Rem.~4.8]{bss} shows that (H$_1$) and (H$_3$) combine to imply that the group $H^{0}(K, \overline{T})$ vanishes.

Set $R := \cO/(p)[\Gal(F'/K)]$. Then $R$ is an artinian local ring since $F'/K$ is a finite abelian $p$-extension and so, by applying Lemma~\ref{vanish general} with $X = {\rm Ind}_{G_{K}}^{G_{F'}}(\overline{T})$ and $\cG=G_{K}$, we deduce that the
group $H^{0}(F', \overline{T}) = H^{0}(K, {\rm Ind}_{G_{K}}^{G_{F'}}(\overline{T}))$ vanishes.

Given this, the claimed result follows immediately from the exact sequence of $\cO$-modules
\[ 0 \to H^{0}(F', T) \xrightarrow{\times \varpi} H^{0}(F', T) \to  H^{0}(F', \overline{T}) \to H^{1}(F', T)[\varpi] \to 0.\]
%
%
\end{proof}

\subsection{The existence of equivariant core vertices}\label{core vertex sec} In this section we characterize the existence of an appropriate analogue of the notion of `core vertex' (in the sense of Mazur and Rubin) in our setting.

For simplicity we set $\cR := \cO[G]$. Recalling that $\varpi$ is a uniformizer of $\cO$, for each natural number $m$ we set $\cR_m:=\cR/\varpi^m \cR$ and $\cA_m:=\cT/\varpi^m\cT$.

We assume hypothesis (H$_2$) is satisfied and fix an element $\tau$ of $G_{F_{p^\infty}}$ as in that hypothesis (so that the $\cR$-module $\cT/(\tau-1)\cT$ is free of rank one).

We fix a positive integer $m$ and, with $S$ as in (\ref{Sdef}), we write $\cP_{m}$ for the set of all primes $\fq$ of $K$ that do not belong to $S$ and are such that the associated Frobenius element ${\rm Fr}_{\fq}$ is conjugate to $\tau$ in $\Gal(F(T)_{p^m}/K)$.

We then write $\cN_{m}$ for the set of all square-free products of the primes in $\cP_{m}$. For an ideal $\fn$ in $\cN_{m}$, we denote by $\nu(\fn)$ the number of prime divisors of $\fn$. For each such $\fn$ we also use the modified Selmer structures $\cF^{\fn}$ and $(\cF^{*})_{\fn}$ that are defined in \cite[Exam.~2.1.8]{MRkoly} and \cite[Def.~2.3]{MRselmer} (and \cite[\S~3.1.3]{bss}).

The following definition is the analogue in our setting of the notion of core vertex used by Mazur and Rubin in \cite{MRkoly} and \cite{MRselmer} (see Remark \ref{core comp} below).

\begin{definition}\label{core def} We say that an ideal $\fn$ in $\cN_{m}$ is a `core vertex' of $\cF$ on $\cA_{m}$ if the $\cR_m$-module $H^{1}_{\cF^{\fn}}(K, \cA_{m})$ is free and the module $H^{1}_{(\cF^{*})_{\fn}}(K, \cA_{m}^{\vee}(1))$ vanishes.

\end{definition}

\begin{remark}\label{core remark} Assume that there exists a core vertex for $\cF$ on $\cA_{m}$. Then by combining the global duality theorem (as in \cite[Th.~3.1]{bss} or \cite[Th.~2.3.4]{MRkoly}) with \cite[Lem.~3.10]{bss}, and taking account of \cite[Lem.~5.4 and Rem.~5.9]{bss}, one finds that an ideal $\fn$ in $\cN_{m}$ is a core vertex for $\cF$ on $\cA_{m}$ if and only if the module $H^{1}_{(\cF^{*})_{\fn}}(K, \overline{T}^{\vee}(1))$ vanishes and, moreover, that the following claims are true for {\em every} core vertex $\fn$ in $\cN_{m}$:
\begin{itemize}
\item[(i)] if $I$ is any ideal of $\cR_{m}$ such that $\cR_{m}/I$ is a zero dimensional Gorenstein ring, then $\fn$ is also a core vertex for $\cF$ on $\cA_{m}/I \cA_{m}$;
\item[(ii)] the rank of the (free) $\cR_m$-module $H^{1}_{\cF^{\fn}}(K, \cA_{m})$ is $\chi(\cF) + \nu(\fn)$.
\end{itemize}
\end{remark}

The main result of this section characterizes the existence of core vertices in the above sense in a way that is straightforward to use in concrete examples.

To state this result we set $W := T \otimes_{\bZ_{p}}(\bQ_{p}/\bZ_p)$ and write $\cI_{\fq}$ for the inertia group at a prime $\fq$ of $K$.


\begin{theorem}\label{free} Let $\cF$ denote either the canonical or unramified Selmer structure on $\cT$. Then, if the standard hypotheses (H$_1$), (H$_2$) and (H$_3$) are satisfied, the following four conditions are equivalent:
\begin{itemize}
\item[(i)] For all natural numbers $m$ there exists a core vertex of $\cF$ on $\cA_{m}$;
\item[(ii)] For all natural numbers $m$ the Selmer structure $\cF$ on $\cA_{m}$ is `cartesian' in the sense of \cite[Def.~3.8]{sakamoto};
\item[(iii)] For every place $\fq$ in $S$, the homomorphism $H^{1}_{/\cF}(K_{\fq}, \overline{T}) \to H^{1}_{/\cF}(K_{\fq}, \cA_{1})$
 induced by (\ref{nat inj}) is injective;
\item[(iv)] All of the following conditions are satisfied:
\begin{itemize}
\item[(a)] for each $p$-adic place $\fp$ of $K$ the following is true:
\begin{itemize}
\item[-] if $\cF = \cF_{\rm can}$, then the natural homomorphisms
$$
H^{1}(K_{\fp},\cT) \to H^{1}(K_{\fp}, T)\,\,\text{ and }\,\, H^{0}(K_{\fp}, \cT^{\vee}(1)) \otimes_{\cO} \Bbbk \to H^{0}(K_{\fp}, T^{\vee}(1)) \otimes_{\cO} \Bbbk
$$
are both surjective;
\item[-] if $\cF = \cF_{\rm ur}$, then the homomorphism $H^{1}_{/\cF}(K_{\fp}, \overline{T}) \to H^{1}_{/\cF}(K_{\fp}, \cA_{1})$
 induced by (\ref{nat inj}) is injective;
\end{itemize}
\item[(b)] if $\fq$ belongs to $S \setminus (S_{\rm ram}(F/K) \cup S_{p}(K))$ and does not split completely in $F/K$, then the group $(W^{\cI_{\fq}}/(W^{\cI_{\fq}})_{\rm div})^{{\rm Fr}_{\fq}=1}$ vanishes;
\item[(c)] for all places $\fq$ in $S \setminus (S_{\rm ram}(T) \cup S_{p}(K))$ the group $H^{0}(K_{\fq}, \overline{T})$ vanishes;
\item[(d)] for all places $\fq \in (S_{\rm ram}(T) \cap S_{\rm ram}(F/K)) \setminus S_{p}(K)$, the natural homomorphisms
$$
H^{1}_{f}(K_{\fq}, \cT) \to H^{1}_{f}(K_{\fq}, T) \,\,\text{ and }\,\,
H^{1}_{f}(K_{\fq}, \cT^{*}(1)) \to H^{1}_{f}(K_{\fq}, T^{*}(1))
$$
are both surjective.
\end{itemize}
\end{itemize}
Furthermore, under the above equivalent conditions, one has
\[ \chi_{\rm can}(\cT) = \sum_{v \in S_\infty(K)\cup S_{p}(K)}{\rm rank}_\cO( H^0(K_v, T^\ast(1)))\]
and
\[ \chi_{\rm ur}(\cT) = \sum_{v \in S_\infty(K)}{\rm rank}_\cO(H^0(K_v, T^\ast(1))).\]
\end{theorem}

\begin{remark}\label{core comp} The notion of core vertex plays  key role in the theory developed by Mazur and Rubin in \cite{MRkoly} and \cite{MRselmer}. In the special case that $\cR_m$ is a principal ideal ring, Mazur and Rubin are able to verify the existence of core vertices under certain hypotheses (see \cite[Cor.~4.1.9]{MRkoly}). However, the structure theorem of modules over principal ideal rings plays a key role in their argument (see the proof of \cite[Th.~4.1.5]{MRkoly}) and so their method cannot be used directly to prove Theorem \ref{free}. In earlier articles \cite{sbA} and \cite{bss}, the authors developed a theory of Kolyvagin and Stark systems under the {\em assumed} existence of core vertices in the above sense (see, in particular, \cite[Hyp.~3.15 and Rem.~3.18]{sbA} and \cite[Hyp.~4.2]{bss}). The significance of Theorem ~\ref{free} is that it now allows one to check whether the hypotheses of loc. cit. are satisfied in concrete cases.
\end{remark}

\begin{remark}\label{cartesian} The notion of cartesian introduced in \cite[Def.~3.8]{sakamoto} differs slightly from the corresponding notion defined by Mazur and Rubin in \cite[Def. 1.1.4]{MRkoly}. However, in the setting of Galois representation over principal ideal rings, one can show that the two notions coincide. In \cite[Lem.~4.6]{sakamoto} the second author proved that, under the hypotheses (H$_{1}$), (H$_2$) and (H$_{3}$), any Selmer structure that is cartesian (in the sense of loc. cit.) has a core vertex in the sense of Definition \ref{core def}. Theorem \ref{free} now shows that the converse of this result is also true (for more details see the proof of Proposition~\ref{core equiv} below).
\end{remark}

Theorem \ref{free} will be proved in \S\ref{proof free}.

%

\subsection{Reinterpreting the conditions}\label{understanding} In this section we clarify the nature of the conditions stated in Theorem~\ref{free} in order to prepare for its proof.

We observe first that the condition of Theorem~\ref{free}(iii) asserts directly that the Selmer structure $\cF$ on $\cA_{1}$ is cartesian in the sense of \cite[Def.~3.8]{sakamoto}. 

The following result implies that whenever this condition is satisfied, then for any natural number $m$ the Selmer structure $\cF$ on $\cA_{m}$ is also cartesian.

\begin{lemma}\label{p-injective} Let $\cF$ be any Selmer structure on $\cT$ with the property that the $\cO$-module
$H^{1}_{/\cF}(K_{\fq}, \cT)$ is torsion-free for all primes $\fq$ in $S(\cF)$. Then for any such prime $\fq$ and any natural number $m$ the homomorphism
\[
H^{1}_{/\cF}(K_{\fq}, \cA_{1}) \to H^{1}_{/\cF}(K_{\fq}, \cA_{m})
\]
that is induced by the map $\cA_{1} \xrightarrow{\times \varpi^{m-1}} \cA_{m}$ is injective.
\end{lemma}

\begin{proof} The proof of this result mimics the argument used by Mazur and Rubin to prove \cite[Lem.~3.7.1(i)]{MRkoly}. However, for the sake of the reader, we provide the details.

Since $\varpi$ is a regular element of $\cR$, there exists a natural commutative diagram with exact rows of the form
\begin{align*}
\xymatrix{
0 \ar[r] &H^{1}(K_{\fq},\cT) \otimes_{\cR} \cR_{1} \ar[r] \ar[d]^{{\rm id} \otimes \varpi^{m-1}} & H^{1}(K_{\fq},\cA_{1}) \ar[r] \ar[d] & H^{2}(K_{\fq},\cT)[\varpi] \ar@{^{(}->}[d] \ar[r] & 0
\\
0 \ar[r] & H^{1}(K_{\fq},\cT) \otimes_{\cR} \cR_{m} \ar[r] & H^{1}(K_{\fq},\cA_{m}) \ar[r] & H^{2}(K_{\fq},\cT)[\varpi^{m}] \ar[r] & 0.
}
\end{align*}

In addition, since $H^{1}_{/\cF}(K_{\fq}, \cT)$ is a torsion-free $\cO$-module, the definitions of $H^{1}_{\cF}(K_{\fq}, \cA_{1})$ and $H^{1}_{\cF}(K_{\fq}, \cA_{m})$ ensure the existence of a commutative diagram
\begin{align*}
\xymatrix{
H^{1}_{\cF}(K_{\fq}, \cT) \otimes_{\cR} \cR_{1} \xrightarrow{\sim} H^{1}_{\cF}(K_{\fq}, \cA_{1}) \ar[d] \ar@{^{(}->}[r] &H^{1}(K_{\fq},\cT) \otimes_{\cR} \cR_{1} \ar@{^{(}->}[r] \ar[d]^{{\rm id} \otimes \varpi^{m-1}} & H^{1}(K_{\fq},\cA_{1}) \ar[d]
\\
H^{1}_{\cF}(K_{\fq}, \cT) \otimes_{\cR} \cR_{m} \xrightarrow{\sim} H^{1}_{\cF}(K_{\fq}, \cA_{m}) \ar@{^{(}->}[r] & H^{1}(K_{\fq},\cT) \otimes_{\cR} \cR_{m} \ar@{^{(}->}[r] & H^{1}(K_{\fq},\cA_{m}).
}
\end{align*}

Upon combining these two diagrams with the 'kernel-cokernel' exact sequence one deduces the existence of the following commutative diagram with exact rows
\begin{align*}
\xymatrix{
0 \ar[r] & H^{1}_{/\cF}(K_{\fq}, \cT) \otimes_{\cR} \cR_{1} \ar[r] \ar@{^{(}->}[d]^{{\rm id} \otimes \varpi^{m-1}} & H^{1}_{/\cF}(K_{\fq}, \cA_{1}) \ar[r] \ar[d] & H^{2}(K_{\fq},\cT)[\varpi] \ar@{^{(}->}[d] \ar[r] & 0
\\
0 \ar[r] & H^{1}_{/\cF}(K_{\fq}, \cT) \otimes_{\cR} \cR_{m} \ar[r] & H^{1}_{/\cF}(K_{\fq}, \cA_{m})  \ar[r] & H^{2}(K_{\fq},\cT)[\varpi^{m}] \ar[r] & 0.
}
\end{align*}
Finally we note that, since the $\cO$-module $H^{1}_{/\cF}(K_{\fq}, \cT)$ is torsion-free, the left hand vertical map in this diagram is injective
 and hence, therefore, the central vertical map is also injective, as claimed.

This completes the proof.
\end{proof}

The next result reduces the condition of Theorem \ref{free}(i) to the case $m=1$.

\begin{proposition}\label{core equiv} Let $\cF$ be any Selmer structure on $\cT$ with the property that the $\cO$-module
$H^{1}_{/\cF}(K_{\fq}, \cT)$ is torsion-free for all primes $\fq$ in $S(\cF)$. Then the following conditions are equivalent:
\begin{itemize}
\item[(a)] there exists a core vertex of $\cF$ on $\cA_{m}$ for any positive integer $m$;
\item[(b)] there exists a core vertex of $\cF$ on $\cA_{1}$;
\item[(c)] for any prime $\fq$ in $S(\cF)$, the homomorphism $H^{1}_{/\cF}(K_{\fq}, \overline{T}) \to H^{1}_{/\cF}(K_{\fq}, \cA_{1})$
induced by (\ref{nat inj}) is injective.
\end{itemize}
\end{proposition}
\begin{proof} Condition (a) obviously implies condition (b).

To show that condition (b) implies condition (c) we fix a core vertex $\fn$ in $\cN_{1}$ for $\cF$ on $\cA_{1}$ and set $S_{\fn} := S(\cF)\cup\{\fq \mid \fn\}$.

Then by using the global duality theorem (as in either \cite[Th.~3.1]{bss} or \cite[Th.~2.3.4]{MRkoly}), one deduces the existence of a natural commutative  diagram of the form
\begin{align*}
\xymatrix{
0 \ar[r] & H^{1}_{\cF^{\fn}}(K, \overline{T}) \ar[r] \ar[d] & H^{1}(\cO_{K, S_{\fn}}, \overline{T}) \ar[r] \ar[d]^{\simeq} & \bigoplus_{\fq \in S(\cF)}H^{1}_{/\cF}(K_{\fq}, \overline{T}) \ar[d] \ar[r] & 0
\\
0 \ar[r] & H^{1}_{\cF^{\fn}}(K, \cA_{1})[\fm_{\cR}] \ar[r] & H^{1}(\cO_{K, S_{\fn}}, \cA_{1})[\fm_{\cR}]  \ar[r] & \bigoplus_{\fq \in S(\cF)}H^{1}_{/\cF}(K_{\fq}, \cA_{1}),
}
\end{align*}
where $\fm_{R}$ denotes the maximal ideal of $\cR$.

By the result of \cite[Prop.~3.5]{bss}, one knows that the central vertical map in this diagram is bijective. It follows that the left hand vertical map is injective and hence that
$$
\dim_{\Bbbk}\left(H^{1}_{\cF^{\fn}}(K, \overline{T})\right) \leq \dim_{\Bbbk} \left(H^{1}_{\cF^{\fn}}(K, \cA_{1})[\fm_{\cR}]\right) = {\rm rank}_{\cR_{1}} \left(H^{1}_{\cF^{\fn}}(K, \cA_{1}) \right)
$$
where the last equality is valid because $H^{1}_{\cF^{\fn}}(K, \cA_{1})$ is a free $\cR_{1}$-module and the $\Bbbk$-dimension of $\cR_{1}[\fm_{\cR}]$ is one.

On the other hand, by \cite[Lem.~3.10]{bss}, one knows that the natural homomorphism
$H^{1}_{\cF^{\fn}}(K, \cA_{1}) \otimes_{\cR_{1}} \Bbbk \to H^{1}_{\cF^{\fn}}(K, \overline{T})$ is injective and hence that
${\rm rank}_{\cR_{1}} \left(H^{1}_{\cF^{\fn}}(K, \cA_{1}) \right) \le \dim_{\Bbbk}\left(H^{1}_{\cF^{\fn}}(K, \overline{T})\right)$.

Taken together, these facts imply that the left hand vertical map in the above diagram is also bijective and hence that the right hand vertical map is injective, as claimed by condition (c).

Finally, we note that Lemma~\ref{p-injective} combines with \cite[Lem.~4.6]{sakamoto} to show that condition (c) implies condition (a).

This completes the proof.
\end{proof}

In the following result we clarify the nature of the conditions in Theorem \ref{free}(iv).

\begin{proposition}\label{explicit} The conditions in Theorem \ref{free}(iv)(b), (c), and (d) are all valid if and only if for all primes $\fq $ in $S \setminus S_{p}(K)$ the natural homomorphisms
\[ H^{1}_{f}(K_{\fq}, \cT) \to H^{1}_{f}(K_{\fq}, T)\,\,\text{ and }\,\, H^{1}_{f}(K_{\fq}, \cT^{*}(1)) \to H^{1}_{f}(K_{\fq}, T^{*}(1))\]
are both surjective.
\end{proposition}

\begin{proof} It suffices to show that if either $\fq$ does not belong to $S_{\rm ram}(F/K)$ or $\fq$ belongs to $S_{\rm ram}(F/K) \setminus (S_{\rm ram}(T) \cup S_{p}(K))$, then the explicit conditions given in Theorem \ref{free}(iv)(b) and (c) are respectively equivalent to the surjectivity of the two maps given above.

However, before doing so, we note that for any prime $\fq$ in $S \setminus S_{p}(K)$ the result of \cite[Lem.~1.3.5(iii)]{R} gives a commutative diagram with exact rows
\begin{equation}\label{R diagram}
\xymatrix{
0 \ar[r] & \cT^{\cI_{\fq}}/({\rm Fr}_{\fq}-1)\cT^{\cI_{\fq}}  \ar[r] \ar[d] & H^{1}_{f}(K_\fq, \cT) \ar[r] \ar[d] & (\cW^{\cI_{\fq}}/(\cW^{\cI_{\fq}})_{\rm div})^{{\rm Fr}_{\fq} = 1} \ar[d] \ar[r] & 0
\\
0 \ar[r] & T^{\cI_{\fq}}/({\rm Fr}_{\fq}-1)T^{\cI_{\fq}}  \ar[r] & H^{1}_{f}(K_\fq, T)  \ar[r] & (W^{\cI_{\fq}}/(W^{\cI_{\fq}})_{\rm div})^{{\rm Fr}_{\fq} = 1} \ar[r] & 0,
}
\end{equation}
in which the vertical maps are induced by the natural projection $\cR \to \cO$ and we set $\cW := {\rm Ind}^{G_{F}}_{G_{K}}(W)$.

We assume first that $\fq$ does not belong to $S_{\rm ram}(F/K)$. Then in this case one has $\cT^{\cI_{\fq}} =  {\rm Ind}^{G_{F}}_{G_{K}}(T^{\cI_{\fq}})$ and so the left hand vertical map in (\ref{R diagram}) is surjective. The surjectivity of the central vertical  map in (\ref{R diagram}) is thus equivalent to the surjectivity of the right hand vertical map and, since $\cW^{\cI_{\fq}}/(\cW^{\cI_{\fq}})_{\rm div} =  {\rm Ind}^{G_{F}}_{G_{K}}(W^{\cI_{\fq}}/(W^{\cI_{\fq}})_{\rm div})$, one finds that this is the case if and only if either  $(W^{\cI_{\fq}}/(W^{\cI_{\fq}})_{\rm div})^{{\rm Fr}_{\fq}=1}$ vanishes or the prime $\fq$ splits completely in $F/K$.

Furthermore, if $(W^{\cI_{\fq}}/(W^{\cI_{\fq}})_{\rm div})^{{\rm Fr}_{\fq}=1}$ vanishes, then \cite[Prop.~1.4.3(i)]{R} implies that
$H^{1}_{f}(K_{\fq}, T^{\vee}(1))$ is the orthogonal complement of $H^{1}_{\rm ur}(K_{\fq}, T)$ with respect to the local Tate pairing and hence that $H^{1}_{f}(K_{\fq}, T^{\vee}(1)) = H^{1}_{\rm ur}(K_{\fq}, T^{\vee}(1))$.

By applying \cite[Lem.~1.3.5(iii)]{R}, we can therefore deduce that the group
$$
T^{\vee}(1)^{\cI_{\fq}}/((T^{\vee}(1)^{\cI_{\fq}})_{\rm div} + ({\rm Fr}_{\fq}-1)T^{\vee}(1)^{\cI_{\fq}})
$$
vanishes. Then, since the quotient $T^{\vee}(1)^{\cI_{\fq}}/(T^{\vee}(1)^{\cI_{\fq}})_{\rm div}$ is finite, the group
$$
(T^{\vee}(1)^{\cI_{\fq}}/(T^{\vee}(1)^{\cI_{\fq}})_{\rm div})^{{\rm Fr}_{\fq}=1}
$$
also vanishes and so, by \cite[Lem.~1.3.5(iii)]{R}, one has
$$
T^{*}(1)^{\cI_{\fq}}/({\rm Fr}_{\fq}-1)T^{*}(1)^{\cI_{\fq}} = H^{1}_{f}(K_{\fq}, T^{*}(1)).
$$

This description implies that the natural homomorphism $H^{1}_{f}(K_\fq, \cT^{*}(1)) \to H^{1}_{f}(K_\fq, T^{*}(1))$ is surjective, as required.

We suppose now that $\fq$ belongs to $S_{\rm ram}(F/K) \setminus (S_{\rm ram}(T) \cup S_{p}(K))$. In this case both groups $\cW^{\cI_{\fq}}/(\cW^{\cI_{\fq}})_{\rm div}$ and $W^{\cI_{\fq}}/(W^{\cI_{\fq}})_{\rm div}$ vanish, one has $T = T^{\cI_{\fq}}$ and $\cT^{\cI_{\fq}}$ is naturally isomorphic to $T \otimes_{\cO} \cR^{\cI_{\fq}}$.

Since $\fq$ belongs to $S_{\rm ram}(F/K)$, the image of $\cT^{\cI_{\fq}}$ in $T$ is contained in $\varpi \cdot T$.
 In this case, therefore, the natural map $H^{1}_{f}(K_\fq, \cT) \to H^{1}_{f}(K_\fq, T)$ is surjective if and only if the quotient $T/({\rm Fr}_{\fq}-1)T$ vanishes. As the $\Bbbk$-vector space
$$
(T/({\rm Fr}_{\fq}-1)T )\otimes_{\cO} \Bbbk = \overline{T}/({\rm Fr}_{\fq}-1)\overline{T}
$$
is finite dimensional, this condition is in turn equivalent to the vanishing of $H^{0}(K_{\fq}, \overline{T})$.

Finally we note that, as $\fq$ ramifies in the abelian $p$-power degree extension $F/K$ one has
${\N}\fq \equiv 1\  (\bmod \  p)$, where ${\N}\fq$ denotes the cardinality of the residue field of $K$ at $\fq$, and so the vanishing of $H^{0}(K_{\fq}, \overline{T})$ implies
the vanishing of $H^{0}(K_{\fq}, \overline{T}^{\vee}(1))$, as required.

This completes the proof. \end{proof}

We finish this subsection by explaining the connection between the explicit conditions in Theorem \ref{free} (iii) and (iv)(a) in the case that $\cF$ is $\cF_{\rm can}$.

\begin{lemma}\label{can exp p} For any $p$-adic place $\fp$ of $K$ the following conditions are equivalent.
\begin{itemize}
\item[(a)] The homomorphism $H^{1}_{/\cF_{\rm can}}(K_{\fp}, \overline{T}) \to H^{1}_{/\cF_{\rm can}}(K_{\fp}, \cA_{1})$
induced by (\ref{nat inj}) is injective;
\item[(b)] The natural maps
$$
H^{1}(K_{\fp},\cT) \to H^{1}(K_{\fp}, T)\,\,\text{ and }\,\, H^{0}(K_{\fp}, \cT^{\vee}(1)) \otimes_{\cO} \Bbbk \to H^{0}(K_{\fp}, T^{\vee}(1)) \otimes_{\cO} \Bbbk
$$
are surjective.
\end{itemize}
\end{lemma}
\begin{proof} The tautological short exact sequence $0 \to \cT \xrightarrow{\times \varpi} \cT \to \cA_{1} \to 0$ induces an isomorphism of
$H^{1}_{/\cF_{\rm can}}(K_{\fp}, \cA_{1}) := {\rm Coker}\left(H^{1}(K_{\fp}, \cT) \to H^{1}(K_{\fp}, \cA_{1})\right)$ with $H^{2}(K_{\fp}, \cT)[\varpi]$.

The homomorphism (\ref{nat inj}) therefore induces a commutative diagram:
\begin{align*}
\xymatrix{
H^{1}_{/\cF_{\rm can}}(K_{\fp}, \overline{T}) \ar@{->>}[d] \ar[r] &
H^{1}_{/\cF_{\rm can}}(K_{\fp}, \cA_{1}) \ar[d]^{\simeq}
\\
H^{2}(K_{\fp}, T)[\varpi] \ar[r]  & H^{2}(K_{\fp}, \cT)[\varpi]
}
\end{align*}
in which the left vertical map is induced by the exact sequence
\begin{equation}\label{non-eq} 0 \to T \xrightarrow{\times \varpi} T \to \overline{T} \to 0\end{equation}
and is surjective.

Condition (a) is thus valid if and only if the left vertical and lower horizontal maps in this diagram are both injective.

Now, by using local Tate duality, injectivity of the lower map is equivalent to surjectivity of the natural map $
H^{0}(K_{\fp}, \cT^{\vee}(1)) \otimes_{\cO} \Bbbk \to H^{0}(K_{\fp}, T^{\vee}(1)) \otimes_{\cO} \Bbbk.$

To consider the left hand map we note $H^{1}_{\cF_{\rm can}}(K_{\fp}, \overline{T})$ is defined to be the image of the natural map $H^{1}(K_{\fp}, \cT) \to H^{1}(K_{\fp}, \overline{T})$ and that (\ref{non-eq}) induces an
 exact sequence of the form
\[   H^{1}(K_{\fp}, T) \xrightarrow{\times \varpi} H^{1}(K_{\fp}, T) \to H^{1}(K_{\fp}, \overline{T}) \to H^{2}(K_{\fp}, T)[\varpi]\to 0. \]

In view of Nakayama's Lemma, this shows injectivity of the left hand map in the above diagram is equivalent to
surjectivity of the natural map $H^{1}(K_{\fp},\cT) \to H^{1}(K_{\fp}, T)$, as required.
\end{proof}

\subsection{The proof of Theorem \ref{free}}\label{proof free} The definitions of the Selmer structures $\cF_{\rm can}$ and $\cF_{\rm ur}$ imply, firstly, that they both satisfy the hypotheses of Proposition~\ref{core equiv} and, secondly, that for every prime $\fq$ in $S \setminus S_{p}(K)$ one has $H^{1}_{\cF_{\rm can}}(K_{\fq}, \cT) = H^{1}_{\cF_{\rm ur}}(K_{\fq}, \cT)$.

Given this, the observations made in \S\ref{understanding} reduce the proof of Theorem \ref{free} to showing that for any prime $\fq$ in $S \setminus S_{p}(K)$ the natural homomorphism
\begin{equation}\label{inj req}
H^{1}_{/\cF}(K_{\fq}, \overline{T}) \to H^{1}_{/\cF}(K_{\fq}, \cA_{1})\end{equation}
is injective (as required by Theorem \ref{free}(iii)) if and only if both of the natural maps
\begin{equation}\label{q case}
 H^{1}_{f}(K_{\fq}, \cT) \to H^{1}_{f}(K_{\fq}, T)\,\,\text{ and }\,\, H^{1}_{f}(K_{\fq}, \cT^{*}(1)) \to H^{1}_{f}(K_{\fq}, T^{*}(1)).
\end{equation}
are surjective.

To prove this we write $\cG$ for the canonical, respectively unramified, Selmer structure on $T$ (and its quotients) in the case that $\cF$ denotes the canonical, respective unramified, Selmer structure on $\cT$ (and its quotients).



Then, for each $\fq$  in $S$, the definitions of $\cF$ and $\cG$ combine to imply that the homomorphism (\ref{nat inj}) induces a commutative diagram
\begin{align}\label{triangle}
\xymatrix{
H^{1}_{/\cF}(K_{\fq}, \overline{T}) \ar@{->>}[d] \ar[rd] &
\\
H^{1}_{/\cG}(K_{\fq}, \overline{T}) \ar[r] & H^{1}_{/\cF}(K_{\fq}, \cA_{1})
}
\end{align}
in which the vertical map denotes the natural homomorphism (and so is surjective).

In particular, if we assume that the map (\ref{inj req}) is injective, then the commutativity of this diagram implies that $H^{1}_{\cG}(K_{\fq}, \overline{T}) = H^{1}_{\cF}(K_{\fq}, \overline{T})$.

 But, since for each such $\fq$ the $\cO$-module $H^{1}(K_{\fq}, T)/H^{1}_{\cG}(K_{\fq}, T)$ is torsion free one has
\begin{align*}
\ker\left(H^{1}_{\cG}(K_{\fq}, T) \to H^{1}_{\cG}(K_{\fq}, \overline{T})\right)
&= \ker\left(H^{1}(K_{\fq}, T) \to H^{1}(K_{\fq}, \overline{T})\right) \cap H^{1}_{\cG}(K_{\fq}, T)
\\
&= \varpi \cdot H^{1}_{\cG}(K_{\fq}, T)
\end{align*}
and so, in this case, Nakayama's Lemma implies that
\begin{align}\label{induce eq}
H^{1}_{\cG}(K_{\fq}, T) = H^{1}_{\cF}(K_{\fq}, T).
\end{align}


%

We now assume that $\fq$ belongs to $S \setminus S_{p}(K)$. Then (\ref{induce eq}) implies that the natural map
$$
H^{1}_{f}(K_{\fq},\cT) \to H^{1}_{f}(K_{\fq},T)
$$
is surjective, whilst \cite[Prop.~1.4.3]{R} combines with the injectivity of (\ref{inj req}) to imply that the natural map $
H^{1}_{f}(K_{\fq}, \cA_{1}^{\vee}(1)) \to H^{1}_{f}(K_{\fq}, \overline{T}^{\vee}(1))$ is surjective. Here $H^{1}_{f}(K_{\fq}, \cA_{1}^{\vee}(1))$ and $H^{1}_{f}(K_{\fq}, \overline{T}^{\vee}(1))$ are respectively equal to the images of $H^{1}_{f}(K_{\fq}, \cT^{*}(1))$ and $H^{1}_{f}(K_{\fq}, T^{*}(1))$ in $H^{1}(K_{\fq}, \cA_{1}^{\vee}(1))$ and $H^{1}(K_{\fq}, \overline{T}^{\vee}(1))$ (by \cite[Lem.~1.3.8(i)]{R}).

It follows that there exists a commutative diagram
\begin{align*}
\xymatrix{
H^{1}_{f}(K_{\fq}, \cT^{*}(1)) \ar@{->>}[r] \ar[d] & H^{1}_{f}(K_{\fq}, \cA_{1}^{\vee}(1)) \ar@{->>}[d]
\\
H^{1}_{f}(K_{\fq}, T^{*}(1)) \ar@{->>}[r] & H^{1}_{f}(K_{\fq}, \overline{T}^{\vee}(1))
}
\end{align*}
in which all but the left hand vertical map are known to be surjective.

But, by the same argument as above, the kernel of the natural projection
\[ H^{1}_{f}(K_{\fq}, T^{*}(1)) \to H^{1}_{f}(K_{\fq}, \overline{T}^{\vee}(1))\]
is equal to
 $\varpi \cdot H^{1}_{f}(K_{\fq}, T^{*}(1))$ and so Nakayama's Lemma implies that the left hand vertical map in the above diagram must also be surjective, as required.

At this stage we have shown that for each $\fq$ in $S \setminus S_{p}(K)$ injectivity of the map (\ref{inj req}) implies surjectivity of both of the maps in 
(\ref{q case}).

To prove the converse we first note that surjectivity of the maps in (\ref{q case}) implies, for each $\fq$ in $S \setminus S_{p}(K)$, that both the natural map
$H^{1}_{f}(K_{\fq}, \cA_1^{\vee}(1)) \to H^{1}_{f}(K_{\fq}, \overline{T}^{\vee}(1))$ is surjective and, in addition, that $H^{1}_{\cF}(K_{\fq}, \overline{T}) = H^{1}_{\cG}(K_{\fq}, \overline{T})$. Given this fact, one need only folow the above argument in reverse to deduce that the map (\ref{inj req}) is injective, as required.

Finally we note that if $\cF = \cF_{\rm can}$, respectively $\cF = \cF_{\rm ur}$, then the equality (\ref{induce eq}) combines with the result of \cite[Th.~5.2.15]{MRkoly}, respectively \cite[Th.~5.4]{MRselmer}, and our definition (\ref{core rank def}) of the core rank $\chi(\cF,\cT)$ to imply the explicit formula given at the end of Theorem \ref{free}.

This completes the proof of Theorem \ref{free}.

\subsection{The proof of Theorem~\ref{main}}\label{proof main}

We are now ready to deduce Theorem~\ref{main} by combining Theorem \ref{free} with the main results of \cite{bss}.

We shall in fact only prove Theorem \ref{main} for the case $\cF = \cF_{\rm can}$ since the corresponding results in the case $\cF = \cF_{\rm ur}$ follow by exactly similar arguments.

At the outset we note that our hypotheses (H$_1$), (H$_2$), (H$_3$) and (H$_4$) are together equivalent to \cite[Hyp.~4.7]{bss}.

In addition, Theorem~\ref{free} implies that, under (H$_1$), (H$_2$), (H$_3$), (H$_4$) and (H$_5^{\rm c}$), the key condition of \cite[Hyp.~4.2]{bss} is satisfied by the structure $\cF_{\rm can}$.

Next we note that our Hypothesis \ref{hyp K} on $\cK$ coincides with \cite[Hyp.~6.7]{bss}, that our condition (H$_0$) is equivalent to \cite[Hyp.~6.11]{bss} and that Lemma \ref{torsionfree} combines with \cite[Rem. 6.2]{bss} to imply that \cite[Hyp. 6.1]{bss} is satisfied whenever (H$_1$) and (H$_3$) are both satisfied.

Finally we note the formula for $\chi_{\rm can}(\cT)$ given at the end of Theorem \ref{free} combines with hypothesis (H$_5$) and the fact $F/K$ has odd degree to imply that the $\cR$-module $\bigoplus_{v \in S_\infty(K)} H^0(K_v, \cT^\ast(1))$ is free of rank $\chi_{\rm can}(\cT)$ (as required by \cite[Cor. 6.17(a)]{bss}) and that the hypothesis (H$_5$) itself coincides with the condition stated in \cite[Cor. 6.17(b)]{bss}.

At this stage we have checked that all of the conditions necessary to apply the results of \cite[Th.~4.12, Th.~5.25, 
and Cor.~6.18(i) and (ii)]{bss}, respectively of \cite[Cor. 6.17 and Cor. 6.18(iii)]{bss}, to the Selmer structure $\cF_{\rm can}$ are satisfied under the hypotheses of Theorem \ref{main}(i), (ii) and (iii), respectively of Theorem \ref{main}(iv).

In this way one finds that, under the stated conditions, claim (i) of Theorem \ref{main} follows from the proof of \cite[Th.~6.12]{bss}, claim (ii) (in the case $\cF = \cF_{\rm can})$ from \cite[Th.~5.25(i)]{bss},
 claim (iii)(a) from \cite[Th.~4.12]{bss}, claim (iii)(b) from \cite[Th.~5.25(iii)]{bss} and claim (iii)(c) from claims~(ii) and (iii)(a) and the fact that $I(c_F) = I_{0}(\cR^{-1}_{r}(\cD_{r}(c)))$. In addition, the proof of \cite[Cor. 6.17]{bss} shows that, under the hypotheses of claim (iv), the map $\cD_{r}$ is surjective and the rest of claim (iv) then follows easily from the final assertions of claim (iii)(a) and (iii)(c).

%

This completes the proof of Theorem \ref{main}.

\section{Bloch-Kato Euler systems and Tamagawa number conjectures} \label{section 3} In this section, we use the special values of motivic $L$-functions to give a precise conjectural construction of a natural family of higher rank Euler systems.

We show that this construction simultaneously incorporates the known constructions of Euler systems from Rubin-Stark elements \cite{rubinstark}, from `generalized Stark elements' \cite{bks2-2} and from Kato's `zeta elements' for modular forms \cite{kato}.

We then show that this family of Euler systems offers a new approach to the study of refined versions of the Tamagawa number conjecture of Bloch and Kato.
\subsection{Statement of the eTNC}\label{etnc sec}

We first quickly review the formulation of the equivariant Tamagawa number conjecture in the case of commutative coefficients.

For convenience we shall follow the presentation of Flach and the first named author in \cite{BFetnc} rather than the approach of Kato in \cite{kato-kodai, katolecture} or the original `non-equivariant' approaches of Bloch and Kato \cite{bk} and Fontaine and Perrin-Riou \cite{fpr}.

\subsubsection{General notation}Let $K$ be a number field. Let $M$ be a (pure) motive over $K$ with coefficients in a finite dimensional commutative semisimple $\QQ$-algebra $B$. Fix an odd prime number $p$, and set $A:=\QQ_p\otimes_\QQ B$. We also fix a $\ZZ_p$-order $\cA$ in $A$. (Later we will assume that $\cA$ is Gorenstein.) Let $V:=V_p(M)$ be the $p$-adic \'etale realization of $M$, and $T$ be a $G_K$-stable lattice of $V$, i.e., $T$ is a free $\cA$-module of finite rank, stable under the action of $G_K$, and satisfies $V=\QQ_p \otimes_{\ZZ_p} T$. We set $V^\ast(1):=\Hom_A(V, A(1))$ and $T^\ast(1):=\Hom_\cA(T,\cA(1))$. For an $\cA$-module $X$, we write $\cA^\ast$ for the $\cA$-dual $\Hom_\cA(X,\cA)$.
When $X$ is an $A$-module, we use the same notation $X^\ast$ to indicate the $A$-dual $\Hom_A(X,A)$, but this abuse of notation would not make any danger of confusion.

We fix a finite set of places $S$ of $K$ containing $S_\infty(K) \cup S_p(K) \cup S_{\rm ram}(T)$ and for each finite extension $K'/K$, we use the following notations:
\begin{itemize}
\item $S(K'):=S \cup S_{\rm ram}(K'/K)$;
\item $U_{K'} := \cO_{K',S(K')}$;
\item $\cG_{K'}:=\Gal(K'/K)$;
\item $Y_{K'}(T):=\bigoplus_{w \in S_\infty(K')}H^0(K'_w,T^\ast(1))$;
\item $C_{K',S}(T):=R\Hom_\cA(R\Gamma_c(U_{K'}, T^\ast(1)), \cA[-2])$, where $R\Gamma_c(U_{K'}, T^\ast(1))$ denotes compactly supported \'etale cohomology, as discussed in \cite[\S 1.4]{sbA};
\item $\QQ(a)_{K'}$ denotes, for each integer $a$, the Tate motive $h^0({\rm Spec}(K'))(a)$;
\item $M_{K'} := \QQ(0)_{K'}\otimes_{\QQ(0)_K}M$ and $M_{K'}^\vee$ is the Kummer dual of $M_{K'}$ (so that $V_p(M_{K'}^\vee)$ is isomorphic to $V_p(M_{K'})^\ast(1)$).
\end{itemize}

For each finite abelian group $\Gamma$ we set $\widehat{\Gamma} := \Hom(\Gamma,\CC^\times)$ and each for $\chi$ in $\widehat \Gamma$ we write $e_\chi$ for the idempotent $(\# \Gamma)^{-1}\cdot\sum_{\gamma\in \Gamma}\chi(\gamma)\gamma^{-1}$ of $\CC[\Gamma]$.

In the sequel we also set  $B_\CC:=\CC \otimes_\QQ B $ and $A_{\CC_p}:=\CC_p \otimes_{\QQ_p}A$. 

\subsubsection{Statement of the conjecture} For each prime $\fq$ of $K$ we write ${\N} \fq$ for the cardinality of its residue field, $I_{K_\fq}$ for the inertia subgroup of $G_{K_\fq}$ and ${\rm Fr}_\fq $ for the Frobenius element in $G_{K_\fq}/I_{K_\fq}$. For each such $\fq$ we set

$$P_\fq(T;x):= {\det}_A(1-{\rm Fr}_\fq^{-1} x\mid V^\ast(1)^{I_{K_\fq}}) \in A[x].$$
It is conjectured that this polynomial belongs to $B[x]$ and we assume this to be true.

The $B_\CC[\cG_{K'}]$-valued $L$-function of $M_{K'}^\vee$ is then defined via the Euler product
\[ L(M_{K'/K}^\vee,s):=\prod_{\fq }P_\fq(T;{\N}\fq^{-s} {\rm Fr}_\fq^{-1})^{-1} \in B_{\CC}[\cG_{K'}].\]
%
%

This function is conjectured to have an analytic continuation to $s=0$ and we assume this to be true.

We then write $L^\ast(M_{K'/K}^\vee,0)$ for the element of $B_{\CC}[\cG_{K'}]^\times$ that is obtained as the leading term of
$L(M_{K'/K}^\vee,s)$ at $s=0$. 

Next we recall that the construction of \cite[\S3.4]{BFetnc} gives a canonical isomorphism of (graded) $A_{\CC_p}[\cG_{K'}]$-modules of the form
\[ \vartheta_{M,K',S}^{\rm BK}: \CC_p\otimes_{\ZZ_p} {\rm det}_{\cA[\cG_{K'}]}(C_{K',S}(T)) \simeq (A_{\CC_p}[\cG_{K'}],0). \]


The construction of this map is in general dependent upon conjectures concerning motivic cohomology groups. However, in the settings that are of most interest to us (and are discussed in Remark \ref{ex L1} below), this construction is both unconditional and can be explicitly described in terms of classical regulators, height pairings and period maps.

For each idempotent $\varepsilon$ in $A[\cG_{K'}]$ and each $\ZZ_p$-order $\cR$ in $A[\cG_{K'}]\varepsilon$, we can now recall the formulation of the equivariant Tamagawa number conjecture for the pair $(M_{K'}^\vee,\cR)$.

\begin{conjecture}\label{etnc} In $(A_{\CC_p}[\cG_{K'}]\varepsilon,0)$ there is an equality of graded invertible $\cR$-modules
\[ \cR \cdot \vartheta_{M, K',S}^{\rm BK}({\det}_{\cA[\cG_{K'}]}(C_{K',S}(T)))= (\cR \cdot L^\ast(M_{K'/K}^\vee,0),0).\]
\end{conjecture}

In the sequel we shall refer to this conjecture as `${\rm TNC}(M_{K'}^\vee,\cR)$'.

\begin{remark} \label{remark etnc}
It is clear that if $\cR'$ is any $\ZZ_p$-order in $A[\cG_{K'}]\varepsilon$ that contains $\cR$, then ${\rm TNC}(M_{K'}^\vee,\cR)$ implies ${\rm TNC}(M_{K'}^\vee,\cR')$. To discuss other functorial properties of the conjecture it is convenient to use the graded invertible $\cA[\cG_{K'}]$-submodule of $(A_{\CC_p}[\cG_{K'}],0)$ obtained by setting
 \[ \Xi(\cA,T,K'/K) := L^\ast(M_{K'/K}^\vee,0)^{-1}\cdot \vartheta_{M, K',S}^{\rm BK}({\det}_{\cA[\cG_{K'}]}(C_{K',S}(T))).\]
Then ${\rm TNC}(M_{K'}^\vee,\cR)$ is valid if and only if one has $\cR\cdot\Xi(\cA,T,K'/K) = (\cR,0)$, whilst the Deligne-Beilinson Conjecture (in the form of \cite[Conj. 4(iii)]{BFetnc}) asserts that
\begin{equation}\label{d-b conj} \Xi(\cA,T,K'/K)\subseteq (A[\cG_{K'}],0).\end{equation}%
In addition, the argument of \cite[\S3.4]{BFetnc} shows that $\Xi(\cA,T,K'/K)$ (and hence also the validity of ${\rm TNC}(M_{K'}^\vee,\cR)$) is, as the notation suggests, independent of $S$, whilst \cite[Th. 4.1]{BFetnc} implies that for each intermediate field $E$ of $K'/K$, with $\Gamma := \Gal(K'/E)$, the following statements are true:
\begin{itemize}
\item[(i)] the projection map $A_{\CC_p}[\cG_{K'}] \to A_{\CC_p}[\cG_{E}]$ sends $\Xi(\cA,T,K'/K)$ to $\Xi(\cA,T,E/K)$;
\item[(ii)] there is a canonical isomorphism ${\det}_{\cA[\Gamma]}(\Xi(\cA,T,K'/K)) \simeq \Xi(\cA,T,K'/E)$.
\end{itemize}
\end{remark}

\begin{remark}\label{p tnc} The Tamagawa number conjecture ${\rm TNC}(M^\vee)$ for $M^\vee$, as formulated by Bloch and Kato \cite{bk} and then extended and reinterpreted by Fontaine and Perrin-Riou \cite{fpr}, is a conjectural formula, up to a sign, for the complex number $L^\ast(M_{K/K}^\vee,0)$. We shall say that the `$p$-part' of this conjecture is valid if the formula is valid up to a unit of the $p$-localisation of $\ZZ$ and one can check that this implies the conjecture ${\rm TNC}(M^\vee,\ZZ_p)$ obtained above in the case $K'=K,B=\QQ, A=\QQ_p, \varepsilon = 1$ and $\cR = \ZZ_p$.
\end{remark}

\begin{remark}\label{ex L1} Conjecture \ref{etnc} specializes to recover several well-known conjectures.

\noindent{}(i) If $M=\QQ(a)_K$ for an integer $a$, then $M_{K'}^\vee= \QQ(1-a)_{K'}$ and one has
\[ L(M_{K'/K}^\vee,s) = \sum_{\chi\in \widehat{\cG}_{K'}}L(\chi^{-1},s + 1-a)e_\chi,\]
where $L(\chi^{-1},s)$ is the Artin $L$-series of $\chi^{-1}$. Using this fact it can be shown that the $p$-part of ${\rm TNC}(\QQ(0)_K)$ follows from the analytic class number formula for $K$.

\noindent{}(ii) Let $E$ be an elliptic curve over $K$ and set $M :=h^1(E)(1)$. Then $M_{K'}^\vee = h^1(E_{/K'})(1)$, where $E_{/K'}$ denotes $E$ regarded as an elliptic curve defined over $K'$, and one has
\[ L(M_{K'/K}^\vee,s)= \sum_{\chi \in \widehat \cG_{K'}} L(E,\chi^{-1},s+1)e_\chi\]
where $L(E,\chi^{-1},s)$ is the Hasse-Weil-Artin $L$-series attached to $E$ and $\chi^{-1}$. Further, if one assumes both that the Tate-Shafarevic group of $E$ over $K$ is finite and that the order of vanishing at $s=1$ of $L(E,s)$ is equal to the rank of $E(K)$, then ${\rm TNC}(M^\vee,\Z_p)$ recovers the `$p$-part' of the explicit formula for the leading term at $s=1$ of $L(E,s)$ that is predicted by the Birch and Swinnerton-Dyer Conjecture. (For details of this deduction see, for example, the argument of Kings in \cite[Lecture 3]{Kings}.)\end{remark}

\subsection{The Bloch-Kato Euler system}\label{bloch kato}

We henceforth assume that all archimedean places of $K$ split completely in $K'$ (we will later restrict to the case that $K'/K$ has odd prime-power degree in which case this condition is automatic).

We will also assume now that $\cA$ is Gorenstein (although in some parts this assumption is not strictly necessary).

\subsubsection{Bloch-Kato elements} In this section we associate a canonical notion of `Bloch-Kato element' to the pair $(T,K')$.

To do this we start by recalling the result of \cite[Prop. 2.21]{sbA}.

\begin{proposition} \label{prop C} The following claims are valid for the complex $C_{K'} = C_{K',S}(T)$.
\begin{itemize}
\item[(i)] $C_{K'}$ is a perfect complex of $\cA[\cG_{K'}]$-modules, acyclic outside degrees $-1, 0, 1$, and its Euler characteristic in $K_0(\cA[\cG_{K'}])$ vanishes.
\item[(ii)] There are canonical isomorphisms of $\cA[\cG_{K'}]$-modules
\[ H^{-1}(C_{K'}) \simeq H^0(K',T)\,\,\text{ and }\,\, H^0(C_{K'}) \simeq H^1(U_{K'},T)\]
and a canonical split exact sequence of $\cA[\cG_{K'}]$-modules
\[ 0 \to H^2(U_{K'}, T) \to H^1(C_{K'}) \to Y_{K'}(T)^\ast \to 0.\]
\item[(iii)] If $H^1(U_{K'},T)$ is $\ZZ_p$-free and $H^0(K',T)=0$, then $C_{K'}$ is acyclic outside degrees $0$ and $1$ and constitutes an `admissible complex' of $\cA[\cG_{K'}]$-modules in the sense of \cite[Def. 2.19]{sbA}.
\end{itemize}
\end{proposition}

Motivated by the exact sequence in Proposition \ref{prop C}(ii) we shall assume the following hypothesis in the sequel (which coincides with {\cite[Hyp. 2.11]{sbA}}).
\begin{hypothesis}\label{hyp Y} $Y_K(T)$ is a free $\cA$-module.
\end{hypothesis}

We then set $r = r_{T} :={\rm rank}_\cA(Y_K(T)).$

\begin{remark}\label{induced basis} Since all archimedean places of $K$ are assumed to split in $K'$, a choice of place of $K'$ above each archimedean place of $K$ induces an isomorphism of the $\cA[\cG_{K'}]$-module $Y_{K'}(T)$ with $Y_K(T)\otimes_\cA \cA[\cG_{K'}]$ and so this module is also free of rank $r$.\end{remark}

The following idempotent will frequently be used in the rest of this article.

\begin{definition}\label{def adm}
We define the `idempotent of admissibility' 
$$\varepsilon = \varepsilon_{T,K'} \in A[\cG_{K'}]$$
for $T$ and $K'$ to be the sum of primitive idempotents of $A[\cG_{K'}]$ that annihilate the space 
$$H^2(U_{K'},V)\oplus H^0(K',V).$$
(Note that, in general, $\varepsilon$ depends on $S$.)
\end{definition}

\begin{remark} If $T=\ZZ_p(a)$ for an integer $a$, then the idempotent $\varepsilon_{T,K'}$ is denoted by 
$\varepsilon_{1-a}$ in \cite[\S 2.2]{bks2-2} and is explicitly described in Remark 2.6 of loc. cit. When $T$ is the $p$-adic Tate module of an elliptic curve over $\QQ$, an explicit description of $\varepsilon_{T,K'}$ is given in Lemma \ref{elliptic coh}(iii) below.
\end{remark}

In the following, we shall often abbreviate $\CC_p \otimes_{\ZZ_p} - $ to $\CC_p \cdot -$.

The descriptions in Proposition \ref{prop C}(ii) imply that the complex $\varepsilon(\CC_p\cdot C_{K',S}(T))$ is acyclic outside degrees zero and one and that its cohomology in these respective degrees is equal to
 $\varepsilon(\CC_p\cdot H^1(U_{K'},T))$ and $\varepsilon(\CC_p\cdot Y_{K'}(T)^\ast)$.

 Upon combining the final assertion of Proposition \ref{prop C}(i) with Hypothesis \ref{hyp Y} we can then also deduce that both of the latter modules are free of rank $r$ over $A_{\CC_p}[\cG_{K'}]\varepsilon$.

In particular, if we fix an ordered $\cA$-basis $\underline{b}$ of $Y_K(T)$, then we obtain a composite isomorphism of $A_{\CC_p}[\cG_{K'}]$-modules
\begin{align}\label{first comp} \varepsilon(\CC_p\cdot {\bigwedge}^r_{\cA[\cG_{K'}]}H^1(U_{K'},T)) &= {\rm det}_{A_{\varepsilon}}(\varepsilon(\CC_p\cdot H^0(C_{K'})))\\
 &\simeq {\rm det}_{A_{\varepsilon}}(\varepsilon(\CC_p\cdot C_{K'}))\otimes_{A_{\varepsilon}} {\rm det}_{A_\varepsilon}(\varepsilon(\CC_p\cdot H^1(C_{K'})))\notag\\
 &=  {\rm det}_{A_{\varepsilon}}(\varepsilon(\CC_p\cdot C_{K'}))\otimes_{A_\varepsilon} \varepsilon(\CC_p\cdot {\bigwedge}^r_{\mathcal{A}[\cG_{K'}]} Y_{K'}(T)^\ast)\notag\\
 &\simeq {\rm det}_{A_{\varepsilon}}(\varepsilon(\CC_p\cdot C_{K'}))\notag\end{align}
in which we have abbreviated $A_{\CC_p}[\cG_{K'}]\varepsilon$ to $A_{\varepsilon}$ and $C_{K',S}(T)$ to $C_{K'}$, the first isomorphism is induced by the canonical `passage to cohomology' map for $\varepsilon(\CC_p\cdot C_{K'})$ and the second by regarding $\underline{b}$ as an ordered $\cA[\cG_{K'}]$-basis of $Y_{K'}(T)$ via the identification in Remark \ref{induced basis} and using the corresponding (ordered) dual basis to identify
 ${\bigwedge}_{\cA[\cG_{K'}]}^r Y_{K'}(T)^\ast$ with $\cA[\cG_{K'}]$.

We can then define a composite isomorphism of $A_{\CC_p}[\cG_{K'}]$-modules
\[ \lambda^{\rm BK}_{V,\underline{b},S,K'} : \varepsilon(\CC_p\cdot {\bigwedge}^r_{\cA[\cG_{K'}]}H^1(U_{K'},T))\simeq {\rm det}_{A_{\varepsilon}}(\varepsilon(\CC_p\cdot C_{K'}))\simeq A_{\CC_p}[\cG_{K'}]\varepsilon, \]
in which the first map is the composition (\ref{first comp}) and the second is $\vartheta^{\rm BK}_{M,K',S}$.

This allows us to make the following key definition.

\begin{definition} \label{def BK} The {\it Bloch-Kato element} $\eta_{V,\underline{b},S,K'}^{\rm BK}$ associated to the data $V, \underline{b}, S$ and $K'$ is the unique element of $\varepsilon_{T,K'}\left( \CC_p \otimes_{\ZZ_p}  {\bigwedge}_{\cA[\cG_{K'}]}^r H^1(U_{K'},T) \right)$ that satisfies
$$ \lambda^{\rm BK}_{V,\underline{b},S,K'}(\eta_{V,\underline{b},S,K'}^{\rm BK}) = \varepsilon_{T,K'}\cdot L^\ast(M_{K'/K}^\vee,0).$$
\end{definition}

\begin{remark} In the sequel the data $V$ and $S$ is usually clear from context and so, to ease notation, we shall often abbreviate $\lambda^{\rm BK}_{V,\underline{b},S,K'}$ and $\eta^{\rm BK}_{V,\underline{b},S,K'}$ to
 $\lambda^{\rm BK}_{\underline{b},K'}$ and $\eta^{\rm BK}_{\underline{b},K'}$. \end{remark}

Upon appropriate specialization these Bloch-Kato elements recover important constructions in the literature.

\begin{example}\label{bkrs1} Let $E$ be an elliptic curve over $K = \QQ$. Then, with $M$ the motive $h^1(E)(1)$ and $T$ the $p$-adic Tate module of $E$, one can make a choice of $\ZZ_p$-basis $\underline{b}$ of $Y_{\QQ}(T)$ for which the Bloch-Kato element $\eta^{\rm BK}_{\underline{b},K'}$ can be explicitly related to a `zeta element' of the form constructed by Kato in \cite{kato}. We give details of this case in \S\ref{section kato}.
\end{example}

\begin{example}\label{bkrs2} Let $L$ be a finite abelian extension of $K$ and set $G := \cG_L$. Fix an integer $a$ and a primitive idempotent $\varepsilon$ of $\ZZ_p[G]$, set $\cA := \ZZ_p[G]\varepsilon$ and write $T$ for the $G_K$-representation given by the $\cA$-module $\cA\otimes_{\ZZ_p}\ZZ_p(a)$ upon which $\cA$ acts via left multiplication and $G_K$ via the rule $\sigma(x\otimes y) = x \sigma_\cA^{-1}\otimes \sigma(y)$ where $\sigma_\cA$ denotes the image of $\sigma$ in $\cA$. Then the $\cA$-module $Y_{K}(T) = \cA\otimes_{\ZZ_p[G]}Y_{L}(\ZZ_p(a))$ is free and we can fix $\underline{b}$ to be the (ordered) basis given by the set $\{ \varepsilon \cdot w(j) \mid w \in W_{-j}^\varepsilon\}$ specified in \cite[Lem. 2.1]{bks2-2} (with $j=1-a$). The linear dual $Y_{K}(T)^\ast$ of $Y_K(T)$ identifies with $\cA\otimes_{\ZZ_p[G]}Y_{L}(a-1)$, where the module $Y_{L}(a-1)$ is as defined in \cite[\S2.1]{bks2-2}, whilst \cite[Rem. 2.6]{bks2-2} implies that $\varepsilon_{T,K}$ agrees with the idempotent $e:= \varepsilon_{1-a}$ defined in loc. cit. One can also check that
\begin{equation}\label{per reg} (\prod_{\fq} P_\fq(1-a)e)\cdot\lambda^{\rm BK}_{\underline{b},K} = \lambda_{1-a},\end{equation}
where $\fq$ runs over all primes in $S\setminus S_\infty(K)$, $P_\fq(s)$ denotes the `equivariant' Euler factor
$$P_\fq(s):=\sum_{\chi \in \widehat G}(1-{\N}\fq^{-s}\chi^{-1}(\fq))e_\chi,$$
and $\lambda_{1-a}$ is the `period-regulator' isomorphism defined in \cite[\S2.2]{bks2-2}. This equality follows from the general observation of \cite[Lem. 2.23]{sbA} and the following facts: the definition of $e$ ensures that each factor $P_\fq(1-a)e$ is non-zero; the map $\lambda^{\rm BK}_{\underline{b},K}$ uses the morphism $\vartheta_p$ defined in \cite[\S3.4]{BFetnc} and so for each $\fq$ in $S\setminus S_\infty(K)$ applies the morphism of \cite[(24)]{BFetnc} to the complex $R\Gamma_f(K_\fq,e(\QQ_p\otimes_{\ZZ_p}T))$ in \cite[(19)]{BFetnc}); the definition of $\lambda_{1-a}$ implicitly uses the acyclicity of the complexes $R\Gamma_f(K_\fq,e(\QQ_p\otimes_{\ZZ_p}T))$ for each $\fq$ in $S\setminus S_\infty(K)$.

The above observations combine to imply that $\eta_{\underline{b},K}^{\rm BK}$ coincides with the `generalized Stark element' $\eta_{L/K,S,\emptyset}^\varepsilon(1-a)$ defined in \cite[Def. 2.7]{bks2-2}.

This shows, in particular, that if $a=1$, then $\eta_{\underline{b},K}^{\rm BK}$ is the image under multiplication by $\varepsilon$ of the Rubin-Stark element denoted by $\epsilon_{L/K,S,\emptyset}^{V}$ in \cite[\S 5.1]{bks1}, where $V$ is the subset of $S_\infty(K)$ obtained by restricting places in the subset $W_0^\varepsilon$ of $S_\infty(L)$.
\end{example}



\subsubsection{The image of Bloch-Kato elements}
 We recall that $r$ denotes the rank of the (free) $\cA$-module $Y_K(T)$. We then define the `image' of $\eta_{\underline{b},K'}^{\rm BK}$ by
$$\im(\eta_{\underline{b},K'}^{\rm BK}):=\left\{ \Phi(\eta_{\underline{b},K'}^{\rm BK}) \ \middle| \ \Phi \in {\bigwedge}_{\cA[\cG_{K'}]}^r H^1(U_{K'}, T)^\ast\right\}.$$

This is an $\cA[\cG_{K'}]$-submodule of $A_{\CC_p}[\cG_{K'}]$ and is sometimes also denoted by $I(\eta_{\underline{b},K'}^{\rm BK})$. In Conjecture \ref{integrality} below we will predict that $\eta_{\underline{b},K'}^{\rm BK}$ belongs to the exterior power bidual ${\bigcap}_{\cA[\cG_{K'}]}^r H^1(U_{K'},T)$  and, if this is the case, then $\im(\eta_{\underline{b},K'}^{\rm BK})$ is an ideal of $\cA[\cG_{K'}]$ and coincides with the image of $\eta_{\underline{b},K'}^{\rm BK}$ as an element of ${\bigcap}_{\cA[\cG_{K'}]}^r H^1(U_{K'},T)$.

The next result computes $\im(\eta_{\underline{b},K'}^{\rm BK})$ and will later play a key role. In this result we use the graded sublattice $\Xi(\cA,T,K'/K)$ of $(A_{\CC_p}[\cG_{K'}],0)$ that is defined in Remark \ref{remark etnc}. We also fix an abelian pro-$p$ extension $\cK/K$ and write $\Omega(\cK/K)$ for the set of finite extensions of $K$ in $\cK$. 



%
%

\begin{theorem} \label{prop Xi} We assume that the representation $T$ is such that for every $K'$ in $\Omega(\cK/K)$ the following two conditions are satisfied:
\begin{itemize}
\item[(a)] the module $H^1(U_{K'},T)$ is torsion-free;
\item[(b)] the module $H^0(K',T)$ vanishes.
\end{itemize}
Then for any choice of ordered $\cA$-basis $\underline{b}$ of $Y_K(T)$, and any field $K'$ in $\Omega(\cK/K)$ there is an equality in $(A_{\CC_p}[\cG_{K'}],0)$ of the form
\[ (\im(\eta_{\underline{b},K'}^{\rm BK}),0) = {\rm Fitt}_{\cA[\cG_{K'}]}^0(H^2(U_{K'},T))\cdot\Xi(\cA,T,K'/K)^{-1}.\]
%
\end{theorem}

\begin{proof}
Under the assumption that the conditions (a) and (b) are satisfied for all $K'$, the first and third authors have constructed (unconditionally) a cyclic $\cA[[\Gal(\cK/K)]]$-submodule $\mathcal{E}^{\rm b}(T,\cK)$ of ${\rm ES}_r(T,\cK)$ comprising `basic' Euler systems (see \cite[Def. 2.19]{sbA}). 

It is also shown that any generator $\eta_T^{\rm b}$ of $\mathcal{E}^{\rm b}(T,\cK)$ has the following properties at each $K'$ in $\Omega(\cK/K)$:
\begin{equation*}\label{basis props} \begin{cases}
& \eta_{T,K'}^{\rm b}\in \left({\bigcap}_{\cA[\cG_{K'}]}^r H^1(U_{K'},T) \right)[1-\varepsilon_{K'}],\\
& \cA[\cG_{K'}]\cdot\lambda^{\rm BK}_{\underline{b},K'}(\eta_{T,K'}^{\rm b}) = \vartheta_{M,K',S}^{\rm BK}({\det}_{\cA[\cG_{K'}]}(C_{K',S}(T)))[1-\varepsilon_{K'}],\\
& \im(\eta_{T,K'}^{\rm b}) = {\rm Fitt}_{\cA[\cG_{K'}]}^0(H^2(U_{K'},T)).\end{cases}\end{equation*} 
Here $\varepsilon_{K'}=\varepsilon_{T,K'}$ is the idempotent defined in Definition \ref{def adm}, and for an idempotent $e$ of a torsion-free $\cA[\cG_{K'}]$-module $X$ we write $X[e]$ for the submodule of $X$ comprising elements that are annihilated by $e$ in $\QQ_p\otimes_{\ZZ_p} X$. We note that the first two of the above properties  follow directly from the construction of basic Euler systems (via \cite[Th. 2.17 and Def. 2.18]{sbA}) and the third is proved in \cite[Th. 2.26(ii)]{sbA}.

In particular, since Definition \ref{def BK} implies that $\lambda_{\underline{b},K'}^{\rm BK}(\eta_{\underline{b},K'}^{\rm BK})$ is equal to $\varepsilon_{K'} L^\ast(M_{K'/K}^\vee,0)$, the injectivity of
 $\lambda_{\underline{b},K'}^{\rm BK}$ combines with the first property displayed above to imply that $\lambda_{\underline{b},K'}^{\rm BK}(\eta_{T,K'}^{\rm b})\cdot \eta_{\underline{b},K'}^{\rm BK}$ is equal to $L^\ast(M_{K'/K}^\vee,0)\cdot \eta_{T,K'}^{\rm b}.$

The claimed result then follows by combining this fact with the second and third equalities displayed above.\end{proof}

\subsubsection{The Euler system} The Deligne-Beilinson Conjecture (in the form of (\ref{d-b conj}))  implies that each element $\eta^{\rm BK}_{\underline{b},K'}$ belongs to $\QQ_p\otimes_{\ZZ_p}{\bigwedge}_{\cA[\cG_{K'}]}^rH^1(U_{K'},T)$. We now use the notion of exterior power bidual to formulate a precise refinement of this conjecture.


\begin{conjecture} \label{integrality} The element $\eta_{\underline{b},K'}^{\rm BK}$ belongs to ${\bigcap}_{\cA[\cG_{K'}]}^r H^1(U_{K'},T)$.
\end{conjecture}

\begin{remark} \label{remark etnc rs} The validity of this conjecture is clearly independent of the choice of $\cA$-basis $\underline{b}$ of $Y_K(T)$. In addition, the observations in Example \ref{bkrs2} show that it extends the `Generalized Rubin-Stark Conjecture' that is formulated via the inclusion of \cite[Conj. 3.5]{bks2-2}. For the representation $T = \ZZ_p(1)$ it therefore recovers the `$p$-part of the Rubin-Stark Conjecture' (as labelled `${\rm RS}(K'/K,S,\emptyset,V)_p$' in \cite[Conj. 2.1]{bks2}) formulated by Rubin in \cite[Conj. B$'$]{rubinstark}. \end{remark}

\begin{remark} If  $T$ satisfies certain mild hypotheses, then one can use the argument of \cite[Th. 5.12]{bks1}  to show that Conjecture \ref{integrality} is a consequence of ${\rm TNC}(M_{K'}^\vee,\cA[\cG_{K'}])$. \end{remark}


We also fix a place of $\cK$ lying over each archimedean place of $K$ and for every field $K'$ in $\Omega(\cK/K)$ we use the restriction of these places to $K'$ to normalize the isomorphism $Y_{K'}(T) \simeq Y_K(T)\otimes_{\cA}\cA[\cG_{K'}]$ that (occurs in Remark \ref{induced basis} and) is used when defining $\eta_{\underline{b},K'}^{\rm BK}$.

We recall that ${\rm ES}_r(T,\cK)$ denotes the $\cA[[\Gal(\cK/K)]]$-module of Euler systems of rank $r$ for $(T,\cK)$

The following result explains the significance of Conjecture \ref{integrality}.

\begin{proposition}\label{cor BK}
If Conjecture \ref{integrality} is valid for all fields $K'$ in $\Omega(\cK/K)$, then the set
$$\eta_{\underline{b},\cK}^{\rm BK}=(\eta_{\underline{b},K'}^{\rm BK})_{K'} \in \prod_{K' \in \Omega(\cK/K)}{\bigcap}_{\cA[\cG_{K'}]}^r H^1(U_{K'},T)$$
belongs to ${\rm ES}_r(T,\cK)$. 
\end{proposition}

\begin{proof} We fix fields $F$ and $F'$ in $\Omega(\cK/K)$ with $F\subseteq F'$ and set
\[ P_{F'/F}:= \prod_{\fq \in S(F')\setminus S(F)} P_\fq (T;{\rm Fr}_\fq^{-1})  \in A[\cG_F].\]
(Note that any prime $\fq$ in $S(F')\setminus S(F)=S_{\rm ram}(F')\setminus S_{\rm ram}(F)$ is unramified in $F$ so that ${\rm Fr}_\fq$ can be regarded as an element of $\cG_{F}$.)

We must show that the natural corestriction map
$$ {\rm Cor}^r_{F'/F}: \CC_p \otimes_{\ZZ_p} {\bigwedge}_{\cA[\cG_{F'}]}^r H^1(\cO_{F',S(F')},T) \to \CC_p \otimes_{\ZZ_p} {\bigwedge}_{\cA[\cG_{F}]}^r H^1(\cO_{F,S(F')},T)$$
sends $\eta_{\underline{b},F'}^{\rm BK}$ to $P_{F'/F}\cdot\eta_{\underline{b},F}^{\rm BK}$.

To do this we set $\varepsilon:=\varepsilon_F, \varepsilon':=\varepsilon_{F'}, G:=\cG_F$, $G':=\cG_{F'}$, $A_\varepsilon := A_{\CC_p}[G]\varepsilon$ and $A_{\varepsilon'} := A_{\CC_p}[G']\varepsilon'$. We also set $W_{F} := H^1(U_{F},V), W_{F'}:= H^1(U_{F'},V)$ and $W_F' := H^1(\cO_{F,S(F')},V)$. We also write $\pi$ for the natural projection map $A_{\CC_p}[G'] \to A_{\CC_p}[G]$.

As a first step, we record the following result.


\begin{lemma}\label{lemma e} One has $\pi(\varepsilon')\cdot W_{F} = \pi(\varepsilon')\cdot W_F'$. In addition, for any primitive idempotent $e$ of $A[G]$ one has $e \cdot \pi( \varepsilon') \neq 0$ if and only if both
$e \cdot P_{F'/F}\neq 0$ and $e \cdot \varepsilon\neq 0.$
\end{lemma}

\begin{proof} The proof is identical to that of \cite[Lem. 2.22]{sbA} and so we omit details.
%
\end{proof}

This result implies, firstly, that $\pi$ induces a ring homomorphism $A_{\varepsilon'} \to A_{\varepsilon}$ and, secondly, that 
 $\pi(\varepsilon')(\CC_p\otimes_{\QQ_p}W'_{F})$ is contained in $\varepsilon(\CC_p\otimes_{\QQ_p}W_{F}).$

We can therefore consider the following diagram
\begin{equation}\label{key diagram}
\xymatrix{
{\bigwedge}_{A_{\varepsilon'}}^r \varepsilon'(\CC_p\otimes_{\QQ_p}W_{F'})\ar[r]^{\delta_{F'}} \ar[d]_{{\rm Cor}_{F'/F}^r} &{\rm det}_{A_{\varepsilon'}}(\varepsilon'(\CC_p\otimes_{\ZZ_p}C_{F'})) \ar[r]^{\quad\quad \quad \quad\vartheta_{F'}} \ar[d]_{\varrho_{F'/F}} & A_{\varepsilon'}\ar[d]^{\pi}\\
{\bigwedge}_{A_{\varepsilon}}^r \pi(\varepsilon')(\CC_p\otimes_{\QQ_p}W'_{F})\ar[r]^{\delta'_F} \ar@{^{(}->}[d]_{\iota} &{\rm det}_{A_{\varepsilon}}(\pi(\varepsilon')(\CC_p\otimes_{\ZZ_p}C'_{F})) \ar[r]^{\quad\quad \quad \quad\vartheta'_{F}} \ar@{^{(}->}[d]_{\iota'} & A_{\varepsilon}\ar[d]^{\times P_{F'/F}}\\
{\bigwedge}_{A_{\varepsilon}}^r \varepsilon(\CC_p\otimes_{\QQ_p}W_{F})\ar[r]^{\delta_F} &{\rm det}_{A_{\varepsilon}}(\varepsilon(\CC_p\otimes_{\ZZ_p}C_{F})) \ar[r]^{\quad\quad \quad \quad\vartheta_{F}}  & A_{\varepsilon}.}
\end{equation}
In this diagram we use the following abbreviations:
$$\begin{cases}
C_F := R\Hom_\cA(R\Gamma_c(\cO_{F,S(F)}, T^\ast(1)), \cA[-2]) (=C_{F,S}(T)), \\
C_{F'} := R\Hom_\cA(R\Gamma_c(\cO_{F',S(F')}, T^\ast(1)), \cA[-2]) (=C_{F',S}(T)), \\
C_F' := R\Hom_\cA(R\Gamma_c(\cO_{F,S(F')}, T^\ast(1)), \cA[-2]);
\end{cases} $$
$\delta_{F}$ and $\delta_{F'}$ are the respective cases of the isomorphism (\ref{first comp}) and $\delta_F'$ is defined similarly; $\varrho_{F'/F}$ is the projection map that is induced by the canonical descent isomorphism $\cA[G]\otimes^\mathbb{L}_{\cA[G']}C_{F'} \simeq C_F'$; $\vartheta_{F}, \vartheta_{F'}$ and $\vartheta_F'$ are the respective isomorphisms $\vartheta_{M,F,S(F)}^{\rm BK}$, $\vartheta_{M,F',S(F')}^{\rm BK}$ and $\vartheta_{M,F,S(F')}^{\rm BK}$; $\iota$ and $\iota'$ denote the natural inclusions.

It is straightforward to check that both left hand squares of this diagram commute and the commutativity of the upper right hand square follows from the argument of \cite[Th. 3.1]{BFetnc}. In addition, by the same general argument used in the verification of (\ref{per reg}), the factor $P_{F'/F}$ guarantees the commutativity of the lower right hand square since $\vartheta_F'$ uses the acyclicity of $R\Gamma_f(K_\fq,\varepsilon \cdot {\rm Ind}_{G_K}^{G_F}(V))$ for each $\fq$ in $S(F')\setminus S(F)$ whilst $\vartheta_F\circ \iota'$ trivializes these complexes by using morphisms of the type \cite[(24)]{BFetnc}.

Now the upper and lower composite homomorphisms in (\ref{key diagram}) are respectively equal to $\lambda^{\rm BK}_{\underline{b},F'}$ and $\lambda^{\rm BK}_{\underline{b},F}$. Thus, since $\lambda^{\rm BK}_{\underline{b},F}$ is injective,
the claimed equality ${\rm Cor}^r_{F'/F}(\eta_{\underline{b},F'}^{\rm BK}) = P_{F'/F}\cdot\eta_{\underline{b},F}^{\rm BK}$ follows directly from the explicit definitions of $\eta_{\underline{b},F'}^{\rm BK}$ and $\eta_{\underline{b},F}^{\rm BK}$, the commutativity of (\ref{key diagram}) and the fact that the standard inflation invariance of $L$-series implies $\pi(L^\ast (M_{F'/K}^\vee,0)) = L^\ast(M_{F/K}^\vee,0).$

This completes the proof of Proposition \ref{cor BK}.\end{proof}

\begin{definition} We refer to the element $\eta_{\underline{b},\cK}^{\rm BK} = \eta_{V,\underline{b},S,\cK}^{\rm BK}$ of ${\rm ES}_r(T,\cK)$ that is constructed (conjecturally) by Proposition \ref{cor BK} as the `Bloch-Kato Euler system' associated to the data $V, \underline{b}$, $S$  and $\cK$.\end{definition}

Examples \ref{bkrs1} and \ref{bkrs2} show that the notion of Bloch-Kato Euler system constitutes a simultaneous conjectural generalization of a number of important Euler systems.

\subsection{A new strategy for proving the eTNC} \label{strategy}

In this section we explain how Theorems \ref{main} and \ref{prop Xi} combine with the construction of Bloch-Kato Euler systems to give a new approach to proving Conjecture \ref{etnc}.

\subsubsection{The general approach} We assume for simplicity that the motive $M$ in \S\ref{etnc sec} has coefficients in a number field $B$, that $A$ is a finite extension $Q$ of $\QQ_p$ that occurs in the decomposition of $\QQ_p\otimes_\QQ B$ and that $\mathcal{A}$ is the valuation ring $\mathcal{O}$ of $Q$. We write $\fp$ for the prime ideal of $B$ that lies above $p$ and is such that $Q=B_\fp$ and then fix 
%
a Galois-stable lattice $T$ in the $\fp$-adic realization of $M$.

We also choose an abelian $p$-extension $\cK/K$ that satisfies Hypothesis \ref{hyp K}. We then fix a field $F$ in $\Omega(\cK/K)$ and set  $G:=\cG_F(=\Gal(F/K))$.

Finally we set
$$S:=S_\infty(K) \cup S_p(K) \cup S_{\rm ram}(T) \cup S_{\rm ram}(F/K).$$
We note that in this case $S(F)=S$.

We also write $\varepsilon$ for the idempotent $\varepsilon_{T,F}$ of $Q[G]$ defined in Definition \ref{def adm}.



\begin{theorem}\label{cor1} Assume that all of the following conditions are satisfied.
\begin{itemize}
\item[(a)] $Y_K(T)$ does not vanish.
\item[(b)] Writing $r$ for ${\rm rank}_{\cO}(Y_K(T))$, there exists an Euler system $c=(c_{K'})_{K'}$ in ${\rm ES}_r(T,\cK)$ for which $c_F$ has the same image as $\eta^{\rm BK}_{\underline{b},F}$.
\item[(c)] 
 $T$ and $F$ satisfy the hypotheses (H$_0$), (H$_1$), (H$_2$), (H$_3$), (H$_4$) and (H$_{5}$) listed in \S\ref{stand hyp sec}.
\end{itemize}
Let $\cR$ denote the subring of $Q[G]$ given by
$$\{ x \in Q[G] \varepsilon \mid x \cdot {\rm Fitt}_{\cO[G]}^0(H^2(\cO_{F,S},T))\subseteq{\rm Fitt}_{\cO[G]}^0(H^2(\cO_{F,S},T))\}.$$
Then the following claims are valid.

\begin{itemize}
\item[(i)] $\cR$ is an $\cO$-order in $Q[G]\varepsilon$.
\item[(ii)] $\cR\cdot L^\ast(M_{F/K}^\vee,0)\subseteq \cR\cdot \vartheta_{M,F,S}^{\rm BK}({\det}_{\cO[G]}(C_{F,S}(T)))$.
\item[(iii)] If ${\rm TNC}(M^\vee, \cO)$ is valid, then ${\rm TNC}(M_{F}^\vee,\cR)$ is valid.
\end{itemize}
\end{theorem}

\begin{proof} At the outset we note that, under the assumed validity of (H$_1$) and (H$_3$), Lemma \ref{torsionfree} implies the conditions (a) and (b) in Theorem \ref{prop Xi} are satisfied for all fields $K'$ in $\Omega(\cK/K)$.

This shows, in particular, that the idempotent  $\varepsilon$ is equal to the sum of all primitive idempotents of $Q[G]$ that annihilate $H^2(\cO_{F,S},V)$. This implies that
  ${\rm Fitt}_{\cO[G]}^0(H^2(\cO_{F,S},T))$ spans $Q[G]\varepsilon$ and hence that $\cR$ is an order in $Q[G]\varepsilon$, as required to prove claim (i).

Next we note that Lemma \ref{compare} combines with the hypothesis (H$_5$) to imply that the Selmer structures $\cF_{\rm can}$ and $\cF_{\rm ur}$ on $\cT$ coincide. Under the stated condition (c) we can therefore apply
Theorem \ref{main}(ii) to deduce that $\chi_{\rm can}(\cT)$ is equal to the integer $r$ defined in claim (ii). As a consequence of condition (a) we also know that $r$ is strictly positive.

This observation combines with condition (c) to allow us to apply Theorem \ref{main}(iii) to the Euler system $c$ that arises in condition (b) in order to deduce that
\begin{equation}\label{step 1} \im(\eta_{\underline{b},F}^{\rm BK}) = \im(c_F) \subseteq  {\rm Fitt}_{\cO[G]}^{0}(H^{1}_{{\cF}_{\rm can}^{*}}(K, \cT^{\vee}(1))^{\vee}) = {\rm Fitt}_{\cO[G]}^0(H^2(\cO_{F,S},T)),\end{equation}
where the second equality follows from Lemma \ref{compare} and Example \ref{selmer exams}(iii).

We abbreviate the module $\Xi(\cA,F/K,T)$ that occurs in Theorem \ref{prop Xi} to $\Xi(F)$. Then, upon comparing the latter result with (\ref{step 1}) one obtains an inclusion
\begin{equation}\label{step 2}\Xi (F)\cdot{\rm Fitt}_{\cO[G]}^0(H^2(\cO_{F,S},T))\subseteq{\rm Fitt}_{\cO[G]}^0(H^2(\cO_{F,S},T))\end{equation}
and then claim (ii) follows immediately from the definition of $\cR$.

To prove claim (iii) we need to show that the validity of ${\rm TNC}(M^\vee, \cO)$ implies that $\cR \cdot \Xi(F) = \cR$.

Thus, since claim (ii) implies $\cR \cdot \Xi(F) \subseteq \cR$, it suffices, by Nakayama's Lemma, to show that the modules of $G$-coinvariants
 $(\cR\cdot \Xi(F))_G$ and $\cR_G$ coincide.

But from Remark \ref{remark etnc} we know firstly that
 $\Xi (F)_G = \Xi (K)$ and then also that the validity of ${\rm TNC}(M^\vee,\cO)$ implies $\Xi(F)_G=\cO$. From the latter equality we can then deduce that $(\cR\cdot \Xi(F))_G = \cR_G$, as required.
\end{proof}

\begin{remark}\label{rem weak} Since $r > 0$ (by condition (a)), the condition (b) in Theorem \ref{cor1} is obviously satisfied if the Euler system $\eta^{\rm BK}_{\underline{b},\cK}$ exists, or equivalently (in view of Proposition \ref{cor BK}) if Conjecture \ref{integrality} is valid for every $K'$ in $\Omega(\cK/K)$. \end{remark}


\begin{remark}\label{ass order rem} Rings of the form $\cR$ that arise in Theorem \ref{cor1} are known as `associated orders' and are much studied in the literature (see, for example, \cite[Def. 24.3]{CR}). Nevertheless, except in special cases, it is difficult to determine such orders explicitly and, in \S\ref{iwasawa sec}, we explain how to use Iwasawa theory to address this issue. As an example, the following result shows that it is even difficult to characterize explicitly the case that $\cR$ is `minimal'.

\begin{lemma}\label{min order} Assume $\mathcal{O}[G]\varepsilon$ is Gorenstein. Then $\mathcal{R}$ is equal to $\mathcal{O}[G]\varepsilon$ if and only if the $\mathcal{O}[G]\varepsilon$-module $\mathcal{O}[G]\varepsilon\otimes_{\mathcal{O}[G]}H^2(U_K,T)$ has projective dimension at most one.
\end{lemma}

\begin{proof} Set $\cR' := \mathcal{O}[G]\varepsilon$. Then, since $\cR'$  is Gorenstein, the general result of \cite[Th. 37.13]{CR} implies that
$\mathcal{R}= \cR'$ if and only if the $\cR'$-module ${\rm Fitt}^0_{\cO[G]}(H^2(U_K,T))$ is free of rank one.

Now, setting $H^2(U_K,T)' := \cR'\otimes_{\cO[G]}H^2(U_K,T)$, one has
\[ {\rm Fitt}^0_{\cO[G]}(H^2(U_K,T)) = {\rm Fitt}^0_{\cO[G]}(H^2(U_K,T))\varepsilon = {\rm Fitt}^0_{\cR'}(H^2(U_K,T)'),\]
where the first equality follows from the proof of Theorem \ref{cor1}(i) and the second from a standard property of Fitting ideals under scalar extension.

Then, since $H^2(U_K,T)'$ is finite, the argument of Cornacchia and Greither in
\cite[Prop. 4]{cogr} implies ${\rm Fitt}^0_{\cR'}(H^2(U_K,T)')$ is a free $\cR'$-module of rank one if and only if the $\cR'$-module  $H^2(U_K,T)'$ has projective dimension at most one, as claimed.
\end{proof}\end{remark}


%
%

\subsubsection{A special case}\label{special case sec} In certain cases the hypothesis of Theorem \ref{cor1}(iii) can be made more explicit.

To explain this point we fix a finite abelian extension $L$ of $K$ of degree prime to $p$ and write $\Delta$ for $\Gal(L/K)$. We also set  $\widehat \Delta:=\Hom(\Delta,\QQ^{c,\times})$ and fix an embedding $\QQ^c \to \QQ^c_p$ in order to identify this group with $\Hom(\Delta, \QQ_p^{c,\times})$. 
%
%

We fix a finite extension $Q$ of $\QQ_p$ that contains the values of all characters in $\widehat{\Delta}$ and write $\cO$ for its valuation ring.

Then for each $\chi$ in $\widehat\Delta$ we define the `$\chi$-component' of a $\Z_p[\Delta]$-module $X$ to be the $\cO$-module
\begin{equation*} X^\chi := e_\chi(\cO\otimes_{\Z_p}X) = \{x \in \cO\otimes_{\ZZ_p}X: \delta(x) = \chi(\delta)\cdot x\,\text{ for all }\, \delta \in \Delta\}.\end{equation*}
(This notation is not, strictly speaking, precise since we do not specify the ring $\cO$ but this ambiguity should not cause confusion.)

In this way, for any $\ZZ_p$-representation $T$ we obtain an $\cO$-representation of $G_K$ by setting
\begin{equation*}\label{chi comp} T_\chi := ({\rm Ind}_{G_L}^{G_K}(T))^\chi = (\Z_p[\Delta]\otimes_{\ZZ_p}T)^\chi ,\end{equation*}
where the tensor product is a $\ZZ_p[\Delta]$-module via left multiplication and a $G_K$-module via the action $\sigma(x\otimes t) := x\sigma_L^{-1}\otimes \sigma(t)$ for all $\sigma\in G_K$, $x \in \Z_p[\Delta]$ and $t \in T$, where $\sigma_L$ denotes the image of $\sigma$ in $\Delta$.

Since $\#\Delta$ is prime to $p$ this gives a direct sum decomposition of $G_K$-representations
\begin{equation}\label{semi decomp} \cO\otimes_{\ZZ_p}{\rm Ind}^{G_K}_{G_L}(T) = \bigoplus_{\chi\in \widehat\Delta}T_\chi .\end{equation}

%
%
%
%

\begin{proposition}\label{cond reduction} Let $T$ be a full $G_K$-stable sublattice of the $p$-adic realisation of a motive $M$ over $K$. 

We also assume to be given a subset $\Upsilon$ of $\widehat{\Delta}$ that satisfies the following three conditions:
\begin{itemize}
\item[(a)] for all $\chi$ in $\Upsilon$ the hypotheses (a), (b) and (c) in Theorem \ref{cor1} are satisfied by the data $(T_\chi,K)$;
\item[(b)] $H^2(\cO_{K,S},T_\chi)$ is finite;
\item[(c)] $\Upsilon$ is stable under the action of $G_\QQ$ (so that $\sum_{\chi \in \Upsilon}\chi$ is a rational-valued character).
\end{itemize}
Then the validity of the $p$-part of ${\rm TNC}(M_{K'}^\vee)$ for all intermediate fields $K'$ of $L/K$ implies the validity of
${\rm TNC}(M_{L}^\vee,\cO e_\chi)$ for all $\chi$ in $\Upsilon$.
\end{proposition}

\begin{proof} 

For fields $K'$ and $L'$ with $K \subseteq K'\subseteq L'\subseteq L$ and $\Delta' = \Gal(L'/K')$ we write $\Xi_{L'/K'}$ for the $\ZZ_p[\Delta']$-lattice $\Xi(\ZZ_p,T,L'/K')$ defined in Remark \ref{remark etnc}. For $\chi$ in $\widehat{\Delta'}$ we then identify $\Xi^\chi_{L'/K'}$ with the fractional $\cO$-ideal obtained via the map $Qe_\chi \to Q$ sending $e_\chi$ to $1$.

We note that, with this notation, the conjecture TNC$(M_{L}^\vee,\mathcal{O}e_\chi)$ is equivalent to an equality $\Xi^\chi_{L/K} = \cO$.

Now condition (a) implies Theorem \ref{cor1} applies to the data $(T_\chi,K,S)$ and then condition (b) implies that, in this case, $\varepsilon= e_\chi$ and so the order $\cR$ in Theorem \ref{cor1} is equal to $\cO e_\chi$. From Theorem \ref{cor1}(ii) we can thus deduce that $\Xi^\chi_{L/K}\subseteq \mathcal{O}$ for all $\chi$ in $\Upsilon$.

Next we note that condition (c) implies the character $\sum_{\chi \in \Upsilon}\chi $ takes values in $\QQ$. Hence, by the Artin induction theorem (see \cite[Chap. II, Th. 1.2]{tatebook}), there exists a natural number $m$ and an integer $n_{K'}$ for each intermediate field $K'$ of $L/K$ such that
$$m\cdot \sum_{\chi \in \Upsilon}\chi = \sum_{K \subseteq K' \subseteq L} n_{K'}\cdot {\rm Ind}^{\Delta}_{\Delta_{K'}}(1_{K'}) = \sum_{K \subseteq K' \subseteq L} n_{K'}\cdot (\sum_{\psi(\Delta_{K'}) = 1}\psi), $$
where we set $\Delta_{K'} := \Gal(L/K')$ and write $1_{K'}$ is the trivial character of $\Delta_{K'}$.

This in turn gives an equality of fractional $\cO$-ideals
\begin{equation}\label{rep thry} \bigl( \prod_{\chi \in \Upsilon} \Xi^\chi_{L/K} \bigr)^m =
\prod_{K \subseteq K' \subseteq L} \bigl(\prod_{\chi(\Delta_{K'}) = 1}\Xi^\chi_{L/K}\bigr)^{n_{K'}} = \prod_{K \subseteq K' \subseteq L}\Xi_{K'}^{n_{K'}},\end{equation}
where we set $\Xi_{K'} := \Xi^{1_{K'}}_{K'/K'}$ and the second equality is a consequence of the following fact: if we set $\Gamma' := \Gal(K'/K)= \Delta/\Delta_{K'}$, then %
\[ \prod_{\chi(\Delta_{K'}) = 1}\Xi^\chi_{L/K} = \prod_{\phi \in \widehat{\Gamma'}}\Xi^\phi_{K'/K} = \Xi_{K'}\]
where the first equality follows from the functorial property recorded in Remark \ref{remark etnc}(i) and the second from (\ref{semi decomp}) (with $L$ replaced by $K'$) and Remark \ref{remark etnc}(ii).

Now the assumed validity of the $p$-part of TNC$(M_{K'}^\vee)$ implies $\Xi_{K'}$ is equal to $\cO$ and hence, via (\ref{rep thry}), that the stated conditions imply $\prod_{\chi \in \Upsilon} \Xi^\chi_{L/K}=\cO$. Since each module $\Xi^\chi_{L/K}$ is an ideal of $\cO$, it therefore follows that $\Xi^\chi_{L/K} = \cO$, as required.\end{proof}

\begin{remark} Proposition \ref{cond reduction} clarifies the interest of Theorem \ref{cor1} since, in important cases, ${\rm TNC}(M_{K'}^\vee)$ is equivalent to well-known results or conjectures (cf. Remark \ref{ex L1}).
\end{remark}

\subsection{Iwasawa theory}\label{iwasawa sec} In this section we explain how Iwasawa theory can in some cases be used to avoid the difficulty of explicitly describing the order $\cR$ that occurs in Theorem \ref{cor1} (cf. Remark \ref{ass order rem}).

To do this we fix a field extension $\cK$ of $K$ that satisfies Hypotheses \ref{hyp K} and a field $F$ in $\Omega(\cK/K)$. We also fix a finite abelian extension $L$ of $K$ of degree prime to $p$, set $\Delta := \Gal(L/K)$ and use the same notation as in \S\ref{special case sec}. 

For each finite extension $K'$ of $K$ we write $\cR_{K'}$ for the $\cO$-order $\cO[\cG_{K'}]$.


Finally we fix a $\ZZ_p$-power extension $K_\infty$ of $K$ in $\cK$ and set $\mathcal{H}_F^\infty := \Gal(FK_\infty/K_\infty)$. We write $\cB_{LF}$ for the smallest $\mathcal{O}$-order in $Q[\cG_{LF}]$ that contains  $\cR_{LF}$ and the image under the natural projection $ \QQ_p[\mathcal{H}_F^\infty] \to \QQ_p[\cG_{F}]$ of the maximal $\ZZ_p$-order in $\QQ_p[\mathcal{H}_F^\infty]$.

\begin{theorem}\label{hes2} Let $M$ be a motive over $K$ and $T$ a full $G_K$-stable sublattice of its $p$-adic realisation. 
Let $\Upsilon$ be a subset of $\widehat{\Delta}$ that is stable under the action of $G_\QQ$ and comprises characters $\chi$ that satisfy all of the following conditions:
\begin{itemize}
\item[(a)] $Y_{K}(T_\chi)$ does not vanish;
\item[(b)] Conjecture \ref{integrality} is valid for $(T_\chi,K')$ for every $K'$ in $\Omega(\cK/K)$;
\item[(c)] $H^2(\cO_{K,S},T_\chi)$ is finite;
\item[(d)] the data $T_\chi$, $F$ and $S$ satisfy the hypotheses (H$_0$), (H$_1$), (H$_4$) and (H$_5$) listed in \S\ref{stand hyp sec};
\item[(e)] for some $\tau$ in $\Gal(\overline{K}/FK_{p^\infty}K_\infty)$ the $\cO$-module $T_\chi/(\tau -1)T_\chi$ is free of rank one;
\item[(f)] $H^{1}(F(T)_{p^{\infty}}K_\infty/K, T_\chi/\pi T_\chi)$ and $H^{1}(F(T)_{p^{\infty}}K_\infty/K, (T_\chi/\pi T_\chi)^{\vee}(1))$ both vanish, where $\pi$ is a prime element of $\cO$;
\item[(g)] for each homomorphism $\psi: \mathcal{H}_F^\infty\to \QQ_p^{c,\times}$ there exists a field $F_\psi$ in $\Omega(FK_\infty/F)$ with $e_\psi \varepsilon_{T_\chi,F_\psi}\not=0$.
\end{itemize}

Then, if the $p$-part of the conjecture ${\rm TNC}(M^\vee_{K'})$ is valid for all intermediate fields $K'$ of $L/K$, the conjecture ${\rm TNC}(M^\vee_{LF},\cB^\chi_{LF})$ is also valid for every $\chi$ in $\Upsilon$.
\end{theorem}

\begin{remark}\label{simple case} If $K_\infty$ is the cyclotomic $\Z_p$-extension of $K$, then the conditions (e) and (f) used above coincide with the respective hypotheses (H$_2$) and (H$_3$) listed in \S\ref{stand hyp sec}.
 Note also that if $F$ is contained in $K_\infty$, then $\mathcal{H}_F^\infty$ is trivial and ${\cB}^\chi_{LF} = \mathcal{O}[\cG_{LF}]e_\chi$ is a direct factor of $\cO[\cG_{LF}]$.\end{remark}

\begin{proof} 

Fix $\chi$ in $\Upsilon$. Then under the given assumptions (d), (e) and (f), one checks easily that all of the hypotheses (H$_0$), (H$_1$), (H$_2$), (H$_3$), (H$_4$) and (H$_5$) are satisfied by the data $T_\chi, F'$ and $S$ for all fields $F'$ in $\Omega(FK_\infty/F)$.

 Recalling Remark \ref{rem weak}, the stated conditions (a) and (b) therefore allow us to apply the argument of Theorem \ref{cor1} to deduce that the inclusions (\ref{step 1}) and (\ref{step 2}) are valid with $F$ replaced by any such field $F'$.

Further, Lemma \ref{codescent} below implies that both of these inclusions are compatible with the natural projection maps as $F'$ varies over $\Omega(FK_\infty/F)$ and so, by passing to the inverse limit over such morphisms, we obtain, for any choice of $\cO$-basis $\underline{b}_\chi$ of $Y_K(T_\chi)$, inclusions
\begin{equation}\label{step22} \varprojlim_{F'}\im(\eta^{\rm BK}_{\underline{b}_\chi,F'})\subseteq {\rm Ann}_{\Lambda_\infty}(H^{2,\chi}_{\infty})\, \,\text{ and }\,\, \Xi^\chi_{\infty}\cdot \mathcal{I}^\chi_\infty \subseteq \mathcal{I}^\chi_\infty.\end{equation}
Here $\Lambda_\infty$ is the Iwasawa algebra $\varprojlim_{F'}\cR_{F'}$ and we use the following inverse limits over $F'$: $H^{2,\chi}_{\infty}$ is the limit of $H^2(U_{F'},T_\chi)$ with respect to the natural corestriction maps; $\Xi^\chi_{\infty}$ is the limit of the lattices $\Xi^\chi_{F'L} := \Xi(\cO,T_\chi,F'/K)$ with respect to the morphisms induced by the natural projection maps $\pi_{F_2,F_1}: \cR_{F_2} \to \cR_{F_1}$ for $F_1\subseteq F_2$ (cf. Remark \ref{remark etnc}(i)); $\mathcal{I}^\chi_\infty$ is the limit of ${\rm Fitt}_{\cR_{F'}}^{0}(H^{2}(U_{F'},T_\chi))$ with respect to the maps $\pi_{F_2,F_1}$.

We note $\Lambda_\infty$ is a power series ring over $\mathcal{O}[\mathcal{H}_F^\infty]$ in variables $\{t_i\}_{\le i\le d}$ (for a suitable $d$), the total quotient ring $Q_\infty$ of $\Lambda_\infty$ is a finite product of fields and the integral closure $\cM_\infty$ of $\Lambda_\infty$ in $Q_\infty$ is $\cM[[t_1,\dots , t_d]]$, with $\cM$ the integral closure of $\mathcal{O}$ in $\mathcal{O}[1/p][\mathcal{H}_F^\infty]$.

In Proposition \ref{inverse} below we show that $\Xi^\chi_{\infty}\subseteq {\cM}_\infty$. Then, by taking coinvariants under the action of the group $\Gal(FK_\infty/F)$ of this inclusion, and noting that the image of $\cM_\infty$ in $\cR_F[1/p]$ generates $\cB^\chi_{LF}$ over $\cR_{LF}^\chi$ we obtain an inclusion
$\cB_{LF}^\chi\cdot \Xi^\chi_{LF} \subseteq \cB^\chi_{LF}$.

To verify TNC$(M^\vee_{LF},\cB^\chi_{LF})$ we need to show the last inclusion is an equality. However, since $\Upsilon$ is stable under the action of $G_\QQ$, the stated conditions (a), (b), (c), (d) and (e) imply that the hypotheses (a), (b) and (c) of Proposition \ref{cond reduction} are satisfied by the data $(T_\chi,K,S)$. The latter result therefore combines with our assumption that the $p$-part of TNC$(M^\vee_{K'})$ is valid for all intermediate fields $K'$ of $L/K$ to imply TNC$(M_L^\vee,\cO e_\chi)$ is valid.

Given this, one deduces that the inclusion $\cB^\chi_{LF}\cdot \Xi^\chi_{LF} \subseteq \cB^\chi_{LF}$ is an equality by using Nakayama's Lemma, just as in the proof of Theorem \ref{cor1}.\end{proof}

%

\begin{lemma}\label{codescent}
Let $F_1$ and $F_2$ be finite abelian $p$-extensions of $K$ with both $F_1\subseteq F_2$ and $S(F_1)=S(F_2)$. Set $\cR_i:=\cR_{F_i}$ for $i \in \{1,2\}$. The the following claims are valid. 
\begin{itemize}
\item[(i)] The $\cR_1$-modules $\cR_1\otimes_{\cR_2}H^2(U_{F_2},T_\chi)$ and $H^2(U_{F_1},T_\chi)$ are naturally isomorphic.
\item[(ii)] The natural projection map $\cR_{2} \to \cR_{1}$ restricts to give surjective homomorphisms ${\rm Fitt}_{\cR_{2}}^{0}(H^2(U_{F_2},T_\chi)) \to {\rm Fitt}_{\cR_{1}}^{0}(H^2(U_{F_1},T_\chi))$ and  $\im(\eta^{\rm BK}_{\underline{b}_\chi,F_2}) \to \im(\eta^{\rm BK}_{\underline{b}_\chi,F_1})$.

\end{itemize}
\end{lemma}

\begin{proof}
There are projection formula isomorphisms
\begin{equation}\label{d iso} \cR_{1}\otimes^{\mathbb{L}}_{\cR_{2}}C_{F_2,S}(T_\chi)\simeq C_{F_1,S}(T_\chi)\,\,\text{ and }\,\, \cR_1\otimes^\mathbb{L}_{\cR_{2}}R\Gamma(U_{F_2},T_\chi) \simeq R\Gamma(U_{F_1},T_\chi).\end{equation}
%

The isomorphism in claim (i) is induced by the second isomorphism here together with the fact that $R\Gamma(U_{F_2},T_\chi)$ is acyclic in degrees greater than two (as $p$ is odd).

The first assertion in claim (ii) follows directly upon combining the isomorphism in claim (i) with standard properties of Fitting invariants under ring extension.

We set $X_i := H^1(U_{F_i},T_\chi)$ for $i\in \{1, 2\}$ and also $\cH := \Gal(F_2/F_1)$.

Next we note that the first isomorphism in (\ref{d iso}) combines with Lemma \ref{torsionfree}, Proposition \ref{prop C} and the assumed validity of (H$_1$) and (H$_3$) to imply both that $X_1$ identifies with $H^0(\cH,X_2)$ and that $X_2$ is torsion-free. These facts combine to imply that the quotient $X_2/X_1$ is torsion-free.

Thus, since the $\cO$-order $\cR_2$ is Gorenstein, upon applying the functor $\Hom_{\cR_{2}}(-,\cR_{2})$ to the tautological short exact sequence $0 \to X_1 \to X_2 \to X_2/X_1 \to 0$ we can therefore deduce that the following composite homomorphism is surjective
\[ \varrho_\cH: \Hom_{\cR_{2}}(X_2,\cR_{2}) \to \Hom_{\cR_{2}}(X_1,\cR_{2}) = \Hom_{\cR_{1}}(X_1,H^0(\cH,\cR_{2})) \to \Hom_{\cR_{1}}(X_1,\cR_{1}).\]
Here the first arrow is the natural restriction map and the second is induced by the isomorphism $H^0(\cH,\cR_{2}) = \cR_2\cdot(\sum_{h \in \cH}h) \to \cR_1$ that sends $\sum_{h \in \cH}h$ to the identity element of $\cG_{F_1}$.

In addition, by explicit computation, one checks that for any subset $\{\varphi_{i}\}_{1\le i\le r}$ of $\Hom_{\cR_2}(X_2,\cR_2)$ the following diagram commutes
\[ \begin{CD}
\bigcap_{\cR_2}^rX_2 @> {\wedge}_{i=1}^{i=r}\varphi_i >> \cR_2 \\
@V {\rm Cor}^r_{T_\chi,F_2/F_1} VV @VV\pi_{F_2,F_1} V\\
\bigcap_{\cR_1}^rX_1 @> {\wedge}_{i=1}^{i=r}\varrho_\cH(\varphi_i) >> \cR_1.\end{CD}\]
%

Given the commutativity of this diagram, the second assertion of claim (ii) follows from the surjectivity of $\varrho_\cH$, the equality ${\rm Cor}^r_{T_\chi,F_2/F_1}(\eta^{\rm BK}_{\underline{b}_\chi,F_2}) = \eta^{\rm BK}_{\underline{b}_\chi,F_1}$ (resulting from Proposition \ref{cor BK} and the fact that $S(F_2) = S(F_1)$) and the fact that for both $j=1$ and $j=2$ one has
\[ \im(\eta^{\rm BK}_{\underline{b}_\chi,F_j}) = \{ (\wedge_{i=1}^{i=r}\varphi_i)(\eta^{\rm BK}_{\underline{b}_\chi,F_j})\! :\! \varphi_i\in H^1(U_{F_j},T_\chi)^{\ast} \,\text{ for }\, 1\le i\le r\}.\]
\end{proof}

\begin{proposition}\label{inverse} Under the hypotheses of Theorem \ref{hes2}, the following claims are valid.
\begin{itemize}
\item[(i)] $H^{2,\chi}_\infty$ is torsion over $\Lambda_\infty$.
\item[(ii)] $\Xi^\chi_{\infty}$ is contained in $\cM_\infty$.
\end{itemize}
\end{proposition}

\begin{proof} The first inclusion in (\ref{step22}) implies the $\Lambda_\infty$-module $H^{2,\chi}_\infty$ is torsion if  $\varprojlim_{F'}\im(\eta^{\rm BK}_{\underline{b}_\chi,F'})$ contains a non-zero divisor of $\Lambda_\infty$.
To prove this it is in turn enough to show that for each homomorphism $\psi: \mathcal{H}^\infty_F \to \QQ_p^{c\times}$ the limit $e_\psi\cdot \varprojlim_{F'}\im(\eta^{\rm BK}_{\underline{b}_\chi,F'})$ is non-zero.

Now, by assumption (g) in Theorem \ref{hes2}, for each such $\psi$ there exists a field $F_\psi$ in $\Omega(FK_\infty/F)$ with $e_\psi \varepsilon_{T_\chi,F_\psi}\not= 0$. Via the interpolation property in Definition \ref{def BK} this implies $e_\psi\cdot \eta^{\rm BK}_{\underline{b}_\chi,F'}\not=0$ and hence also that $e_\psi\cdot \im(\eta^{\rm BK}_{\underline{b}_\chi,F'})\not= 0$. Given this, the non-vanishing of $e_\psi\cdot \varprojlim_{F'}\im(\eta^{\rm BK}_{\underline{b}_\chi,F'})$ is an immediate consequence of the final assertion of Lemma \ref{codescent}(ii) and this proves claim (i).

Turning to claim (ii) we note that the first claim of Lemma \ref{codescent}(ii) implies $\mathcal{I}^\chi_\infty$ is equal to ${\rm Fitt}^0_{\Lambda_\infty}(H^{2,\chi}_\infty)$. It therefore follows from claim (i) that $\mathcal{I}^\chi_\infty$ contains a non-zero divisor $\xi$ of $\Lambda_\infty$.

This fact in turn implies that the $\Lambda_\infty$-module $\cA^\chi_\infty := \{\lambda \in Q_\infty: \lambda\cdot \mathcal{I}^\chi_\infty \subseteq \mathcal{I}^\chi_\infty\}$ is contained in $\Lambda_\infty\cdot \xi^{-1}$ and so is finitely generated. Since $\cA^\chi_\infty$ is a subring of $Q_\infty$ that contains $\Lambda_\infty$ it follows that $\cA^\chi_\infty$ is contained in $\cM_\infty$.

We are now reduced to showing that $\Xi^\chi_\infty$ is a subset of $Q_\infty$ since, if this is true, then the second inclusion in (\ref{step22}) implies $\Xi^\chi_{\infty}$ is contained in $\cA^\chi_\infty$ and hence also in $\cM_\infty$, as required to prove claim (ii).

To prove $\Xi^\chi_\infty$ is a subset of $Q_\infty$ we set $Y_\infty^\chi := Y_{\infty}(T_\chi)^\ast\otimes_{\ZZ_p}\Lambda_\infty$ and write $C^\chi_\infty$ for the inverse limit $\varprojlim_{F'}C_{F',S(F')}(T_\chi)$, where $F'$ runs over $\Omega(FK_\infty/F)$ and the transition morphisms are induced by the first isomorphism in (\ref{d iso}).

Then, by passing to the inverse limit over $F'$ of the result in Lemma \ref{prop C}, one finds that $C^\chi_\infty$ belongs to $D^{\rm perf}(\Lambda_\infty)$, is acyclic outside degrees zero and one and that there exists a canonical short exact sequence $0 \to H^{2,\chi}_\infty \to H^1(C_\infty^\chi) \to Y_\infty^\chi \to 0$.

In particular, since $Q_\infty$ is a semisimple algebra that, by claim (i), annihilates $\sha^\chi_\infty$, one deduces that $H^0(C_\infty^\chi)$ spans a free $Q_\infty$-module of rank $r$ and hence that the canonical `passage to cohomology' map induces an isomorphism of $Q_\infty$-modules
\begin{equation*}\label{infty pass} \vartheta_\infty: Q_\infty\otimes_{\Lambda_\infty}{\rm det}_{\Lambda_\infty}(C_\infty^\chi) \simeq Q_\infty\otimes_{\Lambda_\infty}\bigl({\bigwedge}_{\Lambda_\infty}^rH^0(C_\infty^\chi)\otimes \Hom_{\Lambda_\infty}({\bigwedge}_{\Lambda_\infty}^rY_\infty^\chi,\Lambda_\infty)\bigr).\end{equation*}

Further, by the argument of \cite[Th. 3.5(i)]{bks2}, one knows that this morphism is induced by the inverse limit over $F'$ of the composite morphisms

\begin{align*} \vartheta_{F'}: {\rm det}_{\Lambda_\infty}(C_\infty^\chi) &\twoheadrightarrow \, \cR_{F'}\otimes_{\Lambda_\infty}{\rm det}_{\Lambda_\infty}(C_\infty^\chi)\\
&\simeq \, {\rm det}_{\cR_{F'}}(C_{F'}^\chi)\\
&\to \, {\rm det}_{\cR_{F'}[1/p]}(H^0(C_{F'}^\chi))\otimes_{\cR_{F'}[1/p]}{\rm det}^{-1}_{\cR_{F'}[1/p]}(H^1(C_{F'}^\chi))\\
&\to \, \QQ_p\cdot \bigl({\bigwedge}_{\cR_{F'}}^rH^0(C_{F'}^\chi)\otimes \Hom_{\cR_{F'}}({\bigwedge}_{\cR_{F'}}^r Y_{F'}(T_\chi),\cR_{F'})\bigr).\end{align*}

Here the first map is the natural projection, the second is induced by the isomorphism $\cR_{F'}\otimes_{\Lambda_\infty}^\mathbb{L}C_\infty^\chi \simeq C_{F'}^\chi$ given by the limit over $F'$ of (\ref{d iso}), the third is induced by the natural `passage to cohomology' map and the fourth is induced by multiplication by the idempotent $\varepsilon_{T_\chi,F'}$.

Now, if we fix a basis $b^\chi_\infty$ of the free rank one $\Lambda_\infty$-module ${\rm det}_{\Lambda_\infty}(C_\infty^\chi)$, then the definition of $\vartheta_{F'}$ combines with Definition \ref{def BK} of the Bloch-Kato element $\eta^{\rm BK}_{\underline{b}_\chi,F'}$ to imply that the lattice $\Xi^\chi_{F'}$ is generated over $\cR_{F'}$ by the unique element $c_{F'}$ of $\cR_{F'}[1/p]\varepsilon_{T_\chi,F'}$ that satisfies
\[ \vartheta_{F'}(b^\chi_\infty) = c_{F'}\cdot (\eta^{\rm BK}_{\underline{b}_\chi,F'}\otimes {\wedge}_{i=1}^{i=r}b_{\chi,i,F'}^\ast),\]
where $b_{\chi,i,F'}$ denotes the $i$-th element of the (ordered) basis $\underline{b}_\chi$ regarded as an element of $Y_{F'}(T_\chi)= Y_K(T_\chi)\otimes_{\cO}\cO[\cG_{F'}]$ and $b_{\chi,i,F'}^\ast$ is the element of $\Hom_{\cR_{F'}}({\bigwedge}_{\cR_{F'}}^r Y_{F'}(T_\chi),\cR_{F'})$ that is dual to it.

In particular, since the elements $c_{F'}$ are compatible with the natural transition morphisms as $F'$ varies, it is enough to prove the element $c_\infty := (c_{F'})_{F'}$ belongs to $Q_\infty$.

This is true since \cite[Lem. 3.6]{bks2} implies that the inverse limit $(\eta^{\rm BK}_{\underline{b}_\chi,F'}\otimes {\wedge}_{i=1}^{i=r}b_{\chi,i,F'}^\ast)_{F'}$ belongs to $\im(\vartheta_\infty)$, whilst the element $\vartheta_{\infty}(1\otimes b^\chi_\infty)$ is a basis over $Q_\infty$ of $\im(\vartheta_\infty)$.\end{proof}

\section{The multiplicative group}\label{section 4}

To give a first application of our approach, we consider the example that was originally considered by Rubin in \cite{rubinstark} and then subsequently by B\"uy\"ukboduk in \cite{Buyuk}.

In this section we write $\omega$ for the Teichm\"uller character of $G_K$ and $A_{K'}$ for the $p$-part of the ideal class group of a finite extension $K'$ of $K$.

\subsection{Statement of the main result}


We fix an odd prime $p$ and set $T := \ZZ_p(1)$.

We also fix a finite abelian extension $L$ of $K$ of degree prime to $p$ and for each $\chi$ in $\widehat{\cG_L}$ we use the notion of `$\chi$-components' from \S\ref{special case sec}. In particular, in the sequel we will use the valuation ring $\cO$, with fraction field $Q$, that was introduced in \S\ref{special case sec}.

We next fix a finite abelian extension $F$ of $K$ of $p$-power degree and set $G := \cG_F$ and $S := S_\infty(K)\cup S_p(K)\cup S_{\rm ram}(LF/K).$

Finally, we assume to be given a pro-$p$ abelian extension $\cK$ of $K$ that contains $F$, satisfies Hypothesis \ref{hyp K}, and so contains a $\ZZ_p$-power extension $K_\infty$ of $K$, and is such that the Rubin-Stark Conjecture is valid for $LF'/K$ for all $F'$ in $\Omega(\cK/K)$.

Then Remark \ref{remark etnc rs} shows Conjecture \ref{integrality} is valid for the data $T_\chi, S$ and $\cK$ and so Proposition \ref{cor BK} implies that 
the Bloch-Kato Euler system $\eta^{\rm BK}_{\underline{b},\cK}$ exists for any choice of ordered $\mathcal{O}$-basis $\underline{b}$ of $Y_{K}(T_\chi)$.

We fix $\underline{b}$ as in Example \ref{bkrs2} so that $\eta^{\rm BK}_{\underline{b},\cK}$ can be interpreted in terms of Rubin-Stark elements. We shall therefore label this Euler system as $\eta^{\rm RS}_\chi = \eta^{\rm RS}_{\chi,\cK}$ and refer to it as the `Rubin-Stark Euler system' for the data $T_\chi, S$ and $\cK$.




In claim (ii) of the following result we write $\cB_{LF}$ for the $\ZZ_p$-order in $\QQ_p[\cG_{LF}]$ that is generated over $\ZZ_p[\cG_{LF}]$ by the integral closure of $\ZZ_p$ in $\QQ_p[\Gal(LF/(F\cap K_\infty))]$ (so that $\cB_{LF} = \ZZ_p[\cG_{LF}]$ if $F$ is contained in $K_\infty$).

\begin{theorem}\label{main RS} Assume that $\chi$ is neither trivial nor equal to $\omega$ and also that if $p=3$, then $\chi^2\not=\omega$. Assume, in addition, that $\chi(G_{K_v}) = 1$ for at least one archimedean place $v$ of $K$ and that $\chi(G_{K_v})\not= 1$ for all $v$ in $S_{\rm ram}(F/K)\cup S_p(K)$.

Then the rank $r = r_\chi := {\rm rank}_{\cO}(Y_K(T_\chi))$ is strictly positive and the following claims are valid.

\begin{itemize}
\item[(i)] For every non-negative integer $j$ one has
\[ I_j(\cR_r^{-1}(\cD_{r}(\eta_{\chi}^{\rm RS})))\!=\! {\rm Fitt}_{\mathcal{O}[G]}^j(A_{LF}^\chi).\]
In particular, letting $j=0$, one has
$$I(\eta_{\chi,F}^{\rm RS})(:=\im (\eta_{\chi,F}^{\rm RS}))={\rm Fitt}_{\cO[G]}^0(A_{LF}^\chi).$$

\item[(ii)] If the group $A_L^\chi$ is trivial, then ${\rm TNC}(\QQ(0)_{LF},\ZZ_p[\cG_{LF}]^\chi)$ is valid. In all cases ${\rm TNC}(\QQ(0)_{LF},\cB_{LF}^\chi)$ is valid.
\end{itemize}
\end{theorem}

\begin{remark} \label{cor a1} This result is of interest for several reasons.

\noindent{}(i) Theorem \ref{main RS}(i) implies Rubin-Stark Euler systems control detailed aspects of the fine Galois structure of ideal class groups and strongly refines results of Rubin in \cite{rubincrelle} and \cite{rubinstark} and, more recently, of B\"uy\"ukboduk in \cite{Buyuk}. In particular, the main result of the latter article deals only with $j=0$ (in which case, we recall, $I_j(\cR_{r_\chi}^{-1}(\cD_{r_\chi}(\eta_{\chi}^{\rm RS}))) = \im(\eta_{\chi,F}^{\rm RS})$) and assumes, amongst other things, that $K$ is totally real, $L/K$ is unramified at $p$ and, crucially, that both $F = K$ and Leopoldt's conjecture is valid.

\noindent{}(ii) Theorem \ref{main RS}(i) also strongly refines the main results of Kurihara and the first and third authors in \cite{bks1}. To explain this we note that \cite[(2)]{bks1} specialises to give a short exact sequence of $\cO[G]$-modules $0 \to A_{LF}^\chi \to \cS^{{\rm tr},\chi}_{LF,p} \to X_{p}^\chi \to 0$, where $\cS^{\rm tr}_{LF,p}$ and $X_{p}$ are the pro-$p$ completions of the transpose Selmer module
$\cS^{\rm tr}_{S_\infty(LF),\emptyset}(\mathbb{G}_{m/LF})$ from \cite[Def. 2.6]{bks1} and of the module $X_{LF,S_\infty(LF)}$ defined after \cite[(2)]{bks1}. In particular, if $ X_{p}^\chi$ is a free $\cO[G]$-module of rank $r$ (as is the case under the hypotheses of Theorem \ref{main RS}), then this sequence combines with \cite[Lem. 7.2]{bks1} to imply that for each $j\ge 0$ the ideal ${\rm Fitt}_{\cO[G]}^j(A_{LF}^\chi)$ is equal to $
{\rm Fitt}_{\cO[G]}^{r+j}(\cS^{{\rm tr},\chi}_{LF,p})$ and also to the higher relative Fitting ideal ${\rm Fitt}_{\cO[G]}^{(r,j)}(\cS^{{\rm tr},\chi}_{LF,p})$ defined in \cite[\S7.6]{bks1}. Theorem \ref{main RS}(i) therefore  shows that, under standard hypotheses,
 both \cite[Conj. 7.3 and Conj. 7.8]{bks1} follow directly from the Rubin-Stark Conjecture and thereby refines the results of \cite[Th. 1.10, Th. 7.5, Th. 7.9 and Th. 8.1]{bks1}.

\noindent{}(iii) Theorem \ref{main RS}(ii) provides the first cases in which the validity of conjectures of the form TNC$(\QQ(0)_{E},\cB)$ for $p$-adic orders $\cB$ that are not regular are deduced from the Rubin-Stark conjecture without assuming the vanishing of suitable $\mu$-invariants, the validity of cases of the `refined class number formula for $\mathbb{G}_m$' of Mazur and Rubin \cite{MRGm} and the third author \cite{sano} and the validity of an appropriate generalization of the `order of vanishing' conjecture for $p$-adic $L$-series due to Gross \cite{Gp}.

\noindent{}(iv) Fix an embedding $\QQ^c\to \QQ_p^c$ and use it to identify $\widehat{\cG_{LF}}$ with $\Hom(\cG_{LF},\QQ_p^{c,\times})$. Then the validity for all $\psi$ in $\widehat{\cG_F}$ of the `$p$-part' of the `Strong-Stark Conjecture' of Chinburg \cite[Conj. 2.2]{chinburg} for $\psi\chi$ in $\widehat{\cG_{LF}}$ is equivalent to the validity of ${\rm TNC}(\QQ(0)_{LF},\cM^\chi)$ with $\cM$ the maximal $\ZZ_p$-order in $\QQ_p[\cG_{LF}]$ (cf. \cite[\S4.3, Rem. 10]{BFetnc}) and is therefore implied by TNC$(\QQ(0)_{LF},\cB_{LF}^\chi)$ (cf. Remark \ref{remark etnc}). It follows that Theorem \ref{main RS}(ii) implies, under mild hypotheses, that the Rubin-Stark Conjecture implies the Strong-Stark Conjecture. This implication is new and, in special cases, as strong as possible. For example, if $L \cap K(\mu_p)=K$ and there exists a unique place of $L$ lying above each $p$-adic place of $K$, then the argument shows the Rubin-Stark Conjecture implies the validity of ${\rm TNC}(\QQ(0)_L,\ZZ_p[\Delta])$.
\end{remark}

%
%


\subsection{The proof of Theorem \ref{main RS}} We first investigate when the hypotheses listed in \S\ref{section hyp} are satisfied by the data $T_\chi$ and $F$.

\begin{lemma}\label{hyp gm} Under the conditions of Theorem \ref{main RS}, the data $T_\chi$ and $F$ satisfy all of the hypotheses  (H$_0$), (H$_1$), (H$_2$), (H$_3$), (H$_4$) and (H$_5$) listed in \S\ref{section hyp}.
\end{lemma}

\begin{proof} Hypothesis (H$_0$) is satisfied since if $\fq$ does not belong to $S$, then ${\rm Fr}_\fq$ acts on $T_\chi$ as multiplication by $\chi(\fq)^{-1}\cdot {\rm N}\fq$.

Since the $\cO$-module $T_\chi$ is free of rank one, hypothesis (H$_1$) is clear and (H$_2$) is satisfied by taking $\tau$ to be the identity element.

Hypothesis (H$_3$) is satisfied as a consequence of Lemma \ref{tech} below and that (H$_4$) is satisfied since we are assume
$\chi^2\not= \omega$ if $p=3$.

Finally, the hypothesis (H$_5$) is satisfied under the stated conditions since $(T_\chi/p)^\vee(1)$ identifies with $\mathcal{O}/(p)$ upon which $G_K$ acts via $\chi$ and so, for any prime $\fq$ of $K$, the module $H^{0}(K_{\fq}, \overline{T_\chi}^{\vee}(1))$ vanishes if and only if $\chi(G_{K_{\fq}})\not= 1$.
\end{proof}

\begin{lemma}\label{tech} Let $K'$ be any pro-$p$ abelian extension of $K$ and set $F_\chi := F(T_\chi)_{p^{\infty}}K'$. Then, if $N$ denotes either $(\mu_p)_\chi$ or $(\ZZ/(p))_{\chi^{-1}}$, the group $H^1(F_\chi/K,N)$ vanishes. \end{lemma}

\begin{proof} We set $K^\dagger := FK(1)K'(\mu_{p^\infty})$ and $(\cO_{K}^{\times})^{1/p^{\infty}}:= \{ x\in \overline{K} \mid x^{p^{m}} \in \cO_{K}^{\times} \text{ for some $m$}\}$ and write $L_\chi$ for the intermediate field of $L/K$ with $\ker(\chi) = \Gal(L/L_\chi)$. We then also define abelian groups $\Delta_1 := \Gal(L_\chi(\mu_p)/K)$ and $\Delta_2 := \Gal(L_\chi(\mu_p)/K(\mu_p))$. 

Then the field $F_\chi$ is, by its definition, equal to $L_\chi K^\ddag ((\cO_{K}^{\times})^{1/p^{\infty}})$. In addition, since $\Delta_1$ has order prime to $p$ and $G_{L_\chi(\mu_p)}$ acts trivially on $N$, the inflation-restriction sequence implies that
\[ H^1(F_\chi/K,N) = H^0(\Delta_1,H^1(F_\chi/L_\chi(\mu_p),N)) = \Hom_{\Delta_1}(\Gal(F_\chi^{\rm ab}/L_\chi(\mu_p)),N)\]
where $F_\chi^{\rm ab}$ denotes the maximal abelian extension of $L_\chi(\mu_p)$ inside $F_\chi$.

In particular, if this group does not vanish, then there exists a degree $p$ extension $E$ of $L_\chi(\mu_p)$ in $F_\chi^{\rm ab}$ that is Galois over $K$ and such that, since $\chi\notin \{1,\omega\}$, the conjugation action of $\Delta_1$ on $\Gal(E/L_\chi(\mu_p))\simeq N$ is not trivial, whilst the conjugation action of $\Delta_2$ on this module is trivial only if the subgroup $\Delta_2$ is itself trivial.

It follows that $E$ is not abelian over $K$ so that $E\cap L_\chi K^\dagger = L_\chi(\mu_p)$ and hence $EL_\chi K^\dagger$ is a cyclic degree $p$ extension of $L_\chi K^\dagger$ in $F_\chi$.

In particular, if $\Delta_2$ is not trivial, then the conjugation action of $\Gal(L_\chi K^\dagger/K^\dagger)\simeq \Delta_2$ on $\Gal(EL_\chi K^\dagger/L_\chi K^\dagger)\simeq \Gal(E/L_\chi(\mu_p))$ is non-trivial, and this contradicts Kummer theory since the definition of $(\cO_{K}^{\times})^{1/p^{\infty}}$ ensures that $EL_\chi K^\dagger$ is generated over $L_\chi K^\dagger$ by adjoining the $p$-th root of an element of $K^\dagger$.

We can therefore assume that the group $\Delta_2$ is trivial so that $L_\chi\subseteq K(\mu_p)$. We write $K_{\rm cyc}$ for the cyclotomic $\ZZ_p$-extension of $K$, set $K^\ddag := FK(1)K'K_{\rm cyc}$ and note that the groups $\Gal(EL_\chi K^\dagger/L_\chi K^\dagger) = \Gal(EK^\dagger/K^\dagger)$ and $\Gal(K^\dagger/K^{\ddag})$ identify with $\Gal(E/K(\mu_p))$ and $\Delta_1= \Gal(K(\mu_p)/K)$ via restriction. This again leads to a contradiction since $\chi$ is non-trivial, $\Delta_1$ acts on $\Gal(E/K(\mu_p))\simeq N$ either via $\chi$ or $\omega\chi^{-1}$ and Kummer theory implies that $\Gal(K^\dagger/K^{\ddag})$ acts on  $\Gal(EK^\dagger/K^\dagger)$ via $\omega$.

It follows that the groups $H^1(F_\chi/K,N)$ must vanish, as required. \end{proof}


Turning now to the proof of Theorem \ref{main RS} we note that for each $\chi$ the representation $T_\chi$ identifies with the module $\cO$ upon which $G_K$ acts via $\chi^{-1}\cdot \epsilon_{\rm cyc}$ with $\epsilon_{\rm cyc}$ the cyclotomic character of $K$.

By using this description, one checks that $T_\chi$ satisfies
Hypothesis \ref{hyp Y} with $\cA = \cO$ and that 
\begin{equation*}\label{rank eq} r_\chi = {\rm rank}_{\cO}\bigl(\bigoplus_{v \in S_\infty(K)}H^0\bigl(K_v, \Z_p[\Delta]^\chi\bigr)\bigr) =  \#\{ v \in S_{\infty}(K) \mid \chi(G_{K_{v}}) = 1 \}.\end{equation*}

In particular, since we assume $\chi(G_{K_v}) = 1$ for at least one archimedean place $v$ of $K$, the rank $r=r_\chi$ is strictly positive, as claimed in Theorem \ref{main RS}.

Set $\cT_\chi := {\rm Ind}^{G_K}_{G_F}(T_\chi)$. Then, under the given hypotheses, the group $H^{1}_{\cF_{{\rm can}}^{*}}(K, \cT_\chi^{\vee}(1))^{\vee}$ identifies with $A_{LF}^{\chi}$ (by Example \ref{selmer exams}(i) and Lemma \ref{compare}).

Since Lemma \ref{hyp gm} allows us to apply Theorem \ref{main}(iii) to $T_\chi$ and the Euler system $\eta^{\rm RS}_\chi$,
 we can therefore deduce that 
 there is for each non-negative integer $j$ an inclusion
\[ I_j(\cR_r^{-1}(\cD_{r}(\eta_{\chi}^{\rm RS}))) \subseteq  {\rm Fitt}_{\cO[G]}^j(A_{LF}^\chi)\]
and that this inclusion is an equality for all $j$ if it is an equality for $j=0$, or equivalently, if one has $\im(\eta_{\chi,F}^{\rm RS}) = {\rm Fitt}_{\cO[G]}^0(A_{LF}^\chi)$.

To verify the latter equality it is enough, by Nakayama's Lemma, to show that the homomorphism of $G$-coinvariants $\im(\eta_{\chi,F}^{\rm RS})_G\to {\rm Fitt}_{\cO[G]}^0(A_{LF}^\chi)_G$ that is induced by the  inclusion  $\im(\eta_{\chi,F}^{\rm RS})\subseteq {\rm Fitt}_{\cO[G]}^0(A_{LF}^\chi)$ is surjective.

To do this we note that the argument of Lemma \ref{codescent}(ii) gives an identification of $\im(\eta_{\chi,F}^{\rm RS})_G$ with $\im(\eta_{\chi,K}^{\rm RS})$. In addition, the identifications $H^{1}_{\cF_{{\rm can}}^{*}}(K, \cT^{\vee}(1))^{\vee} = A_{LF}^{\chi}$ and $H^{1}_{\cF_{{\rm can}}^{*}}(K, T^{\vee}(1))^{\vee} = A_{L}^{\chi}$ combine
with \cite[Cor. 3.8]{bss} to imply $(A_{LF}^{\chi})_G$ is naturally isomorphic to $A_{L}^{\chi}$ and hence that ${\rm Fitt}_{\cO[G]}^0(A_{LF}^\chi)_G = {\rm Fitt}_{\cO}^0(A_{L}^\chi)$.

To complete the proof of Theorem \ref{main RS}(i) it is thus enough to show that $\im(\eta_{\chi,K}^{\rm RS}) = {\rm Fitt}_{\cO}^0(A_{L}^\chi)$. However, the argument of \cite[Th. 8.1]{bks1} (with $\Gamma_K$ trivial and $i=0$) shows that this equality is implied by ${\rm TNC}(\QQ(0)_L,\cO e_\chi)$ and so can be deduced from the analytic class number formula via Remark \ref{ex L1}(i) and the argument of Proposition \ref{cond reduction}.

In a similar way, to prove the second assertion of Theorem \ref{main RS}(ii) we shall apply Theorem \ref{hes2} to the data $T = T_\chi$ and $c = \eta^{\rm RS}_\chi$.

Under the stated assumptions, the conditions (d), (e) and (f) in Theorem \ref{hes2} follow from Lemmas \ref{hyp gm} and \ref{tech}. In addition, condition (a) of Theorem \ref{hes2} is clear, (b) is equivalent to the assumed existence of $\eta_\chi^{\rm RS}$ and (c) is true since the hypothesis that $\chi(G_{K_v})\not= 1$ for all $v$ in $S_{\rm ram}(LF/K)\cup S_p(K)$ implies that for each intermediate field $F'$ of $F/K$ the module
$H^2(\cO_{F',S},T_\chi)$ identifies with $A_{LF'}^\chi$. Finally, condition (g) in Theorem \ref{hes2} follows from the fact that no non-archimedean place of $K$ splits completely in $K_\infty$. Given these observations, the second assertion of Theorem \ref{main RS}(ii) follows directly from Theorem \ref{hes2}.

It now only remains to prove the first assertion of Theorem \ref{main RS}(ii) and to do this we shall directly apply Theorem \ref{cor1} (and Proposition \ref{cond reduction}) to $T_\chi$ and $\eta^{\rm RS}_\chi$ rather than relying on Theorem \ref{hes2}.

In particular, just as above, the hypotheses of Theorem \ref{main RS} imply that the necessary conditions are satisfied in order to deduce from Theorem \ref{cor1} the validity of ${\rm TNC}(\QQ(0)_{LF},\mathcal{R})$ with $\mathcal{R}$ equal to the $\mathcal{O}$-order
\[\{ x \in Q[G]e_\chi \mid x \cdot {\rm Fitt}_{\cO[G]}^0(H^2(\cO_{F,S},T_\chi))\subseteq{\rm Fitt}_{\cO[G]}^0(H^2(\cO_{F,S},T_\chi))\}.\]

It is thus enough to show that if $A_{L}^\chi$ vanishes, then $H^2(\cO_{F,S},T_\chi)$ also vanishes, since then one would have ${\rm Fitt}_{\cO[G]}^0(H^2(\cO_{F,S},T_\chi)) = \cO[G]$ and hence $\mathcal{R} = \mathcal{O}[G]e_\chi = \ZZ_p[\cG_{LF}]^\chi$.

But this is an easy consequence of Nakayama's Lemma since, as noted above, the present hypotheses imply both that $H^2(\cO_{F,S},T_\chi)$ identifies with $A_{LF}^\chi$ and that $(A_{LF}^{\chi})_G$ is isomorphic to $A_{L}^{\chi}$.

This completes the proof of Theorem \ref{main RS}. 





\section{Elliptic curves}\label{section ell curves}

In this section we discuss applications of Theorems \ref{main} and \ref{hes2} to elliptic curves. In this way we formulate a natural generalization and refinement of a conjecture of Perrin-Riou \cite{PR} (see \S \ref{connection PR}) and also obtain concrete new evidence in favour of the equivariant Birch and Swinnerton-Dyer Conjecture (see \S\ref{eBSD}).

To do so we fix an elliptic curve $E$ over $K$. (In most parts we assume $K=\QQ$.) We also fix an odd prime $p$, write $T = T_p(E)$ for the $p$-adic Tate module of $E$ and set $V :=\QQ_p \otimes_{\ZZ_p} T$.

We also fix a finite abelian $p$-extension $F$ of $K$ and set $G:=\Gal(F/K)$. We then set
$$S:=S_\infty(K)\cup S_p(K) \cup S_{\rm ram}(F/K)\cup S_{\rm ram}(T) .$$


We write $\sha(E/F)$ and ${\rm Sel}_{p}(E/F)$ for the classical Tate-Shafarevich and $p$-primary Selmer groups of $E$ over $F$ and recall that the `strict $p$-Selmer group' ${\rm Sel}_{p}^{\rm str}(E/F)$ is the subgroup of ${\rm Sel}_{p}(E/F)$ defined in Example \ref{selmer exams}(ii).

In the sequel we shall abbreviate the idempotent $\varepsilon_{T,F}$ of $\QQ_p[G]$ that is defined in Definition \ref{def adm} to $\varepsilon$.

For each $\chi$ in $\widehat G$, we set
\[ {\rm rk}_\chi(E(F)) :=\dim_\CC(e_\chi(\CC\otimes_\ZZ E(F))).\]
For each $a$ in $\{0,1\}$ we then define an idempotent of $\QQ[G]$ by setting
\begin{equation}\label{eps a def} \varepsilon_a := \sum_\chi e_\chi \end{equation}
where $\chi$ runs over all characters in $\widehat G$ for which one has ${\rm rk}_\chi(E(F))=a$.


%

The following technical result will be useful in the sequel.

\begin{lemma}\label{elliptic coh}\
\begin{itemize}
\item[(i)] There is a canonical exact sequence
\begin{multline*}
0 \to \QQ_p \otimes_\ZZ E(F) \to H^1(\cO_{F,S},V) \to \bigoplus_{\fp \in S_p(F)} (\QQ_p \otimes_{\ZZ_p}E(F_\fp)^{\wedge})^\ast \\
\to \QQ_p \otimes_{\ZZ_p} {\rm Sel}_p(E/F)^\vee \to H^2(\cO_{F,S},V) \to 0,
\end{multline*}
where we denote by $(-)^\wedge$ and $(-)^\ast$ the $p$-completion and $\QQ_p$-dual respectively.
In particular, we have a canonical isomorphism
$$H^2(\cO_{F,S},V) \simeq \QQ_p \otimes_{\ZZ_p} {\rm Sel}_p^{\rm str}(E/F)^\vee.$$
\item[(ii)] If $\sha(E/F)[p^\infty]$ is finite, then $H^2(\cO_{F,S},V)^\ast$ identifies with the kernel of the diagonal localization map
$ \QQ_p \otimes_\ZZ E(F) \to \bigoplus_{\fp \in S_p(F)} \QQ_p \otimes_{\ZZ_p} E(F_\fp)^\wedge.$
\item[(iii)] If $K = \QQ$ and $\sha(E/F)[p^\infty]$ is finite, then $\varepsilon = \varepsilon_0 + \varepsilon_1$.
\end{itemize}
\end{lemma}
\begin{proof} Claim (i) follows by considering the long exact cohomology sequence of the exact triangle
$$R\Gamma_f(F,V) \to R\Gamma(\cO_{F,S},V) \to \bigoplus_{w \in S_F}R\Gamma_{/f}(F_w,V).$$
(See, for example, the bottom row of \cite[(26)]{BFetnc}.) 

Claim (ii) then follows from claim (i) by noting that if $\sha(E/F)[p^\infty]$ is finite, then $\QQ_p \otimes_{\ZZ_p} {\rm Sel}_p(E/F)^\vee$ identifies with $(\QQ_p \otimes_{\ZZ} E(F))^\ast.$

To prove claim (iii) we note that  $H^0(F,V)$ vanishes  and hence that $\varepsilon = \sum_\chi e_\chi$, where $\chi$ runs over all characters in $\widehat G$ for which $e_\chi(\QQ_p^c\otimes_{\QQ_p}H^2(\cO_{F,S},V))$ vanishes.

 Given this fact, and the explicit definitions of $\varepsilon_0$ and $\varepsilon_1$, it is straightforward to derive the equality $\varepsilon = \varepsilon_0 + \varepsilon_1$ from the description of $H^2(\cO_{F,S},V)$ given in claim (ii) and the fact that, for each $\chi$ in $\widehat G$, the $\chi$-component of the localization map in claim (ii) is injective if and only if one has ${\rm rk}_\chi(E(F))\le 1$ (as follows, for example, from the argument of Jones in \cite[Prop. 7.1]{jones}).
\end{proof}

\subsection{Bloch-Kato elements and Kato's zeta elements}\label{section kato}
 In this subsection we fix $K$ to be $\QQ$. We assume that the group $\sha(E/F)[p^\infty]$ is finite.

In this case one has $r_T:={\rm rank}_{\ZZ_p}(Y_\QQ(T))={\rm rank}_{\ZZ_p}(H^0(\RR,T))=1$ and so an ordered $\Z_p$-basis $\underline{b}$ of $Y_\QQ(T)$ comprises a single element. To specify this element
 we fix a generator $\gamma$ of $H_1(E(\CC),\ZZ)^+$ and then take $\underline{b}$ to be the element of $Y_\QQ(T)$ that corresponds to $1\otimes \gamma$ under the canonical comparison isomorphism
 $\ZZ_p \otimes_\ZZ H_1(E(\CC),\ZZ)^+\simeq T^+=Y_\QQ(T)$.

Noting that $F$ is totally real (since it is an extension of $\QQ$ of odd degree), we next fix an embedding $\iota_0: F \hookrightarrow \RR$ and use it to identify $Y_F(T)$ with $\ZZ_p[G]\otimes_{\ZZ_p} Y_\QQ(T)$ (as per Remark \ref{induced basis}). In this way we regard $\underline{b}$ as a $\ZZ_p[G]$-basis of $Y_F(T)$.

We then write $\eta_{F}^{\rm BSD}$ for the Bloch-Kato element $\eta_{\underline{b},F}^{\rm BK}$ in $\CC_p \otimes_{\ZZ_p} H^1(\cO_{F,S},T)$ (see Definition \ref{def BK}).

%

We next recall that Kato \cite{kato} has constructed a canonical `zeta element' $z_F^{\rm Kato}$ in $H^1(\cO_{F,S},V)$ and we shall specify this element precisely in Definition \ref{choose kato} below.

For the moment we note only the following: the article \cite{kato} implicitly fixes an embedding $\QQ^c \hookrightarrow \CC$ (as $\QQ^c$ is regarded as a subset of $\CC$) and the elements constructed in loc. cit. depend on this choice - we always assume that this embedding is fixed so that it restricts to give the embedding $\sigma_0: F \to \RR$ used above; we shall also further normalize the definition of $z_F^{\rm Kato}$ to take account of the generator $\gamma$ of $H_1(E(\CC),\ZZ)^+$ fixed above.
%

We can now formulate the following conjecture.

\begin{conjecture}\label{kato conj} $\eta_{F}^{\rm BSD} = z_F^{\rm Kato}.$
\end{conjecture}

In the next section we shall show that this conjecture is a natural generalization and refinement of the conjecture formulated by Perrin-Riou in \cite{PR}.

In the rest of this section we state a useful, and very explicit, reinterpretation of the conjecture.

To do this we fix a minimal Weierstrass model of $E$ over $\ZZ$ and let $\omega$ be the corresponding N\'eron differential in $ \Gamma(E,\Omega_{E/\QQ}^1)$. We then set
$$\Omega^+ :=\int_\gamma \omega.$$

We consider the composite homomorphism
$$\exp_\omega^\ast: \left(\QQ_p \otimes_{\ZZ_p} \bigoplus_{\fp \in S_p(F)}E(F_\fp)^\wedge\right)^\ast \xrightarrow{\exp^\ast} \QQ_p \otimes_\QQ F\otimes_\QQ \Gamma(E,\Omega_{E/\QQ}^1) \xrightarrow{\omega \mapsto 1} \QQ_p \otimes_\QQ F$$
where $\exp^\ast$ denotes the dual exponential map. We will also denote by $\exp^\ast_\omega$ the composition of this map with the localization map
$$H^1(\cO_{F,S},V) \to \bigoplus_{\fp \in S_p(F)} H^1_{/f}(F_\fp,V)\simeq \bigoplus_{\fp \in S_p(F)}H^1_f(F_\fp,V)^\ast\simeq \left(\QQ_p \otimes_{\ZZ_p} \bigoplus_{\fp \in S_p(F)}E(F_\fp)^\wedge\right)^\ast . $$

For the idempotent $\varepsilon_1$ defined in (\ref{eps a def}), Lemma \ref{elliptic coh} implies that there are  natural isomorphisms

\begin{eqnarray}\label{isom1}
\varepsilon_1 (\QQ_p \otimes_\ZZ E(F)) \xrightarrow{\sim} \varepsilon_1 H^1(\cO_{F,S},V),
\end{eqnarray}
and
\begin{eqnarray}\label{isom2}
\varepsilon_1 \left(\QQ_p \otimes_{\ZZ_p} \bigoplus_{\fp \in S_p(F)} E(F_\fp)^\wedge \right)^\ast \xrightarrow{\sim} \varepsilon_1(\QQ_p \otimes_\ZZ E(F))^\ast.
\end{eqnarray}
In particular, we may regard $\varepsilon_1 z_F^{\rm Kato}$ as an element of $\varepsilon_1(\QQ_p \otimes_\ZZ E(F))$ and $\exp_\omega^\ast$ as a map on $\varepsilon_1(\QQ_p \otimes_{\ZZ} E(F))^\ast$.

Finally, we write $\langle -,-\rangle: E(F) \times E(F) \to \RR$ for the classical N\'eron-Tate height pairing. Since we fix an embedding of $\RR$ into $\CC_p$ we shall always identify this pairing
 with the pairing $\CC_p \otimes_\ZZ E(F) \times \CC_p \otimes_\ZZ E(F) \to \CC_p$ that it induces.

The following result will be proved in \S\ref{proof of kato prop}.

\begin{proposition}\label{kato prop}
Conjecture \ref{kato conj} is valid if and only if all of the following conditions are satisfied.
\begin{itemize}
\item[(a)] For all $\chi$ in $\widehat G$ with both $L(E,\chi,1) = 0$ and ${\rm rk}_\chi(E(F))=0$ one has
\[ \Omega^+ \sum_{\sigma \in G} \iota_0( \exp_\omega^\ast(\sigma z_F^{\rm Kato})) \chi(\sigma)=L_S^\ast(E,\chi,1).\]
\item[(b)] The element $\varepsilon_1 z_F^{\rm Kato}$ generates $\varepsilon_1(\CC_p \otimes_\ZZ E(F))$ over $\CC_p[G]\varepsilon_1$ and it's dual  $(\varepsilon_1z_F^{\rm Kato})^\ast$ in $\varepsilon_1(\CC_p \otimes_\ZZ E(F))^\ast$ is such that for all $\chi$ in $\widehat G$ with
    ${\rm rk}_\chi(E(F))=1$ one has
$$\Omega^+ \langle \varepsilon_1 z_F^{\rm Kato}, \varepsilon_1 z_F^{\rm Kato} \rangle \sum_{\sigma \in G} \iota_0(\exp_\omega^\ast(\sigma (\varepsilon_1 z_F^{\rm Kato})^\ast)) \chi(\sigma)=L_S^\ast(E,\chi,1),$$
where we extend $\iota_0: F \rightarrow \RR$ to a map $\CC_p \otimes_\QQ F \to \CC_p$ in the obvious way.
\item[(c)] For all $\chi$ in $\widehat G$ with ${\rm rk}_\chi(E(F)) >1$ one has $e_\chi z_F^{\rm Kato}=0$.
\end{itemize}
In particular, if Conjecture \ref{kato conj} is true, then for every $\chi \in \widehat G$ we have
\begin{eqnarray}\label{kato vanish}
e_\chi z_F^{\rm Kato}=0 \Longleftrightarrow {\rm rk}_\chi(E(F))>1.
\end{eqnarray}
\end{proposition}

\begin{remark}\label{rem kato prop} The explicit conditions listed in Proposition \ref{kato prop} do not involve any characters for which one has $L(E,\chi,1)\not= 0$. In fact, our proof of Proposition \ref{kato prop} will show that for all such $\chi$ the required equality $e_\chi\cdot \eta_{F}^{\rm BSD} = e_\chi\cdot z_F^{\rm Kato}$ is unconditionally valid. For details see Proposition \ref{kato bk} below.\end{remark}

\subsection{Perrin-Riou's Conjecture}\label{connection PR} In \S\ref{eBSD} we shall obtain evidence in favour of Conjecture \ref{kato conj} in the setting of general abelian fields $F$ (see Theorem \ref{cor kato}(iii)).

In this section, however, we focus on the case $F =K= \QQ$ and, following Remark \ref{rem kato prop}, we may also assume that $L(E,1)=0$.

In the following result we show that this special case of Conjecture \ref{kato conj} recovers a well-known conjecture of Perrin-Riou.

\begin{proposition}\label{kato cor} \
\begin{itemize}
\item[(i)] If ${\rm rank}(E(\QQ))=1$, then Conjecture \ref{kato conj} is valid for $E$ if and only if for any point $P$ in $E(\QQ)$ that generates $E(\QQ)/E(\QQ)_{\rm tors}$ one has
$$\log_\omega(z_\QQ^{\rm Kato})=\frac{L_S^\ast(E,1)}{\Omega^+ \langle P, P \rangle} (\log_\omega (P))^2,$$
where $\log_\omega: \QQ_p\otimes_\ZZ E(\QQ)\simeq \QQ_p \otimes_{\ZZ_p} E(\QQ_p)^\wedge \to \QQ_p$ is the formal logarithm associated to $\omega$ and we regard $z_\QQ^{\rm Kato}$ as an element of $\QQ_p \otimes_\ZZ E(\QQ)$ via (\ref{isom1}).
\item[(ii)] If ${\rm rank}(E(\QQ))>1$, then Conjecture \ref{kato conj} is valid for $E$ if and only if $z_\QQ^{\rm Kato}$ vanishes.
\item[(iii)] Assume that $E$ validates the Birch and Swinnerton-Dyer conjecture over $\QQ$ and has good reduction at $p$. Then the conditions in claims (i) and (ii) are respectively equivalent to the conjectures \cite[Conj. 3.3.5(i) and Conj. 3.3.2]{PR} of Perrin-Riou.
\end{itemize}
\end{proposition}

\begin{proof} If $E(\QQ)$ has rank one, then (\ref{isom1}) gives an identification  $H^1(\cO_{\QQ,S},V) = \QQ_p \otimes_\ZZ E(\QQ)$ and so $z_\QQ^{\rm Kato}=\alpha \cdot P$ for some $\alpha$ in $\QQ_p$.

Since $\langle \alpha P , \alpha P \rangle \cdot (\alpha P)^\ast =\langle P,P\rangle \cdot \alpha \cdot P^\ast$, Proposition \ref{kato prop}(b) implies that Conjecture \ref{kato conj} is in this case equivalent to an equality
$$\alpha \cdot \exp_\omega^\ast(P^\ast)=\frac{L_S^\ast(E,1)}{\Omega^+ \langle P, P\rangle},$$
and hence, since $\log_\omega(P)\exp^\ast_\omega(P^\ast)=1$, to an equality
$$z_\QQ^{\rm Kato}=\frac{L_S^\ast(E,1)}{\Omega^+ \langle P, P\rangle} \log_\omega(P)\cdot P.$$
Since the map $\log_\omega$ is injective, this shows that the validity of Conjecture \ref{kato conj} is equivalent to the displayed equality in claim (i), as required.

Claim (ii) is true since if ${\rm rank}(E(\QQ))>1$, then Proposition \ref{kato prop}(c) directly implies that Conjecture \ref{kato conj} is equivalent to the vanishing of $z_\QQ^{\rm Kato}$.

To prove claim (iii) we assume that $E$ validates the Birch and Swinnerton-Dyer conjecture and also has good reduction at $p$. In particular, the rank of $E(\QQ)$ is equal to the order of vanishing of $L(E,s)$ at $s=1$.
If, firstly, $E(\QQ)$ has rank one, then 
the Birch-Swinnerton-Dyer formula implies
$$\frac{L_S'(E,1)}{\Omega^+ \langle P, P\rangle}={\rm Eul}_S^{-1}\cdot \frac{\# \sha(E/\QQ) {\rm Tam}(E)}{(\#E(\QQ)_{\rm tors})^2},$$
where 
${\rm Eul}_S \in \QQ^\times$ is the product of Euler factors at primes in $S$, which satisfies ${\rm Eul}_S \cdot L_S'(E,1)=L'(E,1)$, and
${\rm Tam}(E)$ is the product of Tamagawa factors.

The displayed equality in claim (i) is therefore equivalent to a formula
$$\log_\omega(z_\QQ^{\rm Kato})= {\rm Eul}_S^{-1}\cdot \frac{\# \sha(E/\QQ) {\rm Tam}(E)}{(\#E(\QQ)_{\rm tors})^2}(\log_\omega (P))^2,$$
and \cite[Prop. 2.2.2]{PR} shows that this formula is equivalent to \cite[Conj. 3.3.5(i)]{PR}. (Note that with our choice of normalization implies that Euler factors at primes in $S$ occur in the formula.)

Next we assume that ${\rm rank}(E(\QQ)) > 1$ and hence that ${\rm ord}_{s=1}L(E,s)>1$. Claim (ii) therefore asserts that $z_\QQ^{\rm Kato}$ vanishes if and only if ${\rm ord}_{s=1}L(E,s)>1$ and this is precisely the
statement of \cite[Conj. 3.3.2]{PR}.
\end{proof}

\begin{remark}\label{perrin riou}\

\noindent{}(i) In Proposition \ref{kato cor}(iii) we assume that $E$ has good reduction at $p$ only because this is assumed by Perrin-Riou.

\noindent{}(ii) The conjectures of Perrin-Riou are much studied in the literature and, via Proposition \ref{kato cor}, all results in this direction can be regarded as evidence in favour of Conjecture \ref{kato conj}. For example, in \cite{BD} Bertolini and Darmon report that they can prove the conjectures of \cite{PR} when $E$ has good ordinary reduction at $p$, whilst in \cite{venerucci} Venerucci proves an analogue of the conjectures when $E$ has split multiplicative reduction at $p$. More precisely, in \cite[Th. B]{venerucci} it is shown that $z_\QQ^{\rm Kato}$ vanishes if and only if ${\rm ord}_{s=1}L(E,s)>1$ and in \cite[Th. A]{venerucci} that if ${\rm ord}_{s=1}L(E,s)=1$, then there exists a non-zero rational number $\ell_1$ and a point $P$ in $E(\QQ)$ such that $\log_\omega(z_\QQ^{\rm Kato})=\ell_1 (\log_\omega(P))^2$ which proves a non-explicit version of the displayed equality in Proposition \ref{kato cor}(i). More recently, in \cite{BPS} B\"uy\"ukboduk, Pollack and Sasaki gave a proof of Perrin-Riou's conjecture in the good ordinary case that is different from that of Bertolini and Darmon. For details of further known results on the conjecture, see the discussion in \cite{buyuk perrin}.\end{remark}

\subsection{Kato's Euler system} In \cite[Ex. 13.3]{kato} Kato uses zeta elements to construct an Euler system. Since his definition of Euler systems is slightly different from ours, we explain how to use zeta elements to construct an Euler system in our sense that has `good' integrality properties.

In this section we continue to assume that $K=\QQ$. We use the standard notations from \cite{kato}.


\subsubsection{The Euler system}\label{section euler system}

Let $N$ be the conductor of $E$ and $f = \sum_{n=1}^\infty a_n q^n \in S_2(X_1(N))$ be the normalized newform corresponding to $E$. Let
$${}_{c,d}z_m^{(p)}(f,1,1,\xi,S_m) \in H^1(\cO_{\QQ(\mu_m),S_m},T(f))$$
be the element in \cite[(8.1.3)]{kato} (with $k=2$, $r=r'=1$), where
\begin{itemize}
\item $m $ is a positive integer,
\item $S_m:=S \cup \{\ell \mid m\}$ (note that with our convention $S$ contains the infinite place),
\item $c$ and $d$ are integers greater than 1 that are coprime to $6$ and to all primes in $S_m$ and are such that $c \equiv d \equiv 1 \text{ (mod $N$)}$ (this condition is necessary in order to apply \cite[Th. 6.6(1)]{kato}),
\item $\xi$ is a matrix in ${\rm SL}_2(\ZZ)$,
\item $T(f)$ is the maximal quotient of $H^1(Y_1(N)\times_\QQ \QQ^c, \ZZ_p(1))$ on which Hecke operators $T(n)$ act via $a_n$.
\end{itemize}
Note that, by fixing a modular parametrization $X_1(N) \to E$, we can regard $T(f)$ as the image of the following map:
\begin{eqnarray*}
H^1(Y_1(N)\times_\QQ \QQ^c ,\ZZ_p(1)) &\hookrightarrow& H^1(Y_1(N)\times_\QQ \QQ^c, \QQ_p(1))\\
&\to& H^1(X_1(N)\times_\QQ \QQ^c, \QQ_p(1)) \\
&\to & H^1(E\times_\QQ \QQ^c,\QQ_p(1)) \\
&\simeq& V,
\end{eqnarray*}
 where the second map is the Drinfeld-Manin splitting, the third is induced by the modular parametrization. In this way one can regard $T(f)$ as a sublattice of $V$. 

 Let $V_\ZZ(f)(1)$ be the maximal quotient of $H^1(Y_1(N)(\CC),\ZZ(1))$ on which Hecke operators $T(n)$ act via $a_n$. This is regarded as a sublattice of $H^1(E(\CC),\QQ(1))\simeq H_1(E(\CC),\QQ)$ via the modular parametrization, in the same way as above. We have the natural comparison isomorphism $V_{\ZZ}(f)(1)\otimes_{\ZZ} \ZZ_p \simeq T(f)$.

Let $\QQ(\mu_m)$ be the minimal cyclotomic field containing $F$. Then for any integers $c$ and $d$ as above we define an element of $\ZZ[G] \cap \QQ[G]^\times$ by setting
$$t_{c,d}:=cd(c-\sigma_c)(d-\sigma_d),$$
where $\sigma_a$ is the element of $G$ obtained by restricting the automorphism of $\QQ(\mu_m)$ that sends $\zeta_m$ to $\zeta_m^a$.


For each matrix $\xi$ in ${\rm SL}_2(\ZZ)$ we write $\delta(\xi)$ for the image in $H_1(E(\CC),\QQ)$ of the modular symbol
$$\{\xi(0),\xi(\infty)\} \in H_1(X_1(N)(\CC), \{{\rm cusps}\}, \ZZ) \simeq H^1(Y_1(N)(\CC),\ZZ)(1)$$
in $V_\ZZ(f)(1) \subset H_1(E(\CC),\QQ)$. We regard $\delta(\xi)$ as an element of $T(f)$ via the comparison isomorphism $V_{\ZZ}(f)(1)\otimes_\ZZ \ZZ_p\simeq T(f)$.

Since, by Manin's theorem, the group $H_1(X_1(N)(\CC), \{{\rm cusps}\}, \ZZ) $ is generated by the set $\{\{\alpha(0),\alpha(\infty)\} \mid \alpha \in {\rm SL}_2(\ZZ)\}$ we may, and will, choose $\xi$ in ${\rm SL}_2(\ZZ)$ so that the $\ZZ_p$-module $T(f)^+$ is generated by $e^+ \delta(\xi)$, where we set $e^+:=(1+c_\infty)/2$, with $c_\infty$ denoting complex conjugation.


We next note that the maximal abelian pro-$p$ extension $\cK$ of $\QQ$ that is unramified at primes dividing $cd$ satisfies Hypothesis \ref{hyp K}.

For each field $F'$ in $\Omega(\cK/\QQ)$ we now define an element ${}_{c,d} z_{F'}^{\rm Kato}$ in $H^1(\cO_{F',S(F')},T(f))$ as follows: writing $\QQ(\mu_{m'})$ for the minimal cyclotomic field that contains $F'$ (so that $S(F')=S_{m'}$), we set
$${}_{c,d}z_{F'}^{\rm Kato}:={\rm Cor}_{\QQ(\mu_{m'})/F'}({}_{c,d} z_{m'}^{(p)}(f,1,1,\xi,S(F')))$$
with the above choices of integers $c$ and $d$ and matrix $\xi$.

\begin{lemma}\label{kato euler} The collection ${}_{c,d}z^{\rm Kato}:=({}_{c,d}z_{F'}^{\rm Kato})_{F'}$ belongs to ${\rm ES}_1(T(f),\cK).$
\end{lemma}

\begin{proof}
This follows directly from \cite[Prop. 8.12]{kato}, by noting that for any prime $\ell \notin S$ the polynomial $P_\ell(x):=\det(1-{\rm Fr}_\ell^{-1}x \mid T(f)^\ast(1))$ is equal to $1-a_\ell \ell^{-1}x + \ell^{-1}x^2.$
\end{proof}

\subsubsection{The definition of $z_F^{\rm Kato}$}\label{norm sec}

We recall that $\underline{b}$ and $e^+\delta(\xi)$ are bases of the $\ZZ_p$-modules $Y_\QQ(T)=T^+$ and $T(f)^+$. In particular, since $T^+$ and $T(f)^+$ are both lattices of $V^+$ there exists a unique $u$ in $\QQ_p^\times$ such that 
\begin{equation}\label{norm equation} e^+\delta(\xi)=u\cdot \underline{b} \text{ in }V^+.
\end{equation}

With the following definition, we now make the statement of Conjecture \ref{kato conj} precise.
\begin{definition}\label{choose kato} With $u$ in $\QQ_p^\times$ fixed as above, and $c$ and $d$ any choice of integers as in \S\ref{section euler system}, we set
$$z_F^{\rm Kato} := u^{-1}t_{c,d}^{-1} \cdot {}_{c,d} z_F^{\rm Kato} \in H^1(\cO_{F,S},V).$$
\end{definition}

\begin{remark}\label{rem euler}
If the residual Galois representation $T/p$ is irreducible,
then $T(f)$ and $T$ are homothetic (by \cite[Chap. I, \S1.1, Exc. 4]{serre}) and $u$ satisfies
\begin{equation*}T(f) = u\cdot T.\end{equation*}
In particular, Lemma \ref{kato euler} implies that the collection
%
%
\[  (u^{-1}\cdot {}_{c,d} z_{F'}^{\rm Kato})_{F'}\]
belongs to ${\rm ES}_1(T,\cK)$. Since $T(f)$ and $T$ are isomorphic Galois representations, we can identify them and, with this identification, the system $(u^{-1}\cdot {}_{c,d} z_{F'}^{\rm Kato})_{F'}$ is identified with ${}_{c,d}z^{\rm Kato}=({}_{c,d}z_{F'}^{\rm Kato})_{F'}$.
\end{remark}

\subsubsection{The proof of Proposition \ref{kato prop}}\label{proof of kato prop} Having defined the element $z_F^{\rm Kato}$ we can now prove Proposition \ref{kato prop}.

At the outset we note that, if satisfied, the respective conditions (a), (b) and (c) in Proposition \ref{kato prop} would uniquely determine the elements $e_{\chi}\cdot z_F^{\rm Kato}$ for all characters $\chi$ in $\widehat G \setminus \widehat G_0$ with
\[ \widehat G_0 := \{\chi\in \widehat G \mid L(E,\chi,1) \not=0\}.\]
(Note that $L(E,\chi,1) \neq 0$ implies ${\rm rk}_\chi(E(F)) = 0$. See \cite[Th. 14.2(2)]{kato}.)

Thus, if we could show that these conditions are satisfied with $z_F^{\rm Kato}$ replaced by $\eta_F := \eta_F^{\rm BSD}$, then it would follow that $(1-\varepsilon_0') \cdot z_F^{\rm Kato} = (1-\varepsilon_0')\cdot \eta_F^{\rm BSD}$, with $\varepsilon_0' := \sum_\chi e_\chi$ where $\chi$ runs over $\widehat G_0$.

To complete the proof of Proposition \ref{kato prop} it would then be enough to show that $\varepsilon_0' \cdot z_F^{\rm Kato} = \varepsilon_0'\cdot \eta_F^{\rm BSD}$ and this follows directly from Proposition \ref{kato bk} below.

It therefore suffices to check that the conditions (a), (b) and (c) in Proposition \ref{kato prop} are satisfied with $z_F^{\rm Kato}$ replaced by $\eta_F$.

To do this we recall that $\eta_F$ is characterized by an equality
$$\lambda(\eta_F)=\varepsilon\cdot \sum_{\chi \in \widehat G}L^\ast(E,\chi^{-1},1) e_\chi,$$
for a canonical isomorphism $\lambda = \lambda_{\underline{b},F}^{\rm BK}: \varepsilon(\CC_p \otimes_{\ZZ_p} H^1(\cO_{F,S},T)) \xrightarrow{\sim} \CC_p[G]\varepsilon$ and that, by Lemma \ref{elliptic coh}(iii), one has  $\varepsilon=\varepsilon_0+\varepsilon_1$.

In particular, since $\lambda$ is injective, this shows that the element $e_\chi \eta_F$ vanishes if and only if ${\rm rk}_\chi(E(F)) > 1$. This establishes that $\eta_F$ has property (c) and also
that the equivalence (\ref{kato vanish}) is valid.

Next we note that a straightforward check shows that $\varepsilon_0\cdot \lambda$ is given by
$$\lambda: \varepsilon_0(\CC_p \otimes_{\ZZ_p} H^1(\cO_{F,S},T)) \xrightarrow{\sim} \CC_p[G]\varepsilon_0; \ a \mapsto  {\rm Eul}_S \cdot \Omega^+ \sum_{\sigma \in G}\iota_0(\exp_\omega^\ast (\sigma a))\sigma^{-1},$$
where ${\rm Eul}_S \in \QQ[G]^\times$ is the product of Euler factors at primes in $S$, which satisfies
$${\rm Eul}_S \cdot \sum_{\chi \in \widehat G}L_S^\ast(E,\chi^{-1},1)e_\chi=\sum_{\chi \in \widehat G}L^\ast(E,\chi^{-1},1)e_\chi.$$
From this, we see that $\eta_F$ has the property (a).

In a similar way, an explicit check shows $\varepsilon_1\cdot \lambda$ is the following composite map:
\begin{eqnarray*}
\lambda: \varepsilon_1(\CC_p \otimes_{\ZZ_p} H^1(\cO_{F,S},T)) &\stackrel{(\ref{isom1})}{\simeq}& \varepsilon_1(\CC_p \otimes_\ZZ E(F)) \\
&\stackrel{P \mapsto (Q \mapsto \langle P,Q\rangle)}{\simeq}& \varepsilon_1(\CC_p \otimes_\ZZ E(F))^\ast \\
&\stackrel{(\ref{isom2})}{\simeq}& \varepsilon_1\left(\bigoplus_{\fp \in S_p(F)}\CC_p\otimes_{\ZZ_p}E(F_\fp)^\wedge \right)^\ast \\
&\simeq& \CC_p[G]\varepsilon_1,
\end{eqnarray*}
where the last isomorphism is given by
$$a \mapsto  {\rm Eul}_S \cdot \Omega^+ \sum_{\sigma \in G}\iota_0(\exp_\omega^\ast (\sigma a))\sigma^{-1}.$$
Since the second isomorphism sends $\varepsilon_1 \eta_F$ to $\langle \varepsilon_1 \eta_F, \varepsilon_1 \eta_F\rangle \cdot (\varepsilon_1 \eta_{F})^\ast$, we see that $\eta_{F}$ has the required property (b).

This completes the proof of Proposition \ref{kato prop}.\qed

\vspace{3mm}

The following (deep) result is essentially due to Kato \cite{kato}. 

\begin{proposition}\label{kato bk} For each $\chi \in \widehat G_0$ (i.e., $\chi$ in $\widehat G$ with $L(E,\chi,1)\not= 0$) one has $e_\chi\cdot  z_F^{\rm Kato}= e_\chi\cdot\eta_{F}^{\rm BSD}$.
\end{proposition}

\begin{proof} 
We regard the element $e^+\delta(\xi)$ of $T(f)^+$ as an element of
$$Y_F(T(f)):=\bigoplus_{w \in S_\infty(F)}H^0(F_w,T(f))=\bigoplus_{\iota: F \hookrightarrow \RR} T(f)^+$$
by placing it in that component of the product which corresponds to the embedding $F \to \RR$ obtained by restricting $\sigma_0$.

In this way the equality in (\ref{norm equation}) implies that $e^+\delta(\xi) = u\cdot \underline{b}$ where $\underline{b}$ is regarded as an element of $Y_F(T)$ as in the definition of $\eta_F := \eta_F^{\rm BSD}(:=\eta_{\underline{b},F}^{\rm BK})$.

%

To proceed, we use the homomorphism
$$\exp^\ast: H^1(\cO_{F,S},V) \to \bigoplus_{\fp \in S_p(F)} H^1_{/f}(F_\fp,V) \to  \Gamma(E,\Omega_{E/F}^1) \otimes_\QQ \QQ_p$$
that is induced by the dual exponential map.

In particular, we recall from Kato \cite[Th. 6.6 and 9.7]{kato} that the image of ${}_{c,d}z_F^{\rm Kato}$ under $\exp^\ast$ belongs to $\Gamma(E,\Omega_{E/F}^1)$ and that for each $\chi$ in $\widehat G_0$ one has
\begin{equation}\label{key equal}\sum_{\sigma \in G} \chi^{-1}(\sigma){\rm per}(\sigma \exp^\ast({}_{c,d}z_F^{\rm Kato}))=L_S(E,\chi^{-1},1)\cdot \chi(t_{c,d})e^+ \delta(\xi),\end{equation}
where
\[{\rm per}: \Gamma(E,\Omega_{E/F}^1)=\Gamma(E,\Omega_{E/\QQ}^1)\otimes_\QQ F \to  H_1(E(\CC),\RR)^{+,\ast} = H_1(E(\CC),\RR)^{+}\]
is the (dual) period map that is induced by sending each $\omega\otimes a$ to the map $\gamma \mapsto \iota_0(a) \int_\gamma \omega$ (and then identifying $H_1(E(\CC),\RR)^{+}$ with its linear dual in the canonical way).

On the other hand, the definition of $\eta_F$ combines with the explicit description of $\varepsilon_0\cdot \lambda$ given in the above proof of Proposition \ref{kato prop} to imply that for each $\chi$ in $\widehat G_0$ one has
$$\sum_{\sigma \in G} \chi^{-1}(\sigma){\rm per}(\sigma \exp^\ast(\eta_{F}))=L_S(E,\chi^{-1},1)\cdot \underline{b}.$$

Upon comparing this equality with (\ref{key equal}), and recalling that $e^+\delta(\xi) = u\cdot \underline{b}$, one finds that
\[ e_\chi \cdot {}_{c,d}z_F^{\rm Kato} = e_\chi\cdot \chi(t_{c,d})u\cdot\eta_F = e_\chi \cdot  t_{c,d}u\cdot \eta_F\]
and hence that $e_\chi\cdot  z_F^{\rm Kato}= e_\chi\cdot\eta_{F}$, as required.
\end{proof}

\subsection{The equivariant Birch and Swinnerton-Dyer Conjecture}\label{eBSD} We assume throughout that $K = \QQ$ and $p>3$. We also continue assuming that $\sha(E/F)[p^\infty]$ is finite.

In this subsection, we apply Theorems \ref{main} and \ref{hes2} to derive evidence for natural equivariant refinements of the Birch and Swinnerton-Dyer Conjecture for $E$ over $F$.

For brevity, we shall say that $E$ `validates the $p$-part of the Birch and Swinnerton-Dyer conjecture over $\QQ$', or more simply that `${\rm BSD}_p(E/\QQ)$ is valid', if the following three conditions are satisfied:
\begin{itemize}
\item[-] $\sha(E/\QQ)$ is finite;
\item[-] the order of vanishing of $L(E,s)$ at $s=1$ is equal to ${\rm rank}(E(\QQ))$;
\item[-] the formula for the leading term at $s=1$ of $L(E,s)$ predicted by the Birch and Swinnerton-Dyer Conjecture is valid up to multiplication by an element of $\ZZ_p^\times$.
\end{itemize}
Note that the last condition is equivalent to ${\rm TNC}(h^1(E)(1),\ZZ_p)$.

The following observation will be used:
if $E$ has good reduction at $p$, then we can take $c $ and $d$ so that
\begin{equation}\label{t condition} t_{c,d}\in \ZZ_p[G]^\times.\end{equation}
In fact, one can take
$$c=d=1+6Ne\prod_{\ell \in S,\ \ell \neq p}\ell,$$
where $\ell$ runs over prime numbers and $e$ is any choice of integer with
\[ 6Ne \prod_{\ell \in S, \ \ell \neq p}\ell \not\equiv\,\, 0,-1 \,\, (\text{mod}\, p).\]
%
%
%
%
Note, however, that if $E$ has bad reduction at $p$, then the element $t_{c,d}^{-1}$ of $\QQ_p[G]^\times$ need not belong to $\ZZ_p[G]$ and so $z_F^{\rm Kato}$ need not belong to $H^1(\cO_{F,S},T)$.

\subsubsection{Statement of the main results}

\begin{theorem}\label{cor kato} Assume that the following two conditions are satisfied:
\begin{itemize}
\item[(a)] The image of the Galois representation $\rho: G_\QQ \to {\rm Aut}(T) \simeq {\rm GL}_2(\ZZ_p)$
contains ${\rm SL}_2(\ZZ_p)$;
\item[(b)] The group $E(\QQ_\ell)[p]$ vanishes for every prime number $\ell$ in $S$.
\end{itemize}
Then the following claims are valid.
\begin{itemize}
\item[(i)] For all non-negative integers $j$ one has
\[ I_j({\cR}_1^{-1}(\cD_{1}({}_{c,d}z^{\rm Kato}))) \subseteq {\rm Fitt}_{\ZZ_p[G]}^j({\rm Sel}_p^{\rm str}(E/F)^\vee).\]
In particular, the ideal 
$$I({}_{c,d}z_F^{\rm Kato}):=\{\psi({}_{c,d}z_F^{\rm Kato}) \mid \psi \in \Hom_{\ZZ_p[G]}(H^1(\cO_{F,S},T),\ZZ_p[G])\}$$
is contained in ${\rm Fitt}_{\ZZ_p[G]}^0({\rm Sel}_p^{\rm str}(E/F)^\vee).$
\item[(ii)] All of the inclusions in claim (i) are equalities whenever each of the following conditions is satisfied:
\begin{itemize}
\item[(c)] $E$ has good reduction at $p$;
\item[(d)] ${\rm BSD}_p(E/\QQ)$ is valid;
\item[(e)] Either $E(\QQ)$ has rank zero or $E$ validates the conjectures \cite[Conj. 3.3.5(i) and Conj. 3.3.2]{PR} of Perrin-Riou.
\end{itemize}

\item[(iii)] If the conditions (d) and (e) in claim (ii) are satisfied, then for all $\chi$ in $\widehat G$ one has $e_\chi z_F^{\rm Kato}\not=0$ if and only if ${\rm rk}_\chi(E(F)) \le 1$.
\end{itemize}
\end{theorem}

\begin{remark}\label{cor kato rem}\

\noindent{}(i) There are by now many circumstances in which ${\rm BSD}_p(E/\QQ)$ is known to be valid. For example, if $L(E,1)\not=0$, $E$ is semistable and has good ordinary reduction at $p$ and $p$ is at least $11$, then the validity of ${\rm BSD}_p(E/\QQ)$ is proved by Skinner and Urban in \cite{SU}. In addition, if $L(E,s)$ vanishes to order one at $s=1$, the conductor of $E$ is square-free, $p$ is at least $5$ and such that the Galois representation $T_p(E)/p$ is irreducible, then the validity of ${\rm BSD}_p(E/\QQ)$ was recently is proved by Jetchev, Skinner and Wan in \cite{JSW}.

\noindent{}(ii) It is believed that the conjectures \cite[Conj. 3.3.5(i) and Conj. 3.3.2]{PR} of Perrin-Riou have been proved whenever $E$ has good ordinary reduction at $p$ - see the discussion in Remark \ref{perrin riou}(ii).

\noindent{}(iii) Proposition \ref{kato prop}(c) implies that Theorem \ref{cor kato}(iii) provides evidence in favour of Conjecture \ref{kato conj}.
\end{remark}

\begin{remark}\label{kurihara rem} Claims (i) and (ii) of Theorem \ref{cor kato} (and parts (i) and (ii) of Remark \ref{cor kato rem}) combine to imply that, in a wide variety of situations, Kato's zeta elements determine the ideal ${\rm Fitt}_{\ZZ_p[G]}^j({\rm Sel}_p^{\rm str}(E/F)^\vee)$ for every $j \ge 0$ and hence strongly control the structure of ${\rm Sel}_p^{\rm str}(E/F)^\vee$ as a $\Z_p[G]$-module (for example, if $G$ is trivial, then the higher Fitting ideals together determine this structure up to isomorphism). This observation complements the main results of Kurihara in \cite{mk, mk2} in which Euler and Kolyvagin systems of Gauss-sum type are used to obtain information about the Galois structure of
 Selmer groups. However, some of the hypotheses that are used in loc. cit. are much stronger than those used here and include, for example, the assumed validity of main conjectures of Iwasawa theory, vanishing of $\mu$-invariants and non-degeneracy of $p$-adic height pairings. \end{remark}


In the next result we use Kato's zeta element to obtain new evidence for equivariant refinements of the Birch and Swinnerton-Dyer conjecture.

In this result we use the $\ZZ_p$-order in $\QQ_p[G]\varepsilon$ that is obtained by setting
\begin{equation}\label{ord def} \cR:=\{x \in \QQ_p[G]\varepsilon \mid x\cdot  {\rm Fitt}_{\ZZ_p[G]}^0( {\rm Sel}_{p}^{\rm str}(E/F)^\vee) \subseteq  {\rm Fitt}_{\ZZ_p[G]}^0( {\rm Sel}_{p}^{\rm str}(E/F)^\vee)\}.\end{equation}

\begin{theorem}\label{theorem ell 2} The conjecture ${\rm TNC}(h^1(E_{/F})(1),\cR)$ is valid whenever all of the following conditions are satisfied:

\begin{itemize}
\item[(i)] The conditions (a), (b), (c), (d) and (e) in Theorem \ref{cor kato} are satisfied;
\item[(ii)] For any $\chi$ in $\widehat G$ that has both $L(E,\chi,1) =0$ and ${\rm rk}_\chi(E(F))\le 1$ one has
\[ \Omega^+ \sum_{\sigma \in G} \iota_0( \exp_\omega^\ast(\sigma z_F^{\rm Kato})) \chi(\sigma)=L_S^\ast(E,\chi,1)\]
if ${\rm rk}_\chi(E(F))=0$, and
$$\Omega^+ \langle \varepsilon_1 z_F^{\rm Kato}, \varepsilon_1 z_F^{\rm Kato} \rangle \sum_{\sigma \in G} \iota_0(\exp_\omega^\ast(\sigma (\varepsilon_1 z_F^{\rm Kato})^\ast)) \chi(\sigma)=L_S^\ast(E,\chi,1)$$
if ${\rm rk}_\chi(E(F))=1$.
\end{itemize}
\end{theorem}

\begin{remark}\label{rem ell 2}\

\noindent{}(i) The displayed equalities in Theorem \ref{theorem ell 2}(ii) are predicted by a special case of Conjecture \ref{kato conj} (see Proposition \ref{kato prop}).

\noindent{}(ii) To give a simple application of Theorem \ref{theorem ell 2}, we assume to be given a finite set of prime numbers $\Sigma$ that contains $p$ and an elliptic curve $E$ over $\QQ$ that has good reduction outside $\Sigma\setminus \{p\}$ and is also such that both the image in ${\rm GL}_2(\ZZ_p)$ of the associated Galois representation on $T$ contains ${\rm SL}_2(\ZZ_p)$ and also $E(\Q_\ell)[p]$ vanishes for all $\ell$ in $\Sigma$. Now let $F/\QQ$ be any abelian extension of $p$-power degree that is unramified outside $\Sigma$ and such that ${\rm Sel}_p(E/F)$ vanishes and set $G := \Gal(F/\QQ)$. Then, in this case the order $\cR$ defined by (\ref{ord def}) is equal to $\ZZ_p[G]$ and, as an easy special case, Theorem \ref{theorem ell 2} implies that ${\rm TNC}(h^1(E_{/F})(1),\ZZ_p[G])$ is valid whenever ${\rm BSD}_p(E/\QQ)$ is valid. This observation strongly refines the main results of Bley in \cite{Bley} which have hitherto provided the best available evidence in support of the equivariant Tamagawa number conjecture for elliptic curves over abelian extensions of $\QQ$.
\end{remark}

To end this section we show that the order $\cR$ defined by (\ref{ord def}) can be explicitly computed in a much more general setting than is considered in Remark \ref{rem ell 2}(ii).

We define the `strict $p$-primary Tate-Shafarevich group' $\sha^{\rm str}_p(E/F)$ of $E$ over a number field $F$ to be the quotient of ${\rm Sel}^{\rm str}_{p}(E/F)$ by its maximal divisible subgroup.

\begin{proposition}\label{yak prop} If $\sha(E/F)$ is finite, then the following claims are valid.
\begin{itemize}
\item[(i)] If  $\sha^{\rm str}_p(E/F)$ vanishes and ${\rm rk}_\phi(E(F)) \le 1$ for all $\phi$ in $\widehat{G}$, then $\cR=\ZZ_p[G]$.
\item[(ii)] Assume that $G$ is cyclic and write $J$ for the maximal subgroup of $G$ for which the quotient ${\rm Sel}^{\rm str}_{p}(E/F)/{\rm Sel}^{\rm str}_{p}(E/F^J)$ is finite. Then, if $\sha^{\rm str}_p(E/K)$ vanishes for all $K$ with $\QQ\subseteq K \subseteq F$, the order $\cR$ is equal to $\ZZ_p[G](1-e_J)$.
\end{itemize}
\end{proposition}

\begin{proof} For each intermediate field $K$ of $F/\QQ$ we set $X_{K} := {\rm Sel}^{\rm str}_{p}(E/K)^\vee$.

To prove claim (i) we note first that the given conditions imply $\varepsilon = 1$ (see the end of the proof of Proposition \ref{kato prop}). To prove claim (i) it therefore suffices to show that the given conditions also imply  that $X_F$ vanishes, and hence that ${\rm Fitt}_{\ZZ_p[G]}^0(X_F) = \ZZ_p[G]$. But this is true since $X_{F,{\rm tor}} = \sha^{\rm str}_p(E/F)^\vee$ whilst Lemma \ref{elliptic coh}(iii) implies that if ${\rm rk}_\phi(E(F))\le 1$ for all $\phi$ in $\widehat {G}$, then $X_F$ is finite.

To prove claim (ii) we note first that, for each non-trivial subgroup $J$ of $G$, there are isomorphisms of $J$-coinvariants
\begin{equation}\label{yak arg}   X_{F,J} \simeq H^1_{\cF_{\rm can}^\ast}(F,T^\vee(1))^\vee _J \simeq H^1_{\cF_{\rm can}^\ast}(F^J,T^\vee(1))^\vee \simeq X_{F^J}.\end{equation}
Here the first and third isomorphisms follow from Example \ref{selmer exams}(ii) and the second from \cite[Cor. 3.8]{bss}. In particular, since $X_{F^J,{\rm tor}} = \sha^{\rm str}_p(E/F^J)^\vee$ is assumed to vanish, the module $X_{F,J}$ is torsion-free.

Now the Tate cohomology group $\hat H^{-1}(J,X_{F,{\rm tf}})$ is equal to $(X_{F,{\rm tf}})_{J,{\rm tor}}$. Thus, by taking $J$-coinvariants of the tautological exact sequence $0 \to X_{F,{\rm tor}} \to X_F \to X_{F,{\rm tf}}\to 0$, one finds that $\hat H^{-1}(J,X_{F,{\rm tf}})$ is isomorphic to a quotient of $(X_{F,J})_{\rm tor}$ and so vanishes.

Next we note that, as $G$ is a cyclic $p$-group and $\hat H^{-1}(J,X_{F,{\rm tf}})$ vanishes for all subgroups $J$ of $G$, the main result (Theorem 2.4 and Lemma  5.2) of Yakovlev \cite{yakovlev} implies that $X_F = X_{F,{\rm tf}}$ is isomorphic as a $\ZZ_p[G]$-module to a direct sum of the form $\bigoplus_{J}\ZZ_p[G/J]^{n(J)}$, where $J$ runs over all subgroups of $G$ and each $n(J)$ is a non-negative integer.

Taken in conjunction with the final isomorphism in Lemma \ref{elliptic coh}(i) this in turn implies that $\varepsilon$ is equal to  $1-e_{J_0}$ where $J_0$ is the smallest subgroup of $G$ with $n(J_0) \not= 0$.

To deduce claim (ii) it thus suffices to note that $J_0$ is the largest subgroup of $\cG_F$ for which the quotient $X_F/X_{F^{J_0}} = {\rm Sel}^{\rm str}_{p}(E/F)/{\rm Sel}^{\rm str}_{p}(E/F^{J_0})$ is finite.
\end{proof}


\subsubsection{Verifying the hypotheses} As preparation for the  proofs of Theorems \ref{cor kato} and \ref{theorem ell 2}, we consider the hypotheses that are necessary to apply Theorem \ref{main}.

In the following result we do not assume $K=\QQ$.

\begin{lemma}\label{lemma hyp}\
\begin{itemize}
\item[(i)] The hypotheses (H$_0$) and (H$_4$) are satisfied in all cases.
\item[(ii)] If the image of the Galois representation
$$\rho: G_K \to {\rm Aut}(T) \simeq {\rm GL}_2(\ZZ_p)$$
contains ${\rm SL}_2(\ZZ_p)$, then the hypotheses (H$_1$), (H$_2$) and (H$_3$) are satisfied.
\item[(iii)] If for every $\fq$ in $S \setminus S_\infty(K)$ the group $E(K_\fq)$ has no element of order $p$, then the hypothesis (H$_5$) is satisfied.
\end{itemize}
\end{lemma}

\begin{proof} Note that (H$_4$) is vacuous since we assume $p>3$. Claim (i) therefore follows from the classical fact that, for any prime $\fq\notin S$, every complex root of the polynomial
$$P_\fq(T;x)=\det(1-{\rm Fr}_\fq^{-1}x \mid T)=1-a_\fq {\N}\fq^{-1}x +{\N}\fq^{-1}x^2$$
has absolute value $\sqrt{{\N}\fq}$.

To prove claim (ii) we assume $\im(\rho)$ contains ${\rm SL}_2(\ZZ_p)$. (Note that this condition is independent of the choice of the isomorphism ${\rm Aut}(T)\simeq {\rm GL}_2(\ZZ_p)$.)

In this case (H$_1$) follows easily. To verify (H$_2$) we note that the given assumption combines with the Weil pairing to imply that $\rho(G_{K(\mu_{p^\infty})})={\rm SL}_2(\ZZ_p)$. In addition, since $p>3$ the group ${\rm SL}_2(\ZZ_p)$ is perfect (i.e., it has no nontrivial abelian quotients) and so, as $F_{p^\infty}/K(\mu_{p^\infty})$ is abelian, one has $\rho(G_{F_{p^\infty}})={\rm SL}_2(\ZZ_p)$.

In particular, any element $\tau$ of $G_{F_{p^\infty}}$ such that $\rho(\tau)$ is equal to $\begin{pmatrix}1 & 1 \\0 & 1\end{pmatrix}$ validates hypothesis (H$_2$).

Finally, we shall check (H$_3$). To do this we use the inflation-restriction exact sequence
$$ H^1(F_{p^\infty}/K,H^0(F_{p^\infty},\overline T)) \xrightarrow{\rm Inf} H^1(F(T)_{p^\infty}/K,\overline T) \xrightarrow{\rm Res} H^1(F(T)_{p^\infty}/F_{p^\infty},\overline T).$$
The first term here vanishes since one has $H^0(F_{p^\infty},\overline T)=H^0({\rm SL}_2(\ZZ_p),\overline T)=0$.
Also, since the above argument implies $\Gal(F(T)_{p^\infty}/F_{p^\infty})$ is isomorphic to ${\rm SL}_2(\ZZ_p),$ the last term in the sequence is isomorphic to $H^1({\rm SL}_2(\ZZ_p),\FF_p^2)$ and one checks easily that this group vanishes.

The exactness of the above sequence therefore implies that $H^1(F(T)_{p^\infty}/K,\overline T)$ vanishes, as required to complete the proof of claim (ii).

Finally, we note that claim (iii) is true since the Weil pairing identifies $H^0(K_\fq, \overline T^\vee(1))$ with $H^0(K_\fq,\overline T)=E(K_\fq)[p]$.\end{proof}

\subsubsection{The proof of Theorem \ref{cor kato}}\label{first proof}




We set $\cT:={\rm Ind}_{G_\QQ}^{G_F}(T)$ and recall (from Example \ref{selmer exams}(ii)) that $H^1_{\cF_{\rm can}^\ast}(\QQ,\cT^\vee(1))$ coincides with ${\rm Sel}_{p}^{\rm str}(E/F)$.

Next we note that the given assumptions combine with Lemma \ref{lemma hyp} to imply that the hypotheses (H$_0$), (H$_1$), (H$_2$), (H$_3$), (H$_4$) and (H$_5$) are satisfied in the setting of Theorem \ref{cor kato} and hence that we may apply Theorem \ref{main} in this context.

In particular, Theorem \ref{main}(ii) implies that in this case one has $\chi_{\rm can}(\cT)=1$.

In addition, since hypothesis (H$_1$) asserts that the residual representation $T/p$ is irreducible, Remark \ref{rem euler} defines an Euler system ${}_{c,d}z^{\rm Kato}$ in ${\rm ES}_1(T,\cK)$.

By applying Theorem \ref{main}(iii) to this Euler system we then deduce both that for each non-negative integers $j$ one has
\[ I_j({\cR}_1^{-1}(\cD_{1}({}_{c,d}z^{\rm Kato}))) \subseteq {\rm Fitt}_{\ZZ_p[G]}^j({\rm Sel}_p^{\rm str}(E/F)^\vee)\]
(as required to prove Theorem \ref{cor kato}(i)) and, in addition, that all of these inclusions are equalities if and only if the inclusion for $j=0$ is an equality.

Thus, by Nakayama's lemma and the argument of Lemma \ref{codescent}(ii), Theorem \ref{cor kato}(ii) will follow if we can show that the stated conditions (c), (d) and (e) are enough to imply that
\begin{equation}\label{BK eq}I({}_{c,d}z_\QQ^{\rm Kato}) = {\rm Fitt}_{\ZZ_p}^0({\rm Sel}_p^{\rm str}(E/\QQ)^\vee). \end{equation}

To show this we note first that, if $E$ has good reduction at $p$, then ${}_{c,d}z_F^{\rm Kato}$ and $z_F^{\rm Kato}$ differ by an element of  $\ZZ_p[G]^\times$ and so in the argument, for any intermediate field $F'$ of $F/\QQ$, one can replace ${}_{c,d}z_{F'}^{\rm Kato}$ by $z_{F'}^{\rm Kato}$. 

Next we note that the stated condition (e) combines with Proposition \ref{kato cor}(iii) to imply that $z_\QQ^{\rm Kato} = \eta_{\QQ}^{\rm BSD}$.

The required equality (\ref{BK eq}) is therefore true if and only if it is true with ${}_{c,d}z_\QQ^{\rm Kato}$ replaced by $\eta_{\QQ}^{\rm BSD}$. To complete the proof of claim (ii) it thus suffices to deduce the latter equality from the validity of ${\rm BSD}_p(E/\QQ)$.

To do this we note that, since (H$_5$) is satisfied, one has ${\rm Sel}_p^{\rm str}(E/\QQ)^\vee = H^2(\cO_{\QQ,S},T)$ (as a consequence of Lemma \ref{compare} and Example \ref{selmer exams}(iii)). It is then enough to note that the validity of ${\rm BSD}_p(E/\QQ)$ combines with Remark \ref{ex L1}(ii) and Theorem \ref{prop Xi} to imply that $I(\eta_{\QQ}^{\rm BSD})
= {\rm Fitt}_{\ZZ_p}^0(H^2(\cO_{\QQ,S},T))$.

To prove claim (iii) we note that the given conditions combine with claim (ii) to imply that $I(z_F^{\rm Kato}) = {\rm Fitt}_{\ZZ_p[G]}^0({\rm Sel}_p^{\rm str}(E/F)^\vee)$.

It is then enough to note that for each $\chi$ in $\widehat G$ it is clear that
\[ e_\chi(\QQ_p^c\otimes_{\ZZ_p}I(z_F^{\rm Kato})) = 0 \Longleftrightarrow e_\chi \cdot z_F^{\rm Kato} = 0, \]
%
whilst one also has
\begin{align*}e_\chi(\QQ_p^c\otimes_{\ZZ_p}{\rm Fitt}_{\ZZ_p[G]}^0({\rm Sel}_p^{\rm str}(E/F)^\vee))  = 0 &\Longleftrightarrow {\rm dim}_{\QQ_p^c}(e_\chi(\QQ_p^c\otimes_{\ZZ_p}{\rm Sel}_p^{\rm str}(E/F)^\vee)) > 0\\
 &\Longleftrightarrow {\rm rk}_\chi( E(F)) > 1,\end{align*}
where the first equivalence follows from a general property of zeroth Fitting ideals and the second from Lemma \ref{elliptic coh}(iii) and the (assumed) finiteness of $\sha(E/F)[p^\infty]$.

This completes the proof of Theorem \ref{cor kato}.

\subsubsection{The proof of Theorem \ref{theorem ell 2}} Since we are assuming that the conditions (a), (b), (c), (d) and (e) of Theorem \ref{cor kato} are satisfied, the argument of \S\ref{first proof} gives an equality
\[ I(z_F^{\rm Kato}) = {\rm Fitt}_{\ZZ_p[G]}^0(H^2(\cO_{F,S},T)).\]

The given hypotheses also imply that the argument of Theorem \ref{cor1} applies in this case to show that the validity of ${\rm TNC}(h^1(E_{/F})(1),\cR)$ will follow from the latter equality if one has $I(z_F^{\rm Kato}) = I(\eta_{F}^{\rm BSD})$.

It is thus enough to note that the stated equalities in Theorem \ref{theorem ell 2}(ii) combine with Proposition \ref{kato prop} to imply that
 $z_F^{\rm Kato} = \eta_{F}^{\rm BSD}$.

This completes the proof of Theorem \ref{theorem ell 2}.

\begin{acknowledgements}
The third author would like to thank Masato Kurihara for illuminating discussions, especially about Kato's zeta elements.
\end{acknowledgements}

\end{document}